\documentclass[12pt,a4paper]{amsart}
\usepackage{amssymb, amstext, amscd, amsmath, color}

\usepackage{graphicx}
\usepackage{epstopdf}

\usepackage{epstopdf}

\usepackage{url}

\usepackage{hhline}

\begin{document}

\newtheorem{thm}{Theorem}[section]
\newtheorem{cor}[thm]{Corollary}
\newtheorem{prop}[thm]{Proposition}
\newtheorem{lem}[thm]{Lemma}
%
\theoremstyle{definition}
\newtheorem{rem}[thm]{Remark}
\newtheorem{defn}[thm]{Definition}
\newtheorem{note}[thm]{Note}
\newtheorem{eg}[thm]{Example}
\newcommand{\Prf}{\noindent\textbf{Proof.\ }}
\newcommand{\bx}{\hfill$\blacksquare$\medbreak}
\newcommand{\upbx}{\vspace{-2.5\baselineskip}\newline\hbox{}%
\hfill$\blacksquare$\newline\medbreak}
\newcommand{\eqbx}[1]{\medbreak\hfill\(\displaystyle #1\)\bx}

\newcommand{\bC}{{\mathbb{C}}}
\newcommand{\bD}{{\mathbb{D}}}
\newcommand{\bN}{{\mathbb{N}}}
\newcommand{\bQ}{{\mathbb{Q}}}
\newcommand{\bR}{{\mathbb{R}}}
\newcommand{\bT}{{\mathbb{T}}}
\newcommand{\bX}{{\mathbb{X}}}
\newcommand{\bZ}{{\mathbb{Z}}}
\newcommand{\bH}{{\mathbb{H}}}
\newcommand{\BH}{{\B(\H)}}
\newcommand{\bsl}{\setminus}
\newcommand{\ca}{\mathrm{C}^*}
\newcommand{\cstar}{\mathrm{C}^*}
\newcommand{\cenv}{\mathrm{C}^*_{\text{env}}}
\newcommand{\rip}{\rangle}
\newcommand{\ol}{\overline}
\newcommand{\td}{\widetilde}
\newcommand{\wh}{\widehat}
\newcommand{\sot}{\textsc{sot}}
\newcommand{\wot}{\textsc{wot}}
\newcommand{\wotclos}[1]{\ol{#1}^{\textsc{wot}}}
 \newcommand{\A}{{\mathcal{A}}}
 \newcommand{\B}{{\mathcal{B}}}
 \newcommand{\C}{{\mathcal{C}}}
 \newcommand{\D}{{\mathcal{D}}}
 \newcommand{\E}{{\mathcal{E}}}
 \newcommand{\F}{{\mathcal{F}}}
 \newcommand{\G}{{\mathcal{G}}}
\renewcommand{\H}{{\mathcal{H}}}
 \newcommand{\I}{{\mathcal{I}}}
 \newcommand{\J}{{\mathcal{J}}}
 \newcommand{\K}{{\mathcal{K}}}
\renewcommand{\L}{{\mathcal{L}}}
 \newcommand{\M}{{\mathcal{M}}}
 \newcommand{\N}{{\mathcal{N}}}
\renewcommand{\O}{{\mathcal{O}}}
\renewcommand{\P}{{\mathcal{P}}}
 \newcommand{\Q}{{\mathcal{Q}}}
 \newcommand{\R}{{\mathcal{R}}}
\renewcommand{\S}{{\mathcal{S}}}
 \newcommand{\T}{{\mathcal{T}}}
 \newcommand{\U}{{\mathcal{U}}}
 \newcommand{\V}{{\mathcal{V}}}
 \newcommand{\W}{{\mathcal{W}}}
 \newcommand{\X}{{\mathcal{X}}}
 \newcommand{\Y}{{\mathcal{Y}}}
 \newcommand{\Z}{{\mathcal{Z}}}

\newcommand{\fA}{{\mathfrak{A}}}
\newcommand{\fB}{{\mathfrak{B}}}
\newcommand{\fC}{{\mathfrak{C}}}
\newcommand{\fD}{{\mathfrak{D}}}
\newcommand{\fE}{{\mathfrak{E}}}
\newcommand{\fF}{{\mathfrak{F}}}
\newcommand{\fG}{{\mathfrak{G}}}
\newcommand{\fH}{{\mathfrak{H}}}
\newcommand{\fI}{{\mathfrak{I}}}
\newcommand{\fJ}{{\mathfrak{J}}}
\newcommand{\fK}{{\mathfrak{K}}}
\newcommand{\fL}{{\mathfrak{L}}}
\newcommand{\fM}{{\mathfrak{M}}}
\newcommand{\fN}{{\mathfrak{N}}}
\newcommand{\fO}{{\mathfrak{O}}}
\newcommand{\fP}{{\mathfrak{P}}}
\newcommand{\fQ}{{\mathfrak{Q}}}
\newcommand{\fR}{{\mathfrak{R}}}
\newcommand{\fS}{{\mathfrak{S}}}
\newcommand{\fT}{{\mathfrak{T}}}
\newcommand{\fU}{{\mathfrak{U}}}
\newcommand{\fV}{{\mathfrak{V}}}
\newcommand{\fW}{{\mathfrak{W}}}
\newcommand{\fX}{{\mathfrak{X}}}
\newcommand{\fY}{{\mathfrak{Y}}}
\newcommand{\fZ}{{\mathfrak{Z}}}

\newcommand{\sgn}{\operatorname{sgn}}
\newcommand{\rank}{\operatorname{rank}}

\newcommand{\Isom}{\operatorname{Isom}}

\newcommand{\qIsom}{\operatorname{q-Isom}}

\newcommand{\sIsom}{\operatorname{s-Isom}}

\newcommand{\grid}{\operatorname{grid}}
\newcommand{\ggk}{\operatorname{ggk}}
\newcommand{\lgk}{\operatorname{lgk}}
\newcommand{\hc}{\operatorname{hc}}
\newcommand{\ecl}{\operatorname{ecl}}
\newcommand{\Aut}{\operatorname{Aut}}
\newcommand{\sep}{\operatorname{sep}}
\newcommand{\ul}{\underline}
\newcommand{\ec}{\operatorname{ec}}
\newcommand{\lec}{\operatorname{lec}}
\newcommand{\Aff}{\operatorname{Aff}}
\newcommand{\stab}{\operatorname{stab}}

\newcommand{\hxl}{\operatorname{hxl}}
\newcommand{\sql}{\operatorname{sql}}
\newcommand{\cyc}{\operatorname{cyc}}
\newcommand{\pen}{\operatorname{pen}}

\newcommand{\cat}{\operatorname{cat}}


 

\bigskip

\title[Isotopy classes of 3-periodic net embeddings] 
{Isotopy classes for 3-periodic net embeddings} 

\author[I. A. Baburin, S. C. Power and D. M. Proserpio]{I. A. Baburin, S. C. Power and D. M. Proserpio}

\address{Theoretische Chemie, Technische Universit{\"a}t Dresden, 
D-01062 Dresden, Germany.}
\address{Dept.\ Math.\ Stats.\\ Lancaster University\\
Lancaster LA1 4YF \\U.K. }
\address{Dipartimento di Chimica Strutturale e Stereochimica Inorganica (DCSSI), Universita di
Milano, 
20133 Milano, Italy.}

\email{baburinssu@gmail.com}
\email{s.power@lancaster.ac.uk}
\email{davide.proserpio@unimi.it}


\thanks{This work was supported by the Engineering and Physical Sciences Research Council [grant number EP/P01108X/1]. }

\begin{abstract}
 Entangled embedded periodic nets and crystal frameworks  are defined, along with 
their \emph{dimension type}, \emph{homogeneity type}, \emph{adjacency depth} and \emph{periodic isotopy type}. We obtain periodic isotopy classifications for various families of embedded nets with small quotient graphs. We enumerate the 25 periodic isotopy classes of depth 1 embedded nets with a single vertex quotient graph. Additionally, we classify embeddings of $n$-fold copies of {\bf pcu} with all connected components in a parallel orientation and $n$ vertices in a repeat unit, and determine their maximal symmetry periodic isotopes. We also introduce the methodology of linear graph knots on the flat 3-torus $[0, 1)^3$. These graph knots, with linear edges, are spatial embeddings of the labelled quotient graphs of an embedded net which are associated with its periodicity bases.
\end{abstract}

\date{}

\maketitle

\section{Introduction}\label{s:intro}

Entangled and interpenetrating coordination polymers
have been investigated intensively by chemists in recent decades. Their classification and analysis in terms of symmetry,  geometry and topological connectivity 
is an ongoing research direction \cite{ale-bla-koc-pro},\cite{ale-bla-pro-1,ale-bla-pro-2},
\cite{bla-car-cia-pro},\cite{bon-oke}, \cite{car-cia-pro}, \cite{car-et-al}, \cite{koc-et-al}.
 These investigations also draw on mathematical methodologies concerned with periodic graphs, group actions and classification \cite{bab},\cite{del-2},\cite{schulte}. 
On the other hand it seems that there have been few investigations to date on the dynamical aspects of entangled periodic structures with regard to  deformations avoiding edge collisions, or with regard to excitation modes and flexibility in the presence of additional constraints. In what follows we take some first steps in this direction and along the way obtain some systematic classifications of basic families.

A \emph{proper linear 3-periodic net} $\N= (N, S)$ is a periodic bond-node structure in three dimensions with a set $N$ of distinct nodes and a set $S$ of noncolliding  line segment bonds.  The  underlying structure graph $G=G(\N)$ is also known as the \emph{topology} of $\N$ (cf. \cite{del-et-al-4}).
Thus, the net $\N$ is an \emph{embedded net} for a topology $G$, it is translationally periodic with respect to each basis vector of some vector space basis for the ambient space, the nodes are distinct points, and the bonds of $\N$ are noncolliding straight line segments between nodes.  We also define the companion structure of a crystallographic bar-joint framework $\C$. In this case the bonds are of fixed lengths which must be conserved in any continuous motion. Additionally a \emph{3-periodic graph} $(G,T)$ is a pair in which a countable graph $G$ carries a specific periodic structure $T$.

Formal definitions of the periodic entities $\C, \N$ and $G$ are given in Definitions \ref{d:CF}, \ref{d:linearperiodicnet} and \ref{d:periodicgraph}. In crystallographic terminology it is usual in such definitions to require  connectedness \cite{del-oke}. However, we find it convenient in these definitions to extend the usage to cover disconnected periodic structures.  

Subclasses of linear $d$-periodic nets $\N$ are defined in terms of the diversity of their connected components and we indicate the connections between these class divisions and those used for entangled  coordination polymers. In particular we define the \emph{dimension type}, which gives a list of the periodic ranks  of connected subcomponents, and the \emph{homogeneity type} which  concerns the congruence properties between these components. 

Fundamental to the structure of an embedded periodic net are its \emph{labelled quotient graphs} which are finite edge-labelled graphs determined by periodicity bases. In particular the infinite structure graph $G(\N)$ is determined by any labelled quotient graph, and the (unique) quotient graph $QG(\N)$ is the graph of the labelled quotient graph of a primitive periodicity basis. 
These constructs for $\N$ provide useful discriminating features for embedded nets even if they are insensitive to entanglement and catenation. 

Our main concern is the entangled nature of linear periodic nets in 3-space  which have more than one connected component, however we also consider the \emph{self-entanglement} of connected structures. 
Specifically, we approach the classification of linear periodic nets in terms of a formal notion of \emph{periodic isotopy equivalence}, as given in Definition \ref{d:deformationequivalentnets}.  This asserts that two embedded periodic nets in $\bR^3$ are periodically isotopic if there is a continuous path of noncrossing embedded periodic nets between them which is associated with a continuous path of periodicity bases.  In this way we formalise an appropriate variant of the notion of ambient isotopy which is familiar in the theory of knots and links.

As a tool for understanding periodic isotopy  we define \emph{linear graph knots} on the \emph{flat $3$-torus} 
and their \emph{isotopy equivalence} classes. Such a graph knot is a spatial graph in the $3$-torus which is a geometric realisation (embedding) of the labelled quotient graph of a linear periodic net arising from a choice of right-handed periodicity basis for $\N$. 
We prove a natural finiteness theorem (Theorem \ref{t:finiteness}) showing that there are
finitely many periodic isotopy  types of linear graph knots with a given labelled quotient graph. This in turn implies that there are finitely many periodic isotopy types of linear 3-periodic nets with a given labelled quotient graph.

Our discussions and results  are structured as follows.
Sections \ref{s:terminology} to
 \ref{s:isotopyForNets} 
cover terminology, illustrative examples and {general} underlying theory. In Section \ref{s:groupmethods} we give group theory methods, 
while in  Sections \ref{s:entanglednets}, \ref{s:latticenets} and \ref{s:furtherdirections} 
we give a range of results, determining periodic isotopy classes and topologies for various families of embedded nets. 
 
More specifically, in Section \ref{s:terminology}   
we give 
comprehensive terminology, \emph{ab initio}, and give the connections with terms used for coordination polymers and with the net notations of both the Reticular Chemistry Structural Resourse (RCSR) \cite{rcsr} and ToposPro \cite{bla-she-pro}. In the key Section \ref{s:quotient graphs} we discuss labelled and unlabelled quotient graphs. The example considered in detail in Section \ref{s:mainexample} illustrates terminology and {motivates} the introduction of \emph{model nets} for the {analysis}  of periodic isotopy types (periodic isotopes). In Section \ref{s:adjacencydepth} we define primitive periodicity bases and introduce a measure of \emph{adjacency depth} for an embedded net.  In Section \ref{s:graphknots}, as preparation for the discussion of periodic isotopy for embedded nets,  we define linear graph knots on the flat 3-torus $\bT^3= [0,1)^3$ as spatial graphs with (generalised) line segment edges. In Section \ref{s:isotopyForNets} we discuss various isotopy equivalences for graph knots. Also we define periodic isotopy equivalence for embedded nets and prove that it is an equivalence relation and that there are finitely many periodic isotopes with a common labelled quotient graph. In the group methods of Section \ref{s:groupmethods} we give the group-supergroup construction of entangled nets \cite{bab}, the definition of maximal symmetry periodic isotopes, and the role of Burnside's lemma in counting periodic isotopes.
 In Section \ref{s:entanglednets} we determine periodic isotopy classes and also restricted periodic isotopy classes for various  multicomponent shift homogeneous embeddings of $n$-fold {\bf pcu}.
Such multicomponent embedded nets are related to the interpenetrated structures with translationally equivalent components which are abundant in coordination polymers. For generic embeddings we give proofs, based  Burnside's lemma for counting orbits of spatially equivalent embeddings, while for maximal symmetry embeddings for $n${\bf -pcu} we use computations based on group-supergroup methods. 
In Section \ref{s:latticenets} 
we give a detailed determination of the  19 {topologies and} periodic isotopy classes of connected linear 3-periodic nets with a single vertex quotient graph and adjacency depth 1 (Table 3).
In the final section we indicate further research directions.



\section{Terminology}\label{s:terminology}
In any investigation with cross-disciplinary intentions, in our case between chemistry  (reticular chemistry and coordination polymers) and mathematics (isotopy types and periodic  frameworks), it is important to be clear of the meaning of terms. Accordingly we begin by defining all terminology from scratch. 

The \emph{structure graph} $G=(V,E)$ of a finite or countably infinite {bar-joint framework} $\G$ is given a priori since, formally, a \emph{bar-joint framework $\G$ in $\bR^d$} is a pair $(G,p)$ consisting of a simple graph $G$, the structure graph, together with a \emph{placement map} $p:V \to \bR^d, p: v \to p(v)$. The \emph{joints} of $\G$ are the points $p(v)$ and the \emph{bars} of $\G$ are  the (unordered) joint pairs $p(v)p(w)$ associated with the edges $vw$ in $E$. It is often assumed 
that $p(v)\neq p(w)$ for the edges $vw$ and hence the bars may also considered to be the associated nondegenerate line segments $[p(v),p(w)]$.

A \emph{$d$-periodic bar-joint framework} in $\bR^d$ is a bar-joint framework $\G = (G,p)$ in $\bR^d$ 
whose periodicity is determined by two sets $F_v, F_e$ of noncoincident joints $p(v)$ and bars $p(u)p(w)$, respectively, together with a set  of basis vectors for translational periodicity, say $\ul{a} = \{a_1,\dots ,a_d\}$. The requirement is that the associated translates of the set $F_v$ and the set $F_e$ are, respectively, disjoint subsets of the set of joints and the set of bars whose unions are the sets of all joints and bars. In particular $p$ is an injective map. 

The pair of sets 
$(F_v, F_e)$ is a building block, or repeating unit, for $\G$. We  refer to this pair of sets also as a \emph{motif} for $\G$ for the basis $\ul{a}$ and note that $\G$ is determined uniquely by any pair of periodic basis and motif. In fact we shall only be concerned with finite motifs. 

\begin{defn}\label{d:CF}
A crystallographic bar-joint framework $\C$ in $\bR^d$, or \emph{crystal framework},  is a $d$-periodic bar-joint framework in $\bR^d$ with finitely many translation classes for joints and bars.
\end{defn}
 
Viewing $\C$ as a bar-joint framework, rather than as a geometric $d$-periodic net, is a conceptual prelude to the consideration of dynamical issues of flexibility and rigidity  \cite{pow-poly}, one in which we may bring to bear geometric and combinatorial rigidity theory. Note however that we have not required a crystal framework to be connected.

In the case of a 3D crystal framework $\C$, particularly an entangled  one of material origin, it is natural
to require that the line segments
$[p(v), p(w)]$, for $vw \in E$,
are \emph{essentially disjoint} in the sense that they intersect at most at a common endpoint $p(x)$ for some $x\in V$. 
We generally adopt this noncrossing assumption and say that $\C$ is a \emph{proper} crystal framework in this case.
Thus a proper crystal framework $\C$ determines a closed set, denoted $|\C|$, formed by the union of the (nondegenerate) line segments $[p(v), p(w)]$, for $vw \in E$. We call this closed set the \emph{body}  of $\C$. By our assumptions one may recover $\C$  from its body and the positions of the joints. 

The connected components of a crystal framework may have a lower rank (or dimension) of periodicity. Accordingly we make the following definition.

\begin{defn}\label{d:SCF}A \emph{subperiodic crystal framework}, or \emph{$d'$-periodic crystal framework}, with rank (or periodicity dimension)  $1\leq d' <d$, is a $d'$-periodic bar-joint framework in $\bR^d$, with $d'$ linearly independent period vectors and finitely many translation classes for joints and bars.
\end{defn}

For completeness we define a $0$-periodic bar-joint framework in $\bR^d$ to be a finite bar-joint framework in $\bR^d$. Thus every connected component of a crystal framework in $\bR^d$ is either itself a crystal framework in $\bR^d$ or is a subperiodic crystal framework with rank 
$0 \leq d' \leq d$. Note that a subperiodic subframework exists for $\C$ if and only if $\C$ has infinitely many connected components, that is, if and only if the body of $\C$ has infinitely many topologically connected components.

In view of the finiteness requirement for the $d'$-periodic translation classes, a subperiodic crystal framework in $\bR^d$ has a joint set consisting of finitely many translates of a sublattice of rank $d'$.
In general $d'$-periodic subperiodic frameworks may differ in the nature of their \emph{affine span}, or \emph{spatial dimension}, which may take any integral value between $d'$ and $d$.
Formally, the spatial dimension is the dimension of the linear span of all the so-called \emph{bar vectors}, $p(w)-p(v)$ associated with the bars $p(v)p(w)$ of the framework.
Once again we define a subperiodic framework to be \emph{proper} if there are no intersections of edges.

The various definitions above,
and also the following definition of \emph{dimension type}, transpose immediately to the simpler category of \emph{linear periodic nets} $\N$, as defined in the next section. 


We now introduce the general terminology which is specific to  3-dimensional space. Also we indicate how later this formulation of dimension type aligns with the terminology used by chemists for entangled periodic nets.

\begin{defn}\label{d:dimtype}
A periodic or subperiodic framework $\C$  in 3-dimensional space has \emph{dimension type} $\ul{d}= \{d';d_1, \dots ,d_s\}$,  where $d'$ is the periodicity rank of $\C$ and where $d_1, \dots ,d_s$ is the decreasing list of periodicity ranks of the connected components.
\end{defn}

In particular there are 15 dimension types $\{d\}$ for rank 3 crystallographic frameworks, or for linear 3-periodic nets, namely

\[
\{3;3\}, \{3;3,2\}, \{3;3,1\}, \{3;3,0\}, \{3;3,2,1\},\{3;3,2,0\}, \{3;3,1,0\},\{3;3,2,1,0\}
\]
\[
\{3;2\},\{3;2,1\},\{3;2,0\},\{3;2,1,0\},
\{3;1\},\{3;1,0\},\{3;0\}
\]

\bigskip


\subsection{Categories of periodic structures}
Consider the following frequently used terms for periodic structures, arranged with an increasing mathematical flavour: \emph{Crystal, crystal framework, linear periodic net, periodic graph}, and \emph{topological crystal}. 

We have  defined proper crystal frameworks in $\bR^d$ with essentially disjoint bars and these may be viewed as forming an ``upper category" of periodic objects for which there is interest in bar-length preserving dynamics. If we disregard bar lengths, but not geometry, then we are in the companion category  of \emph{positions}, or \emph{line drawings}, or \emph{embeddings}  of $d$-periodic nets in $\bR^d$. Such embeddings are of interest in reticular chemistry and in this connection we may define
\emph{a linear $d$-periodic net in $\bR^d$} to be a pair $(N, S)$, consisting of a set $N$ of nodes and a set $S$ of line segments, where these sets correspond to the joints and bars of a \emph{proper} $d$-periodic crystal framework. 
A stand-alone definition is the following

\begin{defn}\label{d:linearperiodicnet}
A (proper) linear $d$-periodic net in $\bR^d$ is a pair $\N= (N, S)$ where
\medskip

(i) $S$, the set of edges (or bonds) of $\N$, is a countable set of essentially disjoint line segments $[p,q]$, with $p\neq q$, 

(ii) $N$, the set of vertices (or nodes) of $\N$, is the set of endpoints of members of $S$,

(iii) there is a basis of vectors for $\bR^d$ such that the sets
$N$ and $S$ are invariant under the translation group $\T$ for this basis,

(iv) the sets $N$ and $S$ partition into finitely many $\T$-orbits.
\end{defn}

Thus a linear periodic net can be thought of as a proper linear embedding of the structure graph of a crystal framework, the relevant crystal frameworks being those with no isolated joints of degree $0$. 
Note that a linear periodic net is not required to be connected.

A linear $d$-periodic net is referred to in reticular chemistry as an \emph{embedding}
of a \emph{``$d$-periodic net"}. This is because the term {$d$-periodic net} has been appropriated  for the underlying structure graph of a linear periodic net. See Delgado-Friedrichs and O'Keeffe \cite{del-oke}, for example. This reference, to a more fundamental category on which to build, so to speak, then allows one to talk of a $d$-periodic net having an embedding with, perhaps, certain symmetry attributes.
It follows then, tautologically, that a $d$-periodic net is a graph with certain periodicity properties and we formally specify this in  Definition \ref{d:periodicgraph}.

 The next definition is a slight variant of the definition given by Delgado-Friedrichs \cite{del-2}, in that we also require edge orbits to be finite in number. 
\medskip
 
\begin{defn}\label{d:periodicgraph}
(i) A \emph{periodic graph} is a pair $(G, T)$, where $G$ is a countably infinite simple graph and $T$
is a free abelian subgroup of $Aut(G)$ which acts on $G$ freely and is such that the set of vertex orbits and edge orbits are finite. The group $T$ is called a \emph{translation
group} for $G$ and its rank is called the \emph{dimension} of $(G, T)$. 

(ii) A \emph{$d$-periodic graph} or a \emph{$d$-periodic net} is a periodic graph of dimension $d$. 

(iii) The translation group $T$ and the periodic graph $(G, T)$ are  \emph{maximal} if no periodic
structure $(G, T')$ exists with $T'$ a proper supergroup of $T$.
\end{defn}
 
The subgroup $T$ (or the pair $(G, T)$) is referred to as
a \emph{periodic structure on $G$}.  
Some care is necessary with assertions such as  ``$\N$ is an embedding of a 3-periodic net $G$". This has two interpretations according to whether $G$ comes with a \emph{given} periodic structure $T$ which is to be represented faithfully in the embedding as a translation group or, on the other hand, whether the embedding respects \emph{some} periodic structure $T'$ in $\Aut(G)$. 

Finally we remark that there is another category of nets which is relevant to more mathematical considerations of entanglement, namely \emph{string-node nets} in the sense of Power and Schulze \cite{pow-sch}. In the discrete case these have a similar definition to linear periodic nets but the edges may be continuous paths rather than line segments.
 
\subsection{Maximal symmetry, the RCSR and self-entanglement}\label{ss:maxsymmetryRCSR} 
Let $\N$ be a linear $d$-periodic net.  Then there is a natural injective inclusion  map
\[  
\iota_\N: \fS(\N) \to\Aut (G(\N))
\]  
from the usual space group $\fS(\N)$ of $\N$ to the automorphism group of its structure graph $G(\N)$.
 
\begin{defn} Let $\N$ be a linear $d$-periodic net.

(i) The \emph{graphical crystallographic group} of $\N$ 
is the automorphism group $\Aut (G(\N))$. This is also called the \emph{maximal symmetry group} of $\N$.

(ii) A \emph{maximal symmetry embedding} of $G(\N)$ is an embedded net $\M$ for which $G(\M) = G(\N)$ and the map 
$\iota_\M$
is a group isomorphism.
\end{defn}

A key result of Delgado-Friedrichs \cite{del-1},\cite{del-2} shows that many connected $3$-periodic graphs have {unique} maximal symmetry placements, {possibly with edge crossings}.
These placements arise for a so-called \emph{stable net} by means of  a minimum energy placement, associated with a fixed lattice of orbits of a single node, followed by a renormalisation by the point group of the structure graph. Moreover,
stable nets are defined as those where the minimum energy placement has no node collisions.  
See also \cite{del-oke}, \cite{sun-book}. {While maximum symmetry positions for connected stable nets are unique, up to spatial congruence and rescaling,  edge crossings may occur for simply-defined nets because, roughly speaking, the local edge density is too high. It becomes an interesting issue then  to define and determine the finitely many classes of maximum symmetry proper placements  and this is true also for multicomponent nets.}  See Section \ref{ss:group-supergroup}.

The Reticular Chemistry Structural Resource (RCSR)\cite{rcsr} is a convenient online database which, in part, \emph{defines} a set of around 3000 topologies $G$ together with an indication of their maximal symmetry embedded nets. The  graphs $G$ are denoted in bold face notation, such as ${\bf pcu}$ and ${\bf dia}$, in what is now standard nomenclature. We shall make use of this and denote the maximal symmetry embedding of a connected topology  ${\bf abc}$ as $\N_{\rm abc}$.
This determines $\N_{\rm abc}$ as a subset of $\bR^3$ up to a scaling factor and spatial congruence.
 ToposPro \cite{bla-she-pro} is a more sophisticated program package, suitable for multi-purpose crystallochemical analysis and has a more extensive periodic net database. In particular it provides labelled quotient graphs for 3-periodic nets.

Both these databases give coordination density data which can be useful for discriminating the structure graphs of embedded nets.
  
 In Section \ref{s:isotopyForNets} we formalise the \emph{periodic isotopy equivalence} of pairs of embedded nets. One of our motivations is to identify and classify connected embedded nets which are not periodically isotopic to their maximal symmetry embedded net.
We refer to such an embedded net to be a \emph{self-entangled embedded net}.

\subsection{Derived periodic nets} We remark that the geometry and structural properties of a linear periodic net or framework can often be analysed in terms of \emph{derived} nets or frameworks. These associated structures can arise through a number of operations and we now indicate some of these. 

(i) The periodic substitution of a (usually connected) finite unit with a new finite unit (possibly even a single node) while maintaining incidence properties. This move is common in reticular chemistry for the creation of  ``underlying nets" \cite{ale-bla-koc-pro}, \cite{bon-oke}.

(ii) A more sophisticated operation which has been used for the  classification of coordination polymers replaces each  minimal ring of edges by a node (barycentrically placed) and adds an edge between a pair of such nodes if their minimal rings are entangled. In this way  one arrives at the Hopf ring net (HRN) of an embedded net $\N$. This is usually well-defined as a {(possibly improper)}  linear 3-periodic net and it has proven to be an effective discriminator in the taxonomic analysis of crystals and coordination polymer databases.

(iii)  There are various conventions in which notational augmentation is used \cite{ale-bla-pro-1}, \cite{ale-bla-pro-3} to indicate the derivation of an embedded net or its relationship with a parent net. In the RCSR listing for example the notation {\bf pcu-c4} indicates the topology made up of 4 disjoint copies of ${\bf pcu}$ \cite{rcsr}.  
In Tables 1, 2 we use a notation for model embedded nets, such as $\M_{\rm pcu}^{ff}, ...$, which is indicative of a hierarchical construction. 

(iv) On the mathematical side, in the rigidity theory of periodic  bar-joint frameworks $\C$ there are natural periodic graph operations and associated geometric moves, such as periodic edge contractions, which lead to inductive schemes in proofs. 
In particular periodic Henneberg moves, which conserve the average degree count, feature in the rigidity and flexibility theory of such frameworks \cite{nix-ros}.

\subsection{Types of entanglement and homogeneity type}\label{s:homtype} Let us return to descriptive aspects  of disconnected  linear $3$-periodic nets  $\N$ in $\bR^3$.

We first note the following scheme of Carlucci et al \cite{car-et-al} which has been used in the classification of observed entangled coordinated polymers. Such a coordination polymer, $\P$ say, is also a proper  linear $d$-periodic net $\N$ in $\bR^3$, and this is either of full rank $d=3$, or is of subperiodicity rank $1 \leq d<3$, or is a finite net (which we shall say has rank 0). Then $\P$  is said to be 
\medskip

(i) in the \emph{interpenetration class} if all connected components of $\N$ are also $d$-periodic,  

(ii) in the \emph{polycatenation class} otherwise. 
\medskip

Thus $\P$ is in the interpenetration class if and only the dimension type of its net is $\{3;3\}, \{2;2\}$,  $\{1;1\}$ or $\{0;0\}$.
\medskip

The entangled coordination polymers in the interpenetration class may be further divided as subclasses of $n$-fold type, according to the number $n$ of components, where, necessarily, $n$ is finite.
\medskip

The linear 3-periodic nets in the polycatenation class have some components which are subperiodic  and in particular they have countably many components. When \emph{all} the components are 2-periodic, that is, when $\N$ has dimension type $\{3;2\}$, then $\N$ is either of  \emph{parallel type} or \emph{inclined type}. Parallel type is characterised by the common coplanarity of the periodicity vectors of the components, whereas $\N$ is of inclined type if there exist 2 components which are not parallel in this manner. 
The diversity here may be neatly quantified by the number, $\nu_2$ say, of planes through the origin that are determined by the (pairs of) periodicity vectors of the components.

Similarly the disconnected linear 3-periodic nets of dimension type $\{3;1\}$ can be viewed as being of parallel type, with all the 1-periodic components going in the same direction, or, if not, as inclined type.  In fact there is a natural further division of the nonparallel (inclined) types for the nets of dimension type $\{3;1\}$ according to whether the periodicity vectors for the components are co-planar or not. We could describe such nets as being of \emph{coplanar inclined type} and \emph{triple inclined type} respectively.  The diversity here may also be neatly quantified by the number, $\nu_1$ say, of lines through the origin that are determined by the periodicity vectors of the components.

The disconnected nets of parallel type are of particular interest for their mathematical and observed entanglement features, such as borromean entanglement and woven or braided structures \cite{car-et-al}, \cite{liu-et-al}.  
\bigskip

 The foregoing terminology is concerned with the periodic and sub-periodic nature of the components of a net without regard to further comparisons between them. On the other hand the following terms identify subclasses according to the possible congruence between the connected components.
 \medskip

$\N$ is of \emph{homogeneous type} if all pairs of components are pairwise congruent. Here the implementing congruences are not assumed to belong to the space group.
\medskip

$\N$ is  \emph{$n$-heterogenous}, with $n>1$, if there are exactly $n$ congruence classes of connected components. 
\medskip

Thus every 3-periodic linear net $\N$ in $\bR^3$ is either of homogeneous type or $n$-heterogeneous for some $n=2,3,\dots $.

The homogeneous linear 3-periodic nets split into two natural subclasses.
\medskip

$\N$ is of \emph{shift-homogenous type} if all components are pairwise shift equivalent.  Otherwise, when $\N$ contains at least one pair of components which are not shift equivalent then we say that the homogeneous net $\N$ is of \emph{rotation type}.
\medskip
 
Finally we take account of the space group of $\N$ to specify a very strong form of homogeneity:  each of the two homogeneous types contain a further subtype according to whether 
$\N$ is also of transitive type or not, where
\medskip

$\N$ is of \emph{component transitive} (or is  \emph{transitive type}) if the space group of $\N$ acts transitively on components. 
\medskip

Such component transitive periodic nets have been considered in detail by Baburin \cite{bab} with regard to their construction through  group-supergroup methods.
\medskip

Note that a homogeneous linear 3-periodic net in $\bR^3$ which is not connected falls into exactly one of 16 possible \emph{dimension-homogeneity types} according to the 4 possible types of homogeneity and the 4 possible dimension types $\ul{d}$, namely $\ul{d}= \{3;3\}, \{3;2\}, \{3;1\}$ or $\{3;0\}$. For a full list of correspondences see Figure \ref{f:summary}.

 
\begin{center}
\begin{figure}[ht]
\centering
\includegraphics[width=12cm]
{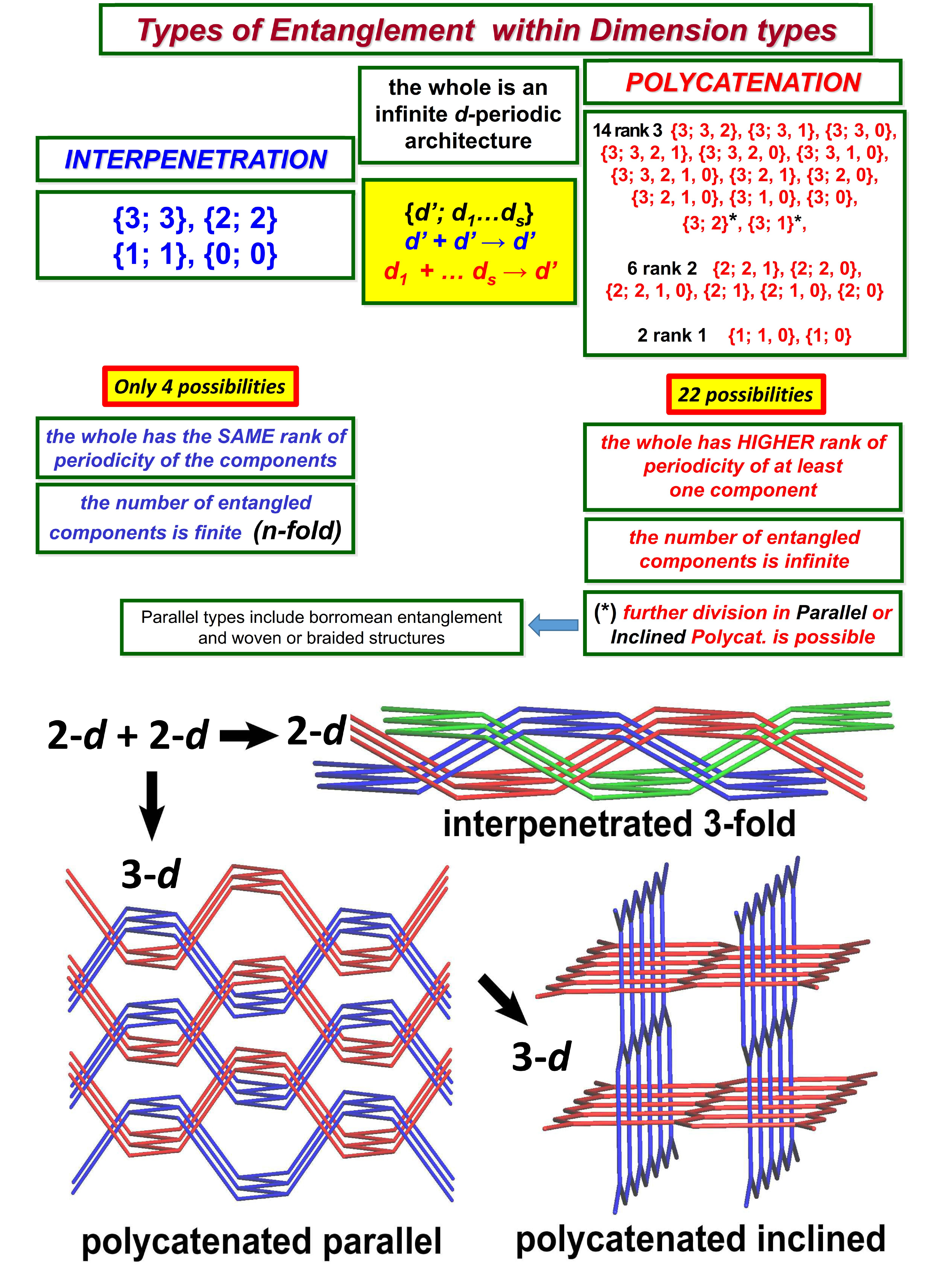}
\caption{Dimension type and polycatenation.}
\label{f:summary}
\end{figure}
\end{center}

\subsection{Catenation and Borromean entanglement}
To the {dimension-homogeneity type}  division of multicomponent embedded nets one may consider further  subclasses which are associated with  entanglement features between the components. Indeed, our main consideration in what follows is a formalisation of such {entanglement} in terms of {linear graph knots}. We note here some natural entanglement invariants of Borromean type.
In fact the embedded nets of dimension type $\{3;2\}$  have been rather thoroughly identified in  \cite{car-et-al}, \cite{ale-bla-pro-3} where it is shown that subdimensional $2$-periodic components can be catenated or woven in diverse ways.
 
To partly quantify this one may define the following entanglement indices. 
Let $\N$ be such a  parallel type embedded net, with dimension type $\{3;2\}$, and let $\S$ be a finite set of components. Then \emph{a separating isotopy} of $\S$ is a continuous deformation of $\S$ to a position which properly lies on both sides of the complement of a plane in $\bR^3$. If there is no separating isotopy for a pair $\S=\{\N_i, \N_j\}$ of components of $\N$ then we say that they are \emph{entangled components}, or are properly entangled. This partial definition can be made rigourous by means of a formal definition of periodic isotopy. 
We may then define the \emph{component entanglement degree} of a  component $\N_i$ of $\N$ to be the maximum number, $\delta(\N_i)$ say, of components $\N_j$ which can form an entangled pair with $\N_i$. Also the component entanglement degree of $\N$ itself may be defined to be the maximum such value.

Likewise one can define the entanglement degrees of components for embedded nets 
of dimension type  $\{3;1\}$ 
and for the 
subdimensional nets of dimension type $\{2;2\}$ (woven layers) and dimension type $\{1;1\}$ (braids).
More formally, we may say that $\N$ has \emph{Borromean entanglement} if there is a set of $n\geq 3$  connected components which admit no separating periodic isotopy while, on the other hand, every pair in this set admits a separating periodic isotopy. 
In a similar way one can formalise the notion of
Brunnian catenation  \cite{lia-mis} for a multicomponent embedded net.

\subsection{When topologies are different}
Two standard graph isomorphism invariants used by crystallographers are the  \emph{point symbol} and the \emph{coordination sequence}.

In a vertex transitive countable graph $G$ the
point symbol (PS), which appears as $4^{24}6^4$ for {\bf bcu} for example, indicates the multiplicities ($24$ and $4$) of the cycle lengths ($4$ and $6$) for a set of minimal cycles which contain a pair of edges incident to a given vertex. If the valency (or \emph{coordination}) of $G$ is $r$ then there are $r(r-1)/2$ such pairs and so  the multiplicity indices sum to  $r(r-1)/2$. For $G$ nontransitive the point symbol is a list of individual point symbols for the vertex classes \cite{ble-oke-pro}.

The coordination sequence (CS) of a vertex transitive countable graph $G$ is usually given partially as a finite list of integers associated with a vertex $v$, say
$n_1, n_2, n_3, n_4, n_5$, where $n_k$ is the number of vertices $w\neq v$ for which there is a edge path from $v$ to $w$ of length $k$ but not of shorter length. For {\bf bcu} this sequence is   8, 26, 56,  98, 152. Cumulative sums of the CS sequence are known as topological densities, and the RCSR, for example, records the 10-fold sum, td10. 

Even the entire coordination sequence is not a complete invariant for the set of underlying graphs of embedded periodic nets. However this counting invariant can be useful for discriminating nets whose local structures are very similar. A case in point is the pair 8T17 and 8T21 appearing in Table 3, which have partial coordination sequences
8, 32, 88 and 8, 32, 80, respectively.

\section{Quotient graphs}\label{s:quotient graphs} We now define quotient graphs and labelled quotient graphs associated with the periodic structure bases of a linear periodic net $\N$. 
Although quotient graphs and labelled quotient graphs are not sensitive to entanglement they nevertheless offer a means of subcategorising linear periodic nets.  See, for example, the discussions in Eon \cite{eon-2011, eon-top}, Klee \cite{kle}, Klein \cite{klein}, Thimm \cite{thi} and Section \ref{ss:isomorphicnets} below. 

 Let $\N= (N, S)$ be a linear 3-periodic net with periodic structure basis $\ul{a}= \{a_1, a_2, a_3\}$. Then $\N$ is completely determined by $\ul{a}$ and any associated building block motif $(F_v, F_e)$. It is natural, especially in illustrating examples, to choose the set $F_e$ of edges of $\N$ to be as connected as possible and to choose $F_v$ to be a subset of the vertices of these edges.
Let $\T$ denote the translation group  associated with  $\ul{a}$, so that $\T$ is the set of transformations
\[
T_k: (x,y,z) \to (x,y,z)+k_1a_1+k_2a_2 + k_3a_3, \quad k\in \bZ^3.
\]
  Each (undirected) line segment edge $p(e)$ in $F_e$ has the form $[T_kp(v_e), T_lp(w_e)]$, where $p(v_e)$ and $p(w_e)$ are the
representatives in $F_v$ for the endpoint nodes $T_kp(v_e), T_lp(w_e)$ of the edge $p(e)$. The labels $k$ and $l$ here may be viewed as the cell labels or translation labels associated with endpoints of $p(e)$.
(As before $v_e, w_e$ indicate vertices of the underlying structure graph $G(\N)$.) 
 
The \emph{labelled quotient graph} LQG$(\N;\ul{a})$ of the pair $(\N, \ul{a})$ is a finite multigraph together with a directed labelling for each edge, where the labelling is by elements $k \in \bZ^3$. The  vertices correspond to (or are labelled by) the vertices $v$ of the nodes $p(v)$ in $F_v$, and the edges correspond to edges $p(e)$ in $F_e$. The directedness is indicated by the ordered pair $(v_e,w_e)$, or by $v_ew_e$, (viewed as directedness ``from $v_e$ to $w_e$"). The label for this directed edge is then $k-l$ where $k, l$ are the translation labels as in the previous paragraph, and so the labelled directed edge is denoted $(v_ew_e,k-l)$. There is no ambiguity since 
the directed labelled edge $(v_ew_e,k-l)$ is considered as the same directed edge as  $(w_ev_e,l-k)$. In particular the following definition of the depth of labelled directed graph is well-defined.

\begin{defn}
Let $\H= (H, \lambda)$ be a quotient graph. Then the \emph{depth} of $\H$ is the maximum modulus of the coordinates of the edge labels.
\end{defn}

The \emph{quotient graph} QG$(\N;\ul{a})$ of the pair $(\N, \ul{a})$ is the undirected graph $G$ obtained from the labelled quotient graph. If $\ul{a}$ is a \emph{primitive periodicity basis}, that is, one associated with a maximal lattice in $\N$, then  QG$(\N;\ul{a})$ is independent of $\ul{a}$ and is the usual \emph{quotient graph} 
of $\N$ in which the vertices are labelled by the translation group orbits of the nodes. Primitive periodicity bases are discussed further in the next section. Moreover we identify there the  ``preferred" primitive periodicity bases which have a ``best fit" for $\N$ in the sense of minimising the maximum size of the associated edge labels.

\begin{defn}
The \emph{quotient graph} QG$(\N)$ of a linear periodic net $\N$ in $\bR^d$ is the unlabelled multigraph graph of the labelled quotient graph determined by a primitive periodicity basis.
\end{defn}

 Finally we remark on the homological terminology related to the edge labellings of a labelled quotient graph.
The homology group $H_1(\bT^3;\bZ)$ of the $3$-torus $\bT^3$ 
is isomorphic to $\bZ^3$. In this isomorphism the standard generators of $\bZ^3$ may be viewed as corresponding to (homology classes of) 
three 1-cycles which wind once around the 3-torus (which we may parametrise naturally by the set $[0,1)^3$) in the positive coordinate directions. Also,
we may associate the standard {ordered basis} for $\bZ^3$ with
a periodicity basis $\ul{a}$ for $\N$. In this case the sum of the labels of a directed cycle of edges in the labelled quotient graph is equal to the homology class of the associated closed path in $\bT^3$.


\subsection{Embedded nets with a common LQG}\label{s:mainexample}
We now consider the family of \emph{all}
linear 3-periodic nets (proper embedded nets) which have a periodic structure basis determining a particular common labelled  quotient graph. This discussion  illuminates some of the terminology set out so far and it also gives a prelude to discussions of periodic isotopy. Also it motivates
the introduction of \emph{model nets} and linear graphs knots on the 3-torus.

Let $H$ be the 6-coordinated graph with two vertices $v_1, v_2$, two connecting edges between them and two loop edges on each vertex. Let $(H,\lambda)$ be the labelled quotient graph with labels 
$(0,0,1), (1,1,1)$ for the loop edges for $v_1$, labels $(0,1,0), (0,0,1)$ for the loop edges for $v_2$, and labels $(0,0,0), (0,-1,-1)$ for the two directed edges from $v_1$ to $v_2$.
Let $\N$ be an embedded net with a generasl periodic structure basis  $\ul{a}$ such that  LQG$(\N;\ul{a})=(H, \lambda)$. (In particular $\N$ has adjacency depth $1$, as defined in the next section.) Note that the 4  loop edges on $v_1$ and $v_2$ imply that $\N$ has two countable sets of two dimensional parallel subnets all of which are pairwise disjoint. These subnets are either parallel to the pair $\{a_2, a_3\}$ or to the pair $\{a_3, a_1+a_2+a_3\}$.
In particular if $\N'\subseteq \N$ is the embedded net which is the union of these 2D subnets then $\N'$ is a derived net of $\N$ of dimension type $\{3;2\}$. Also $\N'$ is in the polycatenation class of inclined type (rather than parallel type). 
By means of a simple oriented affine equivalence (see Definition \ref{d:affineequivalence}) the general pair $(\N,\ul{a})$ with LQG $(H,\lambda)$ is equivalent to a pair $(\M,\ul{b})$, having the same LQG, where  $\ul{b}$ is the standard right-handed orthonormal basis. We shall call the pair $(\M,\ul{b})$ a \emph{model net}.

By translation (another oriented affine transformation) we may assume that there is a node $p_1$ of $\M$ at the origin which is associated with the vertex $v_1$ of $H$.
Let $p_2$ be the unique node associated with $v_2$ which lies in  the unit cell $[0,1)^3$. Now the pair $(\M,\ul{b})$ is uniquely determined by $p_2$ and we denote  it simply as $\M(p_2)$.
Figure \ref{f:mainexample} illustrates the part of the linear periodic net  $\M(p_2)$ which is visible in $[0,1)^3$.
In Section \ref{s:graphknots} we shall formalise diagrams such as Figure \ref{f:mainexample} in terms of linear graph knots on the flat 3-torus.


\medskip

\begin{center}
\begin{figure}[ht]
\centering
\includegraphics[width=4.5cm]{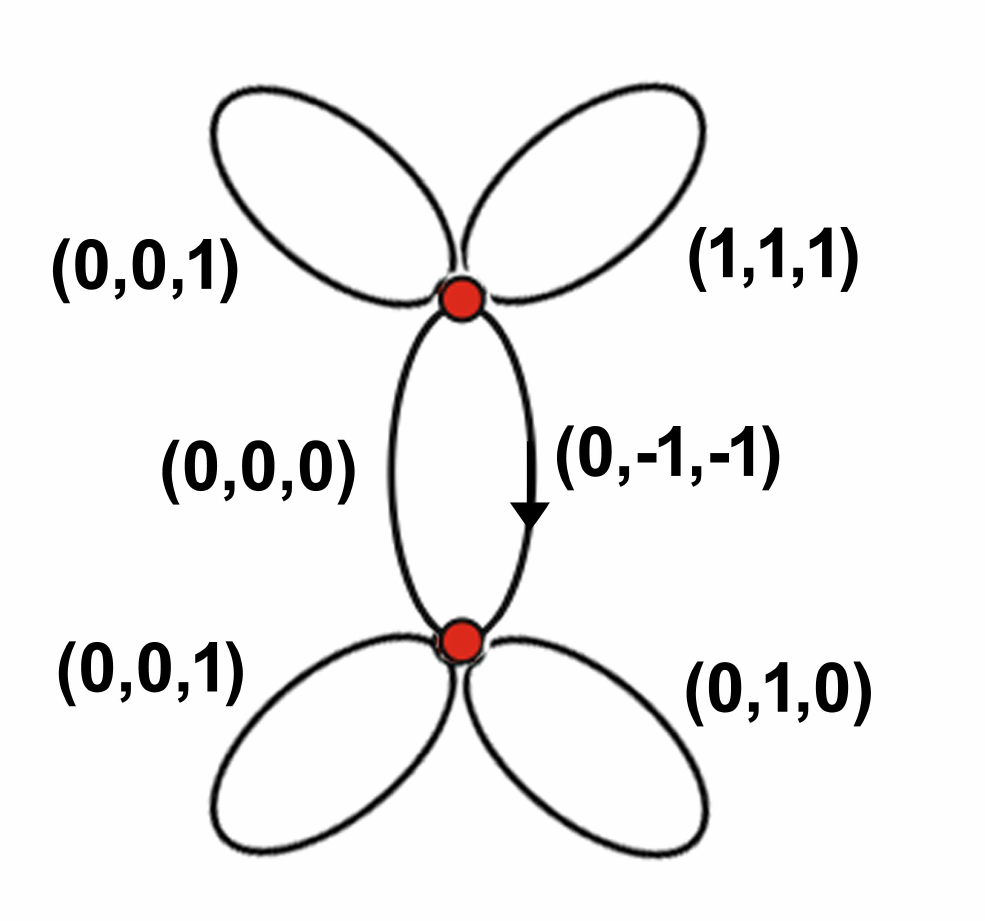}\quad \quad
\includegraphics[width=4.3cm]{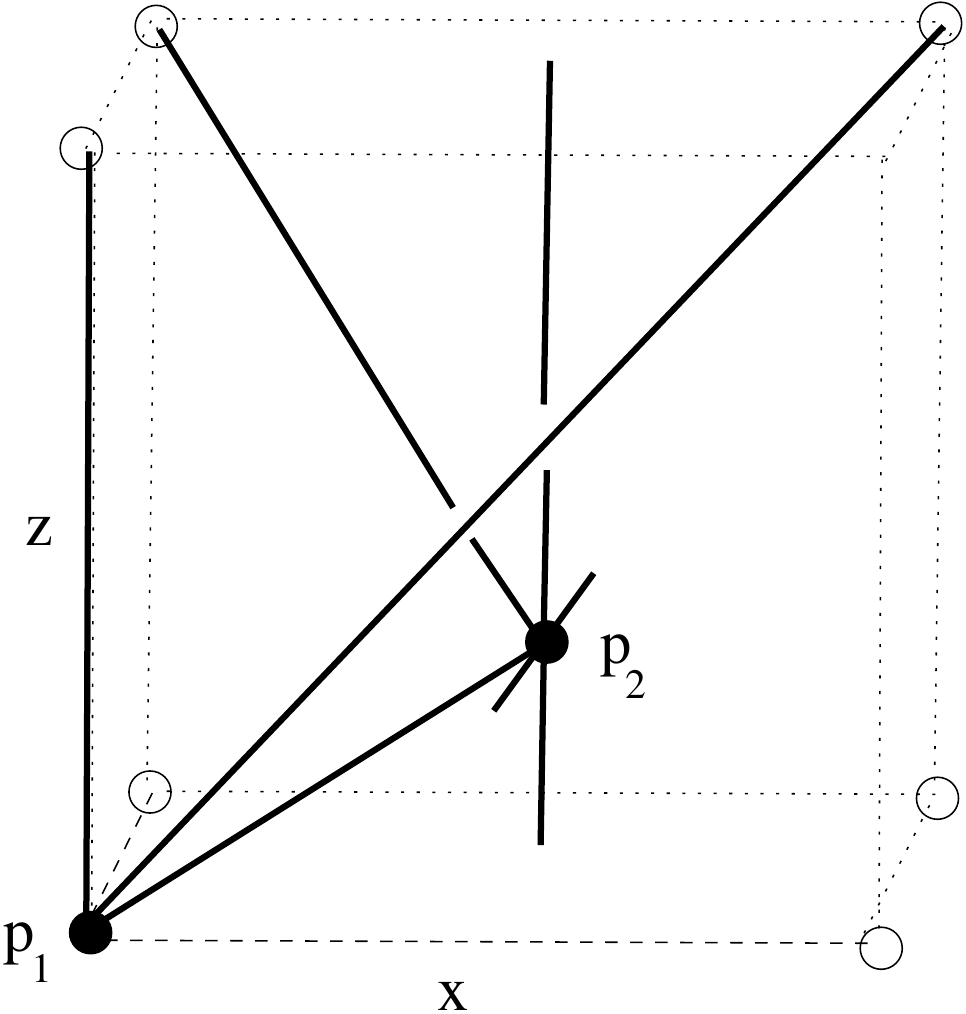} 
\caption{(a) A labelled quotient graph $(H,\lambda)$. (b) Part of the net $\M(p_2)$ in the cube $[0,1)^3$ where $\M(p_2)$ is determined by the LQG together with the standard basis periodic structure $\ul{b}$, a node $p_1$ at the origin and the node $p_2$ in the unit cell $[0,1)^3$.}
 \label{f:mainexample}
\end{figure}
\end{center}

With this normalisation the point $p_2$ can be any point in $[0,1)^3$, subject to the essential disjointness of edges, and we write $\O$ for this set of positions of $p_2$.  
Note that as $p_2$ moves on a small closed circular path around the main diagonal its incident edges are determined and there will be 5 edge crossings with the diagonal. In fact the 2 vertical edges and the 2 horizontal edges which are incident to $p_2$ contribute 2 crossings each, and the other edge incident to $p_2$ contributes 1 crossing. These are the only edge crossings that occur as $p_2$ "carries" its 6 edges of incidence during this motion. 
It follows from similar observations that $\O$ is the disjoint union of 5 pathwise connected sets.  

In this way we see that a pair of nets $\M(p_2), \M(p_2')$, with $p_2, p_2'$ in the same component set, are  \emph{strictly} periodically isotopic in the sense that there is a continuous path of linear periodic nets between them each of which has the same periodic structure basis, namely $\ul{b}$. From this we may deduce that there are at most $5$ periodic isotopy classes of embedded nets $\N$ which have the specific labelled quotient graph $(H,\lambda)$ for some periodic structure. Conceivably there could be  fewer periodic isotopy classes since we have not contemplated isotopy paths of nets, with associated paths of periodicity bases, for which the labelled quotient graph changes several times before returning to $(H,\lambda)$.

Let us also note the following incidental facts about the nets $\M(p_2)$. They are $6$-coordinated periodic nets and so provide examples of critically coordinated bar-joint frameworks, of interest in rigidity theory and the analysis of rigid unit modes. This is also true of course for all frameworks with the same underlying quotient graph.  



\section{Adjacency depth and model nets}\label{s:adjacencydepth}

We now define the {adjacency depth} of a linear 3-periodic net  $\N$.
This positive integer can serve as a useful taxonomic index and
in Sections \ref{s:latticenets}, \ref{s:furtherdirections}  we determine, in the case of some small quotient graphs, the 3-periodic graphs which possess an embedding as a (proper)  linear 3-periodic net with depth 1. 
These identifications also serve as a starting point for the determination of the periodic isotopy types of more general depth 1 embedded nets.

We first review maximal periodicity lattices for embedded nets $\N$ and their primitive periodicity bases.

\subsection{Primitive periodic structure}
Let  $\ul{a}$ be a vector space basis for $\bR^d$ which consists of a periodicity basis for a linear $d$-periodic net $\N$.
The associated translation group $\T(\ul{a})$ of isometries of $\N$ is a subgroup of the space group of $\N$. 
We say that $\ul{a}$ is a \emph{primitive}, or a {\emph{maximal periodicity basis}}, if there is no periodicity  basis $\ul{b}$  such that $\T(\ul{a})$ is a proper subset of $\T(\ul{b})$. 

We focus on 3 dimensions and in order to distinguish mirror related nets we generally consider right-handed periodicity bases the embedded nets $\N$.

The next {well-known} lemma shows that different primitive bases
are simply related by the matrix of an invertible transformation with integer entries and  determinant $1$.
Let $GL(d,\bR)$ be the group of invertible $d\times d$ real matrices, viewed also as linear transformations of $\bR^d$, and let $GL^+(d,\bR)$  be the subgroup of  matrices with positive determinant. Also, let $SL(d,\bR)$ be the subgroup of elements with determinant 1, and  $SL(d,\bZ)$ the subgroup of $SL(d,\bR)$ with integer entries. 

\begin{lem}\label{l:periodicstructures}
Let $\N= (N,S)$ be a linear $3$-periodic net in $\bR^3$ with a primitive right-handed periodicity basis $\ul{b}$ and a right-handed periodicity basis $\ul{a}$. Then $\ul{a}$ is primitive if and only if there is a matrix 
$Z \in SL(3,\bZ)$ with $\ul{a}= Z\ul{b}=(Zb_1, Zb_2, Zb_d)$. 
\end{lem}


\subsection{The adjacency depth $\nu(\N)$ of a linear periodic net} 
While certain elementary linear periodic nets $\N$  have ''natural" primitive periodicity bases it follows from Lemma \ref{l:periodicstructures}
that such a basis is not determined by $\N$.
It is natural then to seek a preferred basis $\ul{a}$ which is a ``good fit" in some sense. 
The next definition provides one such sense, namely that the primitive basis $\ul{a}$ should be one that minimises the adjacency depth of the pair $(\N, \ul{a})$.

\begin{defn}
The \emph{adjacency depth} of the pair $(\N, \ul{a})$, denoted $\nu(\N, \ul{a})$, is depth of the labelled quotient graph LQG$(\N;\ul{a})$, that is
the maximum modulus of its edge labels. 
The \emph{adjacency depth}, or \emph{depth}, of $\N$ is the minimum value, $\nu(\N)$, of the adjacency depths 
 $\nu(\N, \ul{b})$ taken over all primitive periodicity bases $\ul{b}$.
\end{defn}
 
Let $\N$ be a linear $3$-periodic net with periodicity basis $\ul{a}$. Consider the semi-open parallelepipeds (rhomboids) 
\[
P_k:=P_k(\ul{a}):=\{t_1a_1+t_2a_2+t_3a_3: k_i\leq t_i < k_i+1, 1\leq i \leq 3 \},\quad k\in \bZ^3.
\]
These sets form a partition of $\bR^3$, with $P_k$ viewed as a \emph{unit cell} with label $k$. Note that each cell $P_j$ has 26 ''neighbours", given by those cells $P_l$ whose closures intersect the closure of $P_k$. (For diagonal neighbours this intersection is a single point.) Thus we have the equivalent geometric description that $\nu(\N)=1$ if and only if there is a primitive periodicity basis such that the pair of end nodes of every edge lie in neighbouring cells of the cell partition, where here we also view each cell as a neighbour of itself.  

It should not be surprising that for the \emph{connected}  embedded periodic nets of materials the adjacency depth is generally $1$. Indeed while the maximum symmetry embedding $\N_{\rm elv}$ for the net  {\bf elv} has adjacency depth $2$, it appears to us to be the only connected example in the current RCSR listing with $\nu(\N)>1$. The periodic net {\bf elv} gets its name from the fact that its minimal edge cycles have length 11.
On the other hand in Section \ref{s:entanglednets}  we shall see simple examples of multicomponent nets with adjacency depth equal to the number of connected components.
 
\begin{defn}\label{d:affineequivalence}
Let $\N_i= (N_i, S_i), i=1,2,$ be linear 3-periodic nets in $\bR^3$. Then $\N_1$ and $\N_2$ are \emph{affinely equivalent} (resp. \emph{orientedly (or chirally) affinely equivalent})
if there are translates of $\N_1$ and $\N_2$ which are conjugate by a matrix $X$ in $GL(3,\bR)$ (resp. $GL^+(3,\bR)$).
\end{defn}

It follows from the definitions that if $\N_1$ and $\N_2$ are affinely equivalent then they have the same adjacency depth.

The next elementary lemma is a consequence of the fact that linear 3-periodic nets are, by assumption, proper in the sense that their edges must be noncrossing (ie. essentially disjoint).
 
\begin{lem} \label{l:7and8lemma} 
Let $\N$ be a linear 3-periodic net with
a depth 1 labelled quotient graph $(H,\lambda)$. Then there are at most 7 loop edges on each vertex of $H$
 and the multiplicity of edges between each pair of vertices is at most 8. 
\end{lem}

\begin{proof} Let $\ul{a}$ be a periodic structure basis such that $\nu(\N;\ul{a})=1$. Without loss of generality we may assume that $\ul{a}$ is an orthonormal basis.
Let $p_1$ be a node of $\N$. Let $p_2,\dots , p_8$ be the nodes $T_kp_1$ where $k\neq (0,0,0)$ with coordinates equal to 0 or 1, and let $p_9, \dots , p_{27}$ be the nodes $T_kp_1$ for the remaining values of $k$ with coordinates  equal to 0, 1 or $-1$. Every line segment $[p_1,p_t]$ with $t \geq 9$ has a lattice translate which either coincides with or intersects, at midpoints, one of the line segments  $[p_1,p_t]$ with $t<9$. Since $\N$ has no edge crossings it follows that there are at most 7 translation classes for the edges associated with multiple loops of $H$ at a vertex.

We may assume that $p_1=(0,0,0)$. Let  $q_1$ be a node in $(0,1)^3$ in a distinct translation class. Since the depth is $1$ it follows that the edges $[q_1, p]$ in $\N$ with $p$ a translate of $p_1$, correspond to the positions $p=\lambda$, where $\lambda \in \bZ^3$ has coordinates taking the values $-1,0$ or $1$. The possible values of $\lambda$ are also the labels in the quotient graph of $\N$ for the edges directed from the orbit vertex of $p_1$ to the orbit vertex of $q_1$. There are thus 27 possibilities for the edges $[q_1, p]$, and we denote the terminal nodes $p$ by $\lambda_a, \lambda_b, \dots $. 

Since $\N$ is a proper net, with no crossing edges, we have the constraint that $k=(\lambda_a+\lambda_b)/2$ is not a lattice point for any pair $\lambda_a, \lambda_b$. For otherwise $[q_1+k,\lambda_b+k]$ is an edge of $\N$  and its midpoint coincides with the midpoint of $[q_1,\lambda_a]$.
It follows from the constraint that there are at most 8 terminal nodes.
\end{proof}

The following proposition gives a necessary condition for a general 3-periodic graph $(G,T)$ to have an embedding as a proper linear 3-periodic net. Moreover this condition is useful later for the computational determination of possible topologies for nets in $\fN_2$.

We say that a labelled quotient graph $(H, \lambda)$ has \emph{the divisibility property}, or is \emph{divisible}, if for some pair of labelled edges $(v_1v_2, k), (v_1v_2,l)$, with the same vertices, and possibly $v_1=v_2$, the vector $k-l$ is divisible in the sense that it is equal to $nt$, with $t \in\bZ^3$ and $n\geq 2$ an integer. If this does not hold then the 3 entries of $k-l$ are coprime and $(H, \lambda)$ is said to be  \emph{indivisible}.

\begin{prop}\label{p:divisible}
Let $\N$ be a (proper) linear 3-periodic net in $\bR^3$ and let $(H,\lambda)$ be a labelled quotient graph  associated with some periodic structure basis for $\N$. Then $(H,\lambda)$ is indivisible. 
\end{prop}

\begin{proof}
Let $(v_1v_2, k), (v_1v_2,l)$ be two edges of $(H,\lambda)$, with $v_1 \neq v_2,$. Then $\N$ has the incident edges $[(p(v_1),0), (p(v_2),k)], [(p(v_1),0), (p(v_2),l)]$ which, by the properness of $\N$, are not colinear. 
Without loss of generality and to simplify notation assume that the periodicity basis defining the labelled quotient graph is the standard orthonormal basis. Then these edges are $[p(v_1), p(v_2)+k]$ and $[p(v_1), p(v_2)+l]$. 
Taking all translates of these 2 edges by integer multiples of $t=k-l$ we obtain a $1$-periodic (zig-zag) subnet, $\Z$ say, of $\N$ with period vector $t=(t_1, t_2, t_3)$.
 
Suppose next that $t$ is divisible with $t=nt',t'\in \bZ^3$ and $n\geq 2$. Since $\Z+t'$ does not coincide with $\Z$ there are crossing edges, a contradiction.

Consider now two loop edges $(v_1v_1, k), (v_1v_1,l)$ and corresponding incident edges in $\N$, say
$[p(v_1), p(v_1)+k]$ and $[p(v_1), p(v_1)+l]$. Taking all translates of these 2 edges by the integer combinations
$n_1k+n_2l$, with $(n_1,n_2) \in \bZ^2$, we obtain a $2$-periodic subnet, with period vectors $\{k,l\}$, which is an embedding of {\bf sql}. The vector  $t=k-l$ is a diagonal vector for the parallelograms of this subnet and so, as before, $t$ cannot be divisible.
\end{proof}

 As a consequence of the proof we also see that an embedded net is improper if either of the following conditions fails to hold: (i) for pairs of loop edges in the LQG with the same vertex the two labels generate a maximal rank 2 subgroup of the translation group, (ii) for pairs of nonloop edges the difference of the two labels  
generates a maximal rank 1 subgroup.

\subsection{Model nets and labelled quotient graphs} \label{ss:isomorphicnets}
We first note that every abstract $3$-periodic graph $(G,T)$ can be represented by a {model net} $\M$ in $\bR^3$ with standard periodicity basis $\ul{b}$, 
in the sense that $G$ is isomorphic to the structure graph $G(\M)$ of $\M$ by an isomorphism which induces a representation of $T$ as the translation group of $\M$ associated with $\ul{b}$. Formally, we define a \emph{model net} to be such a pair $(\M, \ul{b})$  but we  generally take the basis choice as understood and use notation such as $\M, \M(p,i), \M(p_2)$ etc.

 

 Let $(G,T)$ be a 3-periodic graph with periodic structure $T$ and let $H=G/T$ be the quotient graph $(V(H),E(H))$ determined by $T$.
Identify the automorphism group $T$ with the integer translation group of $\bR^3$. This is achieved through the choice of a group isomorphism $i:T \to \bZ^3$ and this choice introduces an ordered triple of generators and coordinates for $T$. Any other such map, $j$ say, has the form $X\circ i$ where $X\in GL(3,\bZ) $.

Label the vertices of $G$ by pairs $(v_k,g)$ where $g\in T$ and $v_1, \dots , v_n$ is a complete set of representatives for the $T$-orbits of vertices. For the sake of economy we also label the vertices of $H$ by $v_1, \dots , v_n$.
Let $p_H:V(H) \to [0,1)^3$ be any \emph{injective} placement map. Then there is a unique injective placement map $p:V(G) \to \bR^3$ induced by $p$ and $i$, with
\[
p((v_k,g)) = p_H(v_k) + i(g),\quad 1\leq k\leq n, g \in T.
\]
Thus the maps $p_H, i$  determine a (possibly improper) {model embedded net} for $(G,T)$ which we denote as $\M(p,i)$. This net is possibly improper in the sense that some edges intersect. Write $\H(i)$ for 
the labelled quotient graph $(H,\lambda)$ of $\M(p,i)$ with respect to  $\ul{b}$. As the notation implies, this depends only on the choice of $i$ which coordinatises the group $T$.

With $i$ fixed we can consider continuous paths of such placements, say $p_H^t, 0\leq t \leq1$, which in turn induce paths of model nets,
$t \to \M_t= \M(p_H^t), 0\leq t \leq1$. (See also Section \ref{s:mainexample}.) When there are no edge collisions, that is, when all the nets $\M_t$ in the path are proper, this provides a \emph{strict periodic isotopy} 
between the  the pairs $(\M_0,\ul{b})$ and $(\M_1,\ul{b})$ and their given periodic structure bases, $\ul{b}$. (Such isotopy is also formally defined in the remarks following Definition \ref{d:deformationequivalentnets}.)  



Note that if $H$ is a \emph{bouquet graph}, that is, has  a single vertex, then the strict periodic isotopy determined by $t \to p_H^t$ 
between two model nets for $\H$  corresponds simply to a path of translations.



 
In the next proposition the isomorphism of 3-periodic graphs in (i) is in the sense of Eon \cite{eon-2011}. This requires that there is a graph isomorphism $G \to G'$, induced by $\gamma :V \to V'$ and a group isomorphism
$\pi:T \to T'$ such that 
\[
\gamma(g(v)) = \pi(g)(\gamma(v)), \quad v \in V, g\in T. 
\]

\begin{prop}\label{p:isomorphiclatticenets}
Let $(G,T), (G', T')$ be 3-periodic graphs with labelled quotient graphs $\H(i)=(H,\lambda), \H(i')=(H',\lambda')$ arising from isomorphisms $i:T \to \bZ^3$ and $i':T' \to \bZ^3$. Then the following statements are equivalent.
\medskip

(i) $(G,T)$ and $(G', T')$ are isomorphic
as 3-periodic graphs.
\medskip

(ii) There is a graph isomorphism $\phi: H \to H'$ and {$X \in GL(3,\bZ)$ with $|\det X|=1$ such} that $\lambda'(\phi(e))= X(\lambda(e))$ for all directed edges $e$ of $\H(i)$. 
\medskip
\end{prop}

\begin{proof}That (i) follows from (ii) is elementary. 
On the other hand consider and isomorphism between $(G,T)$ and $(G', T')$. The basis choice $i:T\to \bZ^3$ and the map $\pi$ imply a basis choice $\hat{i}=i\circ \pi^{-1}$ for $T'$. Let 
$X$ be the matrix, with transformation $\tilde{X}$, such that $\tilde{X}\circ \hat{i}= i'$.  Then (ii) follows.
\end{proof}

In the case when $H=G/T$ and $H'=G/{T'}$ 
are bouquet graphs one can say much more.  Any graph isomorphism $\gamma:G \to G'$ lifts to a linear isomorphism between the model nets $\M, \M'$ determined by \emph{any} pair $T, T'$ of maximal periodic structures. See for example Theorem 3 of Kostousov \cite{kos}.  It follows that for bouquet quotient graph nets we have the following stronger theorem.

\begin{thm}\label{t:kostousov} Let $\M(p,i)$ and $\M(p',i')$ be model nets, with nodes on the integer lattice, for 3-periodic graphs $(G,T)$ and $(G,T')$ with bouquet quotient graphs. Then the following are equivalent.
\medskip

(i) $G$ and $G'$ are isomorphic as countable graphs.

(ii) $\M(p,i)$ and $\M(p',i')$ are affinely equivalent by a matrix $X$ in $GL(3,\bZ)$.
\end{thm}

\begin{defn}
A  (proper) linear 3-periodic net $\N$ is a lattice net if its set of nodes is a lattice in $\bR^3$.
\end{defn}

 Equivalently $\N$ is a lattice net if its quotient graph is a bouquet graph. One may also define a \emph{general lattice net} in $\bR^3$ as a (not necessarily proper) embedded net  whose quotient graph is a bouquet graph.
Theorem \ref{t:kostousov} shows that lattice nets (even general ones) are classified up to affine equivalence by their topologies.
In Theorem \ref{t:1vertexQG} we obtain a proof of this 
in the depth 1 case, independent of Kostousov's theorem, through a case-by-case analysis. Also we show that for the connected depth 1 lattice nets there are 19 classes.

In principle Proposition \ref{p:isomorphiclatticenets} could be used as a basis for a computational classification of periodic nets with small quotient graphs with a depth 1 labelling. 
However we note that there are more practical filtering methods such as those underlying Proposition \ref{p:117topologies} which determines the 117 connected topologies associated with certain depth 1 nets which are supported on 2 parallel vertex lattices in a bipartite manner.

\section{Linear graph knots}\label{s:graphknots} 

Let $H$ be a multigraph, that is, a general finite graph, possibly with loops and with an arbitrary multiplicity of "parallel" edges between any pair of vertices. Then a \emph{graph knot} in $\bR^3$ is a faithful geometric representation of $H$ where the vertices $v$ are represented as distinct points $p(v)$ in $\bR^3$ and  each edge $e$ with vertices $v, w$ is represented by a  smooth path $\tilde{p}(e)\subseteq \bR^3$, with endpoints $p(v), p(w)$. Such paths are required to be disjoint except possibly at their endpoints. Thus a graph knot $K$ is formally a triple $K=(H,p,\tilde{p})$, and we may also refer to this triple as a \emph{spatial graph} or as a \emph{proper placement} of $H$ in $\bR^3$.
It is natural also to denote a graph knot $K$ simply as a pair $(N,S)$, where $N$ is the set of vertices, or \emph{nodes}, $p(v)$ in $\bR^d$, and $S$ is the set of nonintersecting paths 
$\tilde{p}(e)$.
We remark that spatial graphs feature in the mathematical theory of  intrinsically linked connected graphs \cite{con-gor}, \cite{koh-suz}.

One can similarly define a graph knot $K$ in any smooth manifold and of particular relevance is the Riemannian 
manifold known as a \emph{flat 3-torus}. This is essentially the topological 3-torus identified naturally with the set $[0,1)^3$ and the topology, in the usual mathematical sense, is the natural one associated with continuity of the quotient map $\bR^3 \to [0,1)^3$.  Moreover we define a \emph{line segment} in the flat torus to be the image of a line segment in $\bR^3$ under this quotient map.
The curiosity here is that such a flat torus line segment may  appear as the union of several line segment sets in $[0,1)^3$.

We formally define a \emph{linear graph knot in the flat torus}  to be a triple $(H,p,\tilde{p})$, or a pair $(N, S)$, where the vertices, or nodes, $p(v)$ lie in $[0,1)^3$ and the paths, or edges, $\tilde{p}(e)$ are essentially disjoint flat torus line segments. Intuitively, this is simply a finite net in the flat 3-torus with linear nonintersecting edges.


We now associate a linear graph knot $K$ in $[0,1)^3$ with an embedded net $\N$ with a specified periodicity basis $\ul{a}$.
Informally, this is done
by replacing $\N$ by its affine normalisation $\N'$, wherein $\ul{a}$ is rescaled to the standard basis, and defining $K$ as the intersection of the body $|\N'|$ with $[0,1)^3$. That is, one takes the simplest model net $\N'$ for $\N$ and ignores everything outside the cube $[0,1)^3$.


For the formal definition, let $(F_v, F_e)$ be a \emph{motif}  for  $(\N, \ul{a})$,  where $F_v\subset N$ (resp. $F_e\subset S$) is a finite set of representatives for translation classes of nodes (resp. edges) of $\N= (N,S)$, with respect to $\ul{a}$. 
Let $\pi: \bR^3 \to [0,1)^3$ be the natural quotient map associated with the ordered basis $\ul{a}$. This is a composition of the linear map for which $\ul{a}$ maps to the standard right-handed basis, followed by the quotient map. Define $p: F_v \to [0,1)^3$ to be the induced injection and  $\tilde{p}: F_e \to [0,1)^3$ to be the induced map from closed line segments to closed line segments of the flat torus $[0,1)^3$.

\begin{defn}\label{d:lineargraphknot} Let $H$ be the  quotient graph for the pair $(\N, \ul{a})$.
The triple
$(H, p, \tilde{p})$ is the \emph{linear graph knot}
of  $(\N, \ul{a})$ and is denoted as $\lgk(\N,\ul{a})$.
\end{defn}

Since $\N$ is necessarily proper, with essentially disjoint edges, the placement
$(H, p, \tilde{p})$ has essentially disjoint edges and so is a linear graph knot. 

 
Note that the  linear graph knot determines
uniquely the net $\N'$ which in fact can be viewed as its covering net. It follows immediately that if
$\lgk(\N,\ul{a}) = \lgk(\M,\ul{b})$ then $\N$ and $\M$ are linear periodic nets which are orientedly affine equivalent. 

We now give some simple examples together with perspective illustrations. Such illustrations are unique up to translations within the flat $3$-torus and so it is always possible to arrange that the nodes are interior to the open unit cube. In this case the 3D diagram  reveals their valencies. On the other hand, as we saw in the partial body examples in Section \ref{s:mainexample} it can be natural to normalise and simplify the depiction by a translation which moves a node to the origin.

\begin{eg}
The simplest proper linear $3$-periodic net is the primitive cubic net $\N= \N_{\rm pcu}$. We may normalise this so that the node set  is a translate of the set $\bZ^3$. The standard primitive periodic structure basis gives the graph knot $\lgk(\N,\ul{b})$, which we denote as $K_{\rm pcu}$ and which is illustrated in  Figure \ref{f:pcu}.  The 3 ``line segment" edges in the flat torus are here depicted by 3 pairs of line segments.
The  quotient graph of $\N_{\rm pcu} $, which is also the underlying graph of $K_{\rm pcu} $, has one vertex and 3 loop edges. 
Note that if the node is translated to the origin 
then the depiction of the loop edges is given by 3 axial line segments.

\begin{center}
\begin{figure}[ht]
\centering
\includegraphics[width=4cm]{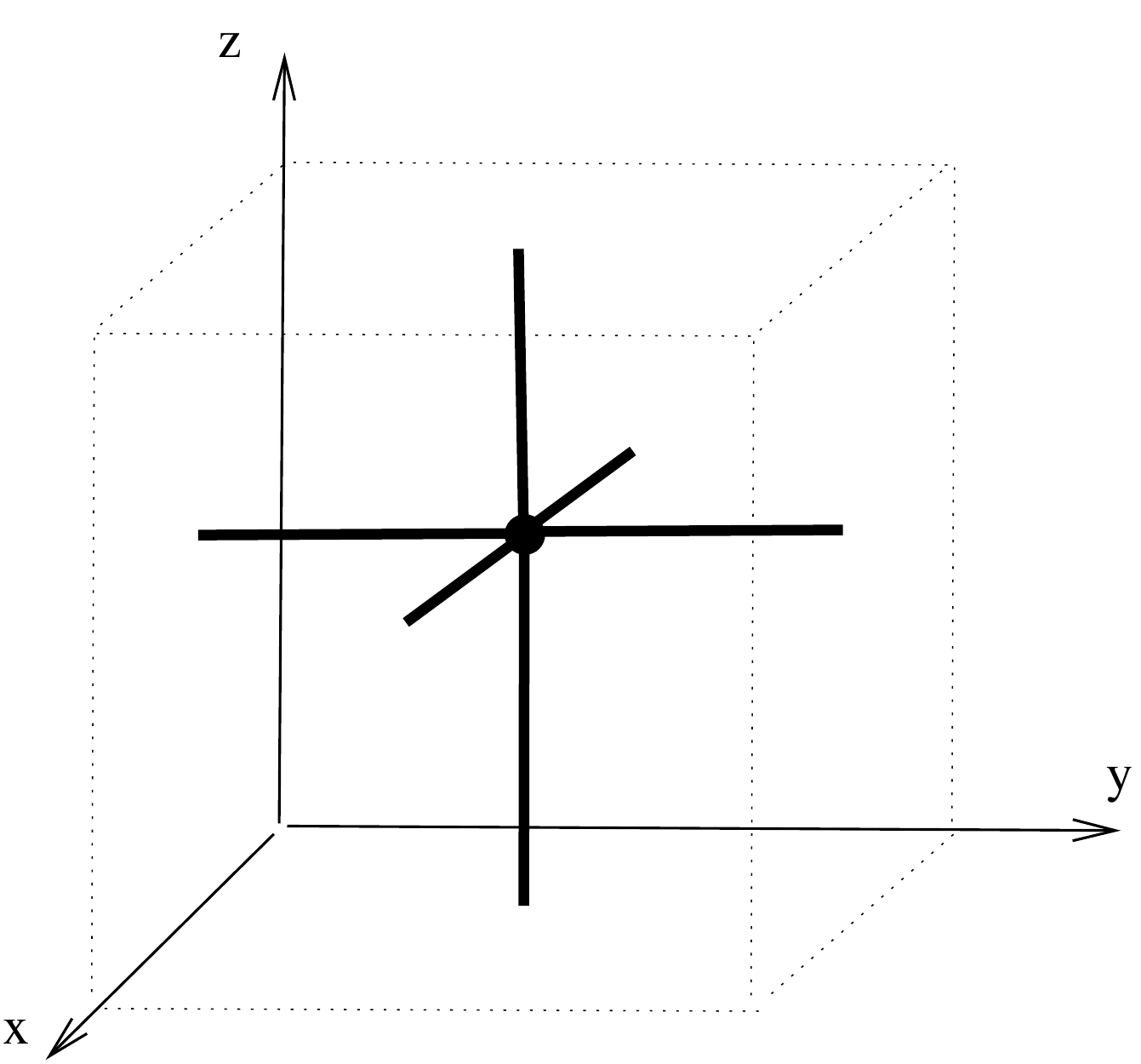}
\caption{The linear graph knot $K_{\rm pcu}$ on the flat 3-torus $[0,1)^3$.} 
\label{f:pcu}
\end{figure}
\end{center}

By taking a union of $n$  disjoint generic translates (within $[0,1)^3$) of $K_{\rm pcu} $
one obtains the linear graph knot of an associated multicomponent linear net. In Theorem \ref{t:multigridcount} we compute the number of periodic  isotopy classes of such nets and the graph knot perspective is helpful for the proof of this.
\end{eg}

\begin{eg}
Figure \ref{f:bcu_and_srs} shows linear graph knots (or finite linear nets) on the flat torus for the maximal symmetry nets $\N_{\rm bcu}$ and $\N_{\rm srs}$.
Each is determined by a natural primitive right-handed depth 1 periodicity basis $\ul{a}$  which, by the definition of $\lgk(\N,\ul{a})$, is normalised to $\ul{b}$. The quotient graphs for these examples  are, respectively, the bouquet graph with 4 loop edges and the complete graph on 4 vertices. The periodic extensions of these graph knots give well-defined {model nets}, say $\M_{\rm bcu}$ and $\M_{\rm srs}$ which are orientedly affinely equivalent to the maximal symmetry nets $\N_{\rm bcu}$ and $\N_{\rm srs}$.
\end{eg}
 
\begin{center}
\begin{figure}[ht]
\centering
\includegraphics[width=4cm]{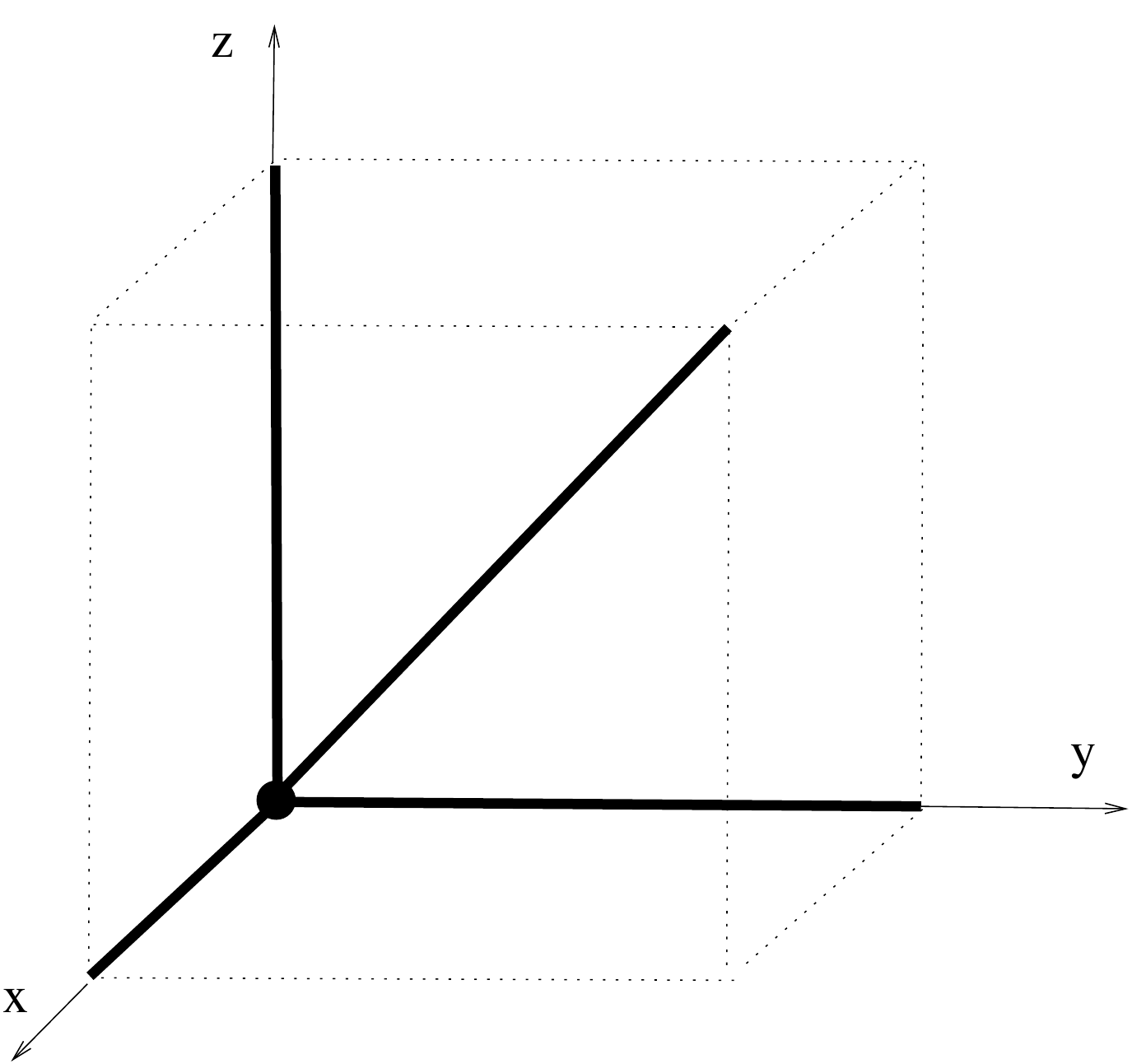}
\quad \quad
\includegraphics[width=4cm]{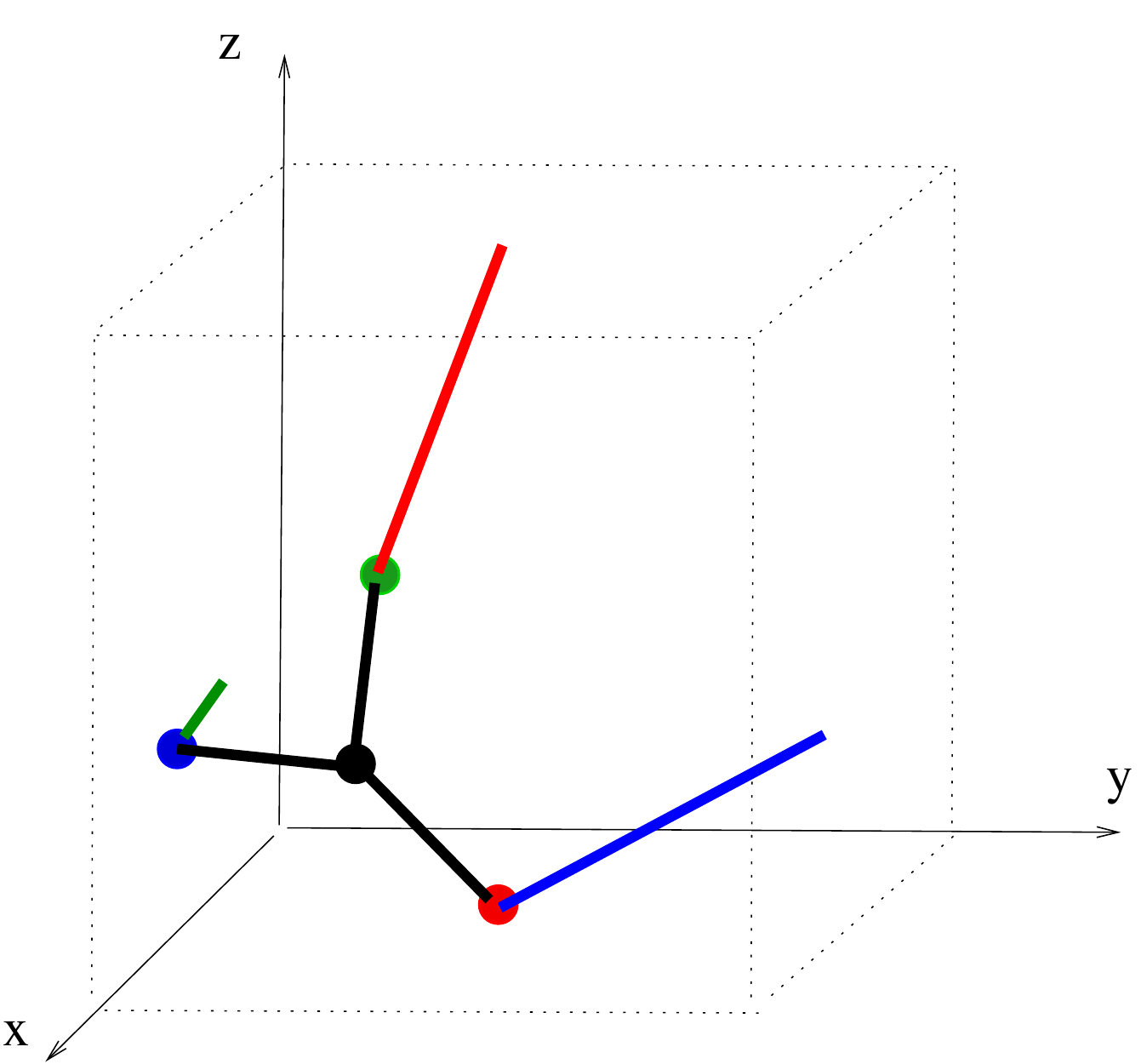}
\caption{Linear graph knots on the flat torus for {\bf bcu} and {\bf srs}.}
\label{f:bcu_and_srs}
\end{figure}
\end{center} 
 
\begin{eg} 
The linear 3-periodic  net $\N_{\rm dia}$ for the diamond crystal net (with maximal symmetry) has a periodic structure basis $\ul{a}$  corresponding to 3 incident edges of a regular tetrahedron, and has a motif consisting of 2 vertices and 4 edges. The graph knot $K_{\rm dia}= \lgk(\N_{\rm dia},\ul{a})$ is obtained by (i) an oriented affine equivalence with a model net $\M_{\rm dia}$ with standard orthonormal periodic structure basis, and (ii) the intersection of $\M_{\rm dia}$ with $[0,1)^3$. This graph knot has an underlying graph  $H(0,4,0)$ (in the notation of Section \ref{s:furtherdirections}) with 2 vertices and 4 nonloop edges.

In Figures \ref{f:dia_graphs}, \ref{f:dia_graphs2}  we indicate 4  graph knots which define model nets each with underlying net (structure graph) equal to ${\bf dia}$. In fact
the graph knots $K_1, K_2$  are \emph{rotationally linearly isotopic} (see Definition \ref{d:knotisotopies}) in the sense that there is an isotopy defined by a motion of the central vertex of $K_2$ through the floor/roof which terminates at a graph knot which is the image of $K_1$ under a rotation automorphism.

\begin{center}
\begin{figure}[ht]
\centering
\includegraphics[width=4cm]{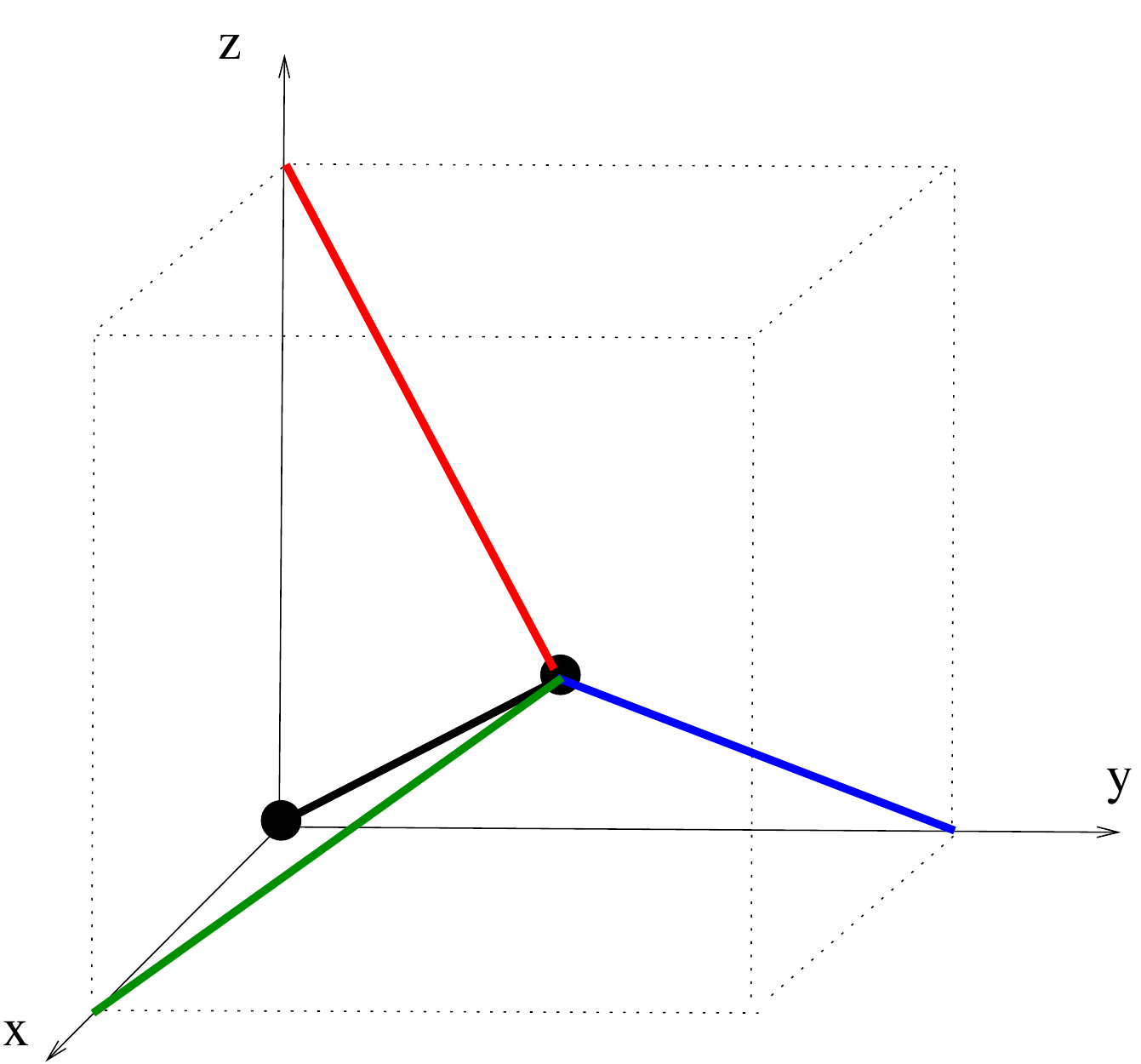}
\includegraphics[width=4cm]{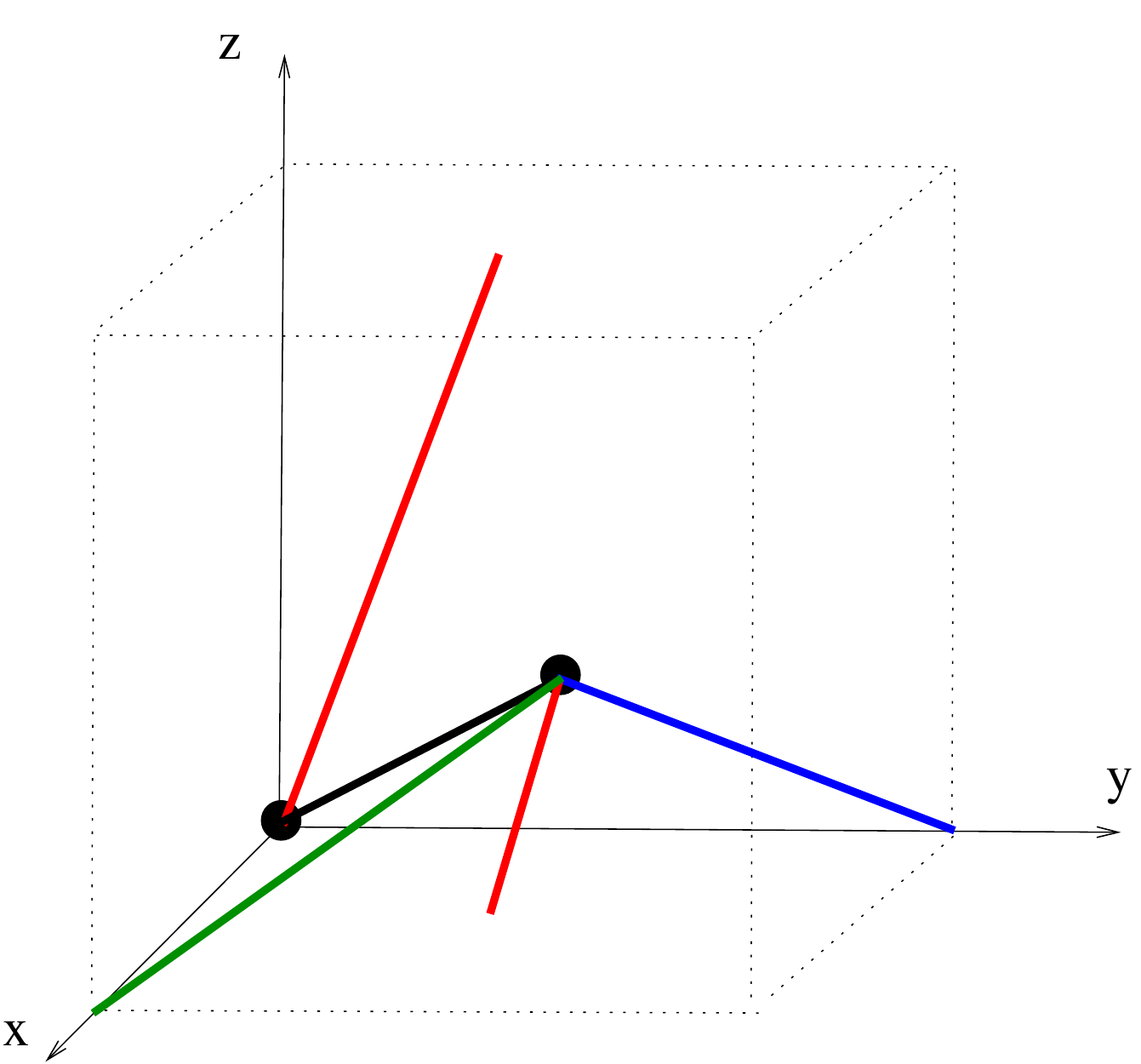}
\caption{Linear graph knots $K_1, K_2$ associated with {\bf dia}.}
\label{f:dia_graphs}
\end{figure}
\end{center}

On the other hand
$K_2$ and $K_3$ are \emph{linearly isotopic} in terms of a motion of the vertex of $K_2$ at the origin to the position of the left hand vertex of $K_3$. It follows from this that the associated model nets $\M_1, \M_2, \M_3$ are strictly periodically isotopic, simply by taking the periodic extension of these isotopies to define periodic isotopies.

\begin{center}
\begin{figure}[ht]
\centering
\includegraphics[width=4.5cm]{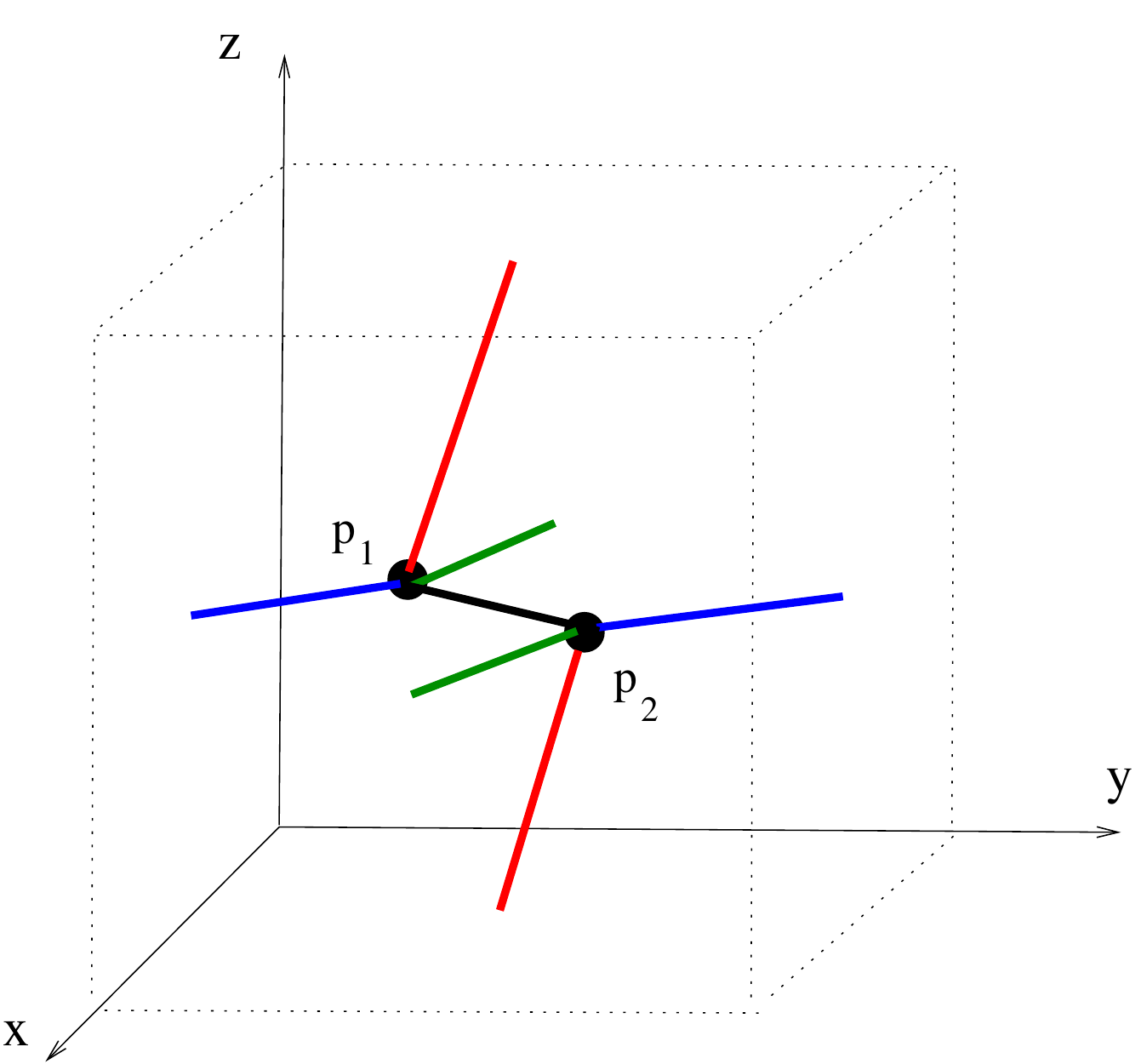}
\includegraphics[width=4.5cm]{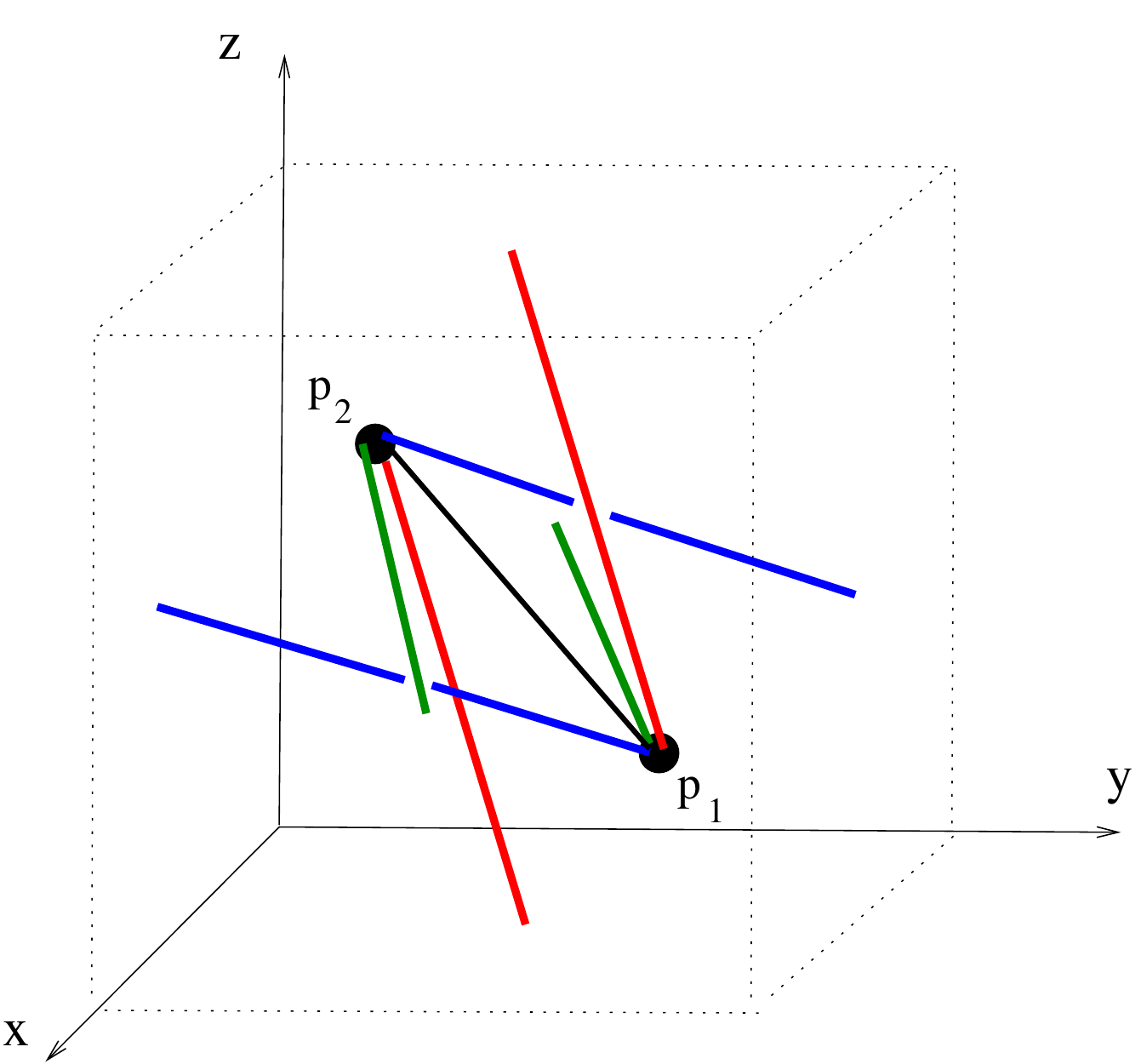}
\caption{Linear graph knots $K_3, K_4$ associated with {\bf dia}.}
\label{f:dia_graphs2}
\end{figure}
\end{center}

In contrast to this observe first that the linear graph knot $K_4$ is obtained from $K_3$ by a continuous motion of the nodes $p_1$ and $p_2$ to their new positions in the 3-torus. Such a motion defines a \emph{linear homotopy} in the natural sense.
The (uniquely) determined edges of the intermediate knots in this case inevitably cross at some point in the motion so these linear  homotopies are not linear isotopies. The model net $\M_4$ for the knot $K_4$ is in fact not periodically isotopic to the unique maximal symmetry embedding $\N_{\rm dia}$, and so is self-entangled. We show this in Example \ref{eg:self_entgld_dia}. 
\end{eg}

In the model nets of the examples above we have taken a primitive periodic structure basis with minimal adjacency depth. In view of this the represented edges between adjacent nodes in these example have at most 2 \emph{diagramatic} components, that is they reenter the cube at most once. In general the linear graph knot associated with a periodic structure basis of depth 1 has edges which can reenter at most 3 times.

\begin{rem}\label{r:mirror2vertex} We shall consider families of embedded nets  up to oriented affine equivalence and up to periodic isotopy. In general there may exist enantiomorphic pairs, that is, mirror images $\N, \N'$ which are not equivalent. This is the case, for example, for embeddings of {\bf srs}. However, such inequivalent pairs do not exist if the quotient graph is a single vertex (lattice nets) or a pair of vertices with no loop edges (double lattice nets with bipartite structure). This becomes evident in the latter case for example on considering an affine equivalence with a model net for which the point $(1/2,1/2,1/2)$ is the midpoint of the 2 representative nodes in the unit cell. This midpoint serves as a point of inversion for the model net (or, equivalently, its graph knot). The  graph knots in Figure \ref{f:dia_graphs2} indicate such centered positions.
\end{rem}

\begin{rem}
We have observed that for a linear $3$-periodic net the primitive right-handed periodicity bases $\ul{a}$ are determined up to transformations by matrices in $SL(3,\bZ)$. Such matrices induce chiral automorphisms of the flat $3$-torus which preserve the linear structure. Accordingly (and echoing the terminology for embedded nets) it is natural to define two graph knots on the same flat torus to be \emph{orientedly (or chirally) affinely equivalent} if they have translates which correspond to each other under such an automorphism. Thus, to each linear 3-periodic net
$\N$ one could associate its  \emph{primitive graph knot}, 
on the understanding that it is only determined up to oriented affine equivalence.
\end{rem}

\begin{rem} 
We remark that triply periodic surfaces may be viewed as periodic extensions of compact surfaces on the flat 3-torus.  It follows that the tilings and triangulations of these compact surfaces generate special classes of linear 3-periodic nets. Such nets have been considered, for example, in the context of periodic hyperbolic surfaces and minimal surfaces, where the methods of hyperbolic geometry play a role in the definition  of isotopy classes \cite{eva-et-al}, \cite{hyd-et-al-2003}. See also Hyde and Delgado-Friedrichs \cite{hyd-del}.
\end{rem}

\section{Isotopy equivalence}\label{s:isotopyForNets}
Consider the following informal question. 
\bigskip

\emph{When can  $\N_1$ be deformed into $\N_2$ by a continuous path with no edge crossings}  ?
\medskip


This question is not straightforward to approach for two reasons. Firstly, a linear periodic net may contain, as a finite subnet, a linear realisation of an arbitrary knot or link. 
For example, the components of $\N$  could be translates of a linear realisation of an arbitrary finite knot where all vertices have degree $2$. (Here $\N$ would have dimension type $\{3;0\}$.) Thus,  resolving the question by means of discriminating invariants is in general as hard a task as the corresponding one for knots and links. Secondly, the rules for such deformation equivalence need to be decided upon, and, a priori, the deformation equivalence classes are dependent on these rules. 


The following definition may be regarded as the natural form of isotopy equivalence appropriate for the category of embedded periodic nets in 3 dimensions which have line segment bonds, no crossing edges and no coincidences of node locations (node collisions). 

\begin{defn}\label{d:deformationequivalentnets} Let $\N_0$ and $\N_1$ be proper linear $3$-periodic nets in $\bR^3$. Then   $\N_0$ and $\N_1$  are  \emph{periodically isotopic}, or have the same \emph{periodic isotopy type}, if there is a family of such {(noncrossing)} nets, $\N_t$, for $0<t<1$, for which
\medskip

(a) there is a continuous path of bases of $\bR^3$, $t \to \ul{a}^t$, $0\leq t\leq1$, where $\ul{a}^t$ is a right-handed  periodicity basis for $\N_t$,
\medskip

(b) there are bijective functions
$
f_t:|\N_0| \to |\N_t|, 
$
{for} $ 0\leq t\leq 1$, which map nodes to nodes,  such that,
\medskip

(i) $f_0$ is the identity map {on $|\N_0|$},

(ii) for each {node} point $p$ in $|\N_0|$ the map $t \to f_t(p)$ is continuous,

(iii) {the restriction of} $f_t$ to each edge $[a,b]$ is the unique affine map onto the image edge, $[f_t(a), f_t(b)]$  in $|\N_t|$ determined by linear interpolation. 

\end{defn}

 We make a number of immediate observations:

1) The condition (iii) could be omitted but is a conceptual convenience in that it implies that each map $f_t$ from the body of $\N_0$ to the body of $\N_t$ is determined by its restriction to the nodes.

2) The definition applies to entangled nets with several connected components and in this case the isotopy can be viewed as a set  of $n$ independent periodic isotopies, for the $n$ components, with the same time parameter $t$ and periodicity bases $\ul{a}^t$, and subject only to the noncollision of components for each value of $t$.

3) Every such net $\N_0$ is periodically isotopic to a model net
$\N_1=\M$ with periodicity basis $\ul{b}$. Indeed, for any right-handed periodicity basis $\ul{a}$ for $\N$  there is an elementary isotopy equivalence from $(\N, \ul{a})$ to a {unique} pair $(\M,\ul{b})$ which is determined  by a path of transformations from $GL^+(d,\bR)$ which in turn is determined by any continuous path of bases from $\ul{a}$ to $\ul{b}$.

4) If $\N_0$ and $\N_1$ are orientedly affine equivalent then they are periodically isotopic since the topological group $GL^+(3,\bR)$ is path-connected.

\medskip

We also define the pair $(\N_0,\ul{a})$ to be \emph{strictly periodically isotopic} to  the pair $(\N_1, \ul{b})$ if there is an isotopy equivalence $(\ul{a}^t, (f_t))$, as in parts (a), (b) of the definition with $\ul{a}^0 = \ul{a}$ and $\ul{a}^1= \ul{b}$.  In view of the previous observations we have the following
\medskip

 \emph{Equivalent definition.} The embedded periodic nets $\N$ and $\N'$ in $\bR^3$  are periodically isotopic if  there is a rescaling and rotation of $\N'$ to a net $\N''$ so that
(i) $\N$ and $\N''$ have a common embedded translation group with basis $\ul{a}$, and
(ii) $(\N, \ul{a})$   and $(\N'',\ul{a})$ are strictly periodic isotopic.
\medskip

Strict periodic isotopy is evidently an equivalence relation on the set of pairs $(\N,\ul{a})$. Periodic isotopy is also an equivalence relation  but this is not so immediate. However, as the next proof shows, one can replace a pair of given periodic isotopies, between $\N_0$ and $\N_1$ and between
$\N_1$ and $\N_2$, by a new pair such that the paths of periodicity bases can be concatenated, and so provide an isotopy between $\N_0$ and $\N_2$.  

\begin{thm}\label{t:equivalencerelation}
Periodic isotopy equivalence is an equivalence relation on the set of proper linear $3$-periodic nets.
\end{thm}

\begin{proof} 
Let $\{(\ul{a_t}),(f_t):0\leq t\leq 1\}$ be an isotopy equivalence for $\N_0, \N_1$ as above and let $\{(\ul{b_t}),(g_t):1\leq t\leq 2\}$ be an isotopy equivalence between $\N_1$ and $\N_2$. 
Suppressing the implementing maps $f_t$ and $g_t$ we may denote this information as
\[
\N_0 \rightsquigarrow_{(a^t)} \N_1,\quad \N_1 \rightsquigarrow_{(b^t)} \N_2.
\]
We now have 2 periodic structures $\ul{a}^1$ and $\ul{b}^0$ on $\N_1$.  If they were the same then a periodic isotopy between $\N_0$ and $\N_2$ could be completed by the simple concatenation of these paths. However, in general we must choose new periodic structures to achieve this.

For a periodic structure basis $\ul{e}= \{e_1,e_2 , e_3\}$ and $k= (k_1, k_2, k_3) \in \bZ^3$ let us write $k\cdot \ul{e}$ for $\{k_1e_1,k_2e_2 , k_3e_3\}$.
We have $\ul{a}^1= j\cdot \ul{e} $ for some primitive periodic structure basis $\ul{e}$ of $\N_1$. Similarly 
$\ul{b}^0= j'\cdot \ul{e'} $ for some primitive periodic structure basis $\ul{e'}$. Since primitive right-handed periodicity bases on the same linear periodic net are equivalent by a linear map $X \in GL^+(3,\bZ)$, it follows that the vectors of $\ul{e'}$ are integral linear combinations of the vectors of $\ul{e}$.
Thus the vectors of $\ul{b}^0$  are integral linear combinations of the vectors of $\ul{e}$. It follows that we can now find elements $k, k' \in \bZ^3$ so that 
the vectors of $k \cdot \ul{a}^1$ are integral linear combinations of the vectors of  $k' \cdot \ul{b}^0$.

Consider now the induced isotopy equivalences
\[
\N_0 \rightsquigarrow_{(k\cdot \ul{a}^t)} \N_1,\quad \N_1 \rightsquigarrow_{(k'\cdot \ul{b}^t)} \N_2.
\]
These isotopies are identical to the previous isotopy equivalences at the level of the paths of individual nodes, but the framing periodic structure bases have been replaced.
These periodic isotopies do not yet match, so to speak, but we note that the second isotopy equivalence implies an isotopy equivalence from $(\N_1, \ul{d})$ to some $(\N_2, \ul{d'})$ whenever the periodic structure basis $\ul{d}$ has vectors
which are  integer combinations of the vectors of $(k'\cdot \ul{b}^0)$. Thus we can do this in the case
$\ul{d}= k\cdot \ul{a}^1$ to obtain matching isotopy paths, in the sense that the terminal and initial periodic structure bases on $\N_1$ agree. Composing these paths we obtain the desired isotopy equivalence between $\N_0$ and $\N_2$.
\end{proof}

\subsection{Isotopy equivalence for linear graph knots}\label{ss:deformation_equiv_knots}

In the next definition we formally define two linear graph knots on the flat torus to be \emph{linearly isotopic} if there is a continuous path of linear graph knots between them. It follows that if the linear graph knots $\lgk(\N_1,\ul{a})$ and $\lgk(\N_2,\ul{b})$ are linearly isotopic then, by simple periodic extension, the nets $\N_1$ and $\N_2$ are periodically  isotopic. Also we see in Proposition \ref{p:deformequivalence} a form of converse, namely that if
$\N_1$ and $\N_2$ are periodically  isotopic then they have graph knots, associated with \emph{some} choice of periodic structures, which are linearly isotopic.

On the other hand note that a linear 3-periodic net $\N$ in $\bR^3$ with the standard periodicity basis $\{b\}$  is periodically isotopic to its image $\N'$ under an isometric map which cyclically permutes the coordinate axes. 
This is because there is a continuous path of rotation maps of $\bR^3$ from the identity map to the cyclic rotation and restricting these maps to $|\N|$ provides maps $(f_t)$ for a periodic isotopy. 
 While the associated graph knots $K= \lgk(\N,\ul{b})$ and $K'=\lgk(\N',\ul{b'})$, considered as knots in the \emph{same} 3-torus, are  \emph{homeomorphic} (under a cyclic automorphism of the 3-torus which maps one graph knot to the other) they need \emph{not} be linearly isotopic. This follows since linear isotopy within a fixed 3-torus must preserve the homology classes of cycles and yet $K$ may contain a directed cycle of edges with a homology class in $H_1(\bT^3;\bZ)=\bZ^3$ which do not appear as a homology class of any cycle of edges in $K'$. 

In view of this, in the next formal definition we also give weaker forms of linear isotopy equivalence which can be considered as linear isotopy up to rotations and linear isotopy up to {affine automorphisms}.

Let $X\in GL^+(3;\bZ)$. Then there is an induced homeomorphism of the flat 3-torus which we denote as $X_\pi$. This is \emph{affine} in the sense that flat torus line segments map to flat torus line segments.

\begin{defn}\label{d:knotisotopies}
Let $K_0=(N_0, S_0)$ and  $K_1= (N_1, S_1)$ be linear graph knots on the flat torus $\bT^3= [0,1)^3$. 

(i) $K_0$ and $ K_1$ are \emph{linearly isotopic}  if there are linear graph knots $K_t= (N_t, S_t)$,
 for $0<t<1$, and bijective continuous functions $f_t:|K_0|\to |K_t|$ such that, {$f_0$ is the identity map on $K_0$}, $f_t(N_0)=N_t, f_t(S_0)=S_t$, and the paths $t \to f_t(p),$ for  $p\in K_0$ and $0\leq t \leq 1$, are continuous. 
 
(ii) $K_0$ and $ K_1$ are \emph{rotationally linearly isotopic} if for some rotation automorphism $X_\pi$, with $X$ a rotation in $ GL^+(3;\bZ)$, the graph knots $K_1$ and $X_\pi K_2$ are linearly isotopic.

(ii) $K_0$ and $ K_1$ are \emph{globally linearly isotopic} if for some affine automorphism $X_\pi$, with $X\in GL^+(3;\bZ)$, the graph knots $K_1$ and $X_\pi K_2$ are linearly isotopic.
\end{defn}


\subsection{Enumerating linear graph knots and embedded nets}
We can indicate  a linear graph knot $K$ on the flat 3-torus by the triple $(Q,h,p)$, where $(Q,h)$ is a labelled directed quotient graph and $p=(x_1, \dots , x_n)$ denotes the positions of its $n$ vertices in the flat 3-torus $\bT^3= [0,1)^3$. We may also define general \emph{placements} of $K$, or of $(Q, h)$, as triples $(Q,h,p')$ associated with points $p'$ in the $n$-fold direct product $(\bT^3)^n$. Such placements  either correspond to proper linear graph knots with the same labelled quotient graph, or  are what we shall call \emph{singular placements}, for which the nodes $x_i'$ of $p'$ may coincide, or where some pairs of line segment bonds determined by $(Q,h)$ and $p'$ are not essentially disjoint. 

The general placements of $K$ are thus parametrised by the points $x'$ of the flat manifold
$\bT^{3n}= [0,1)^{3n}$, and this manifold is the disjoint union of the set $\K(Q,h)$ of proper placements and the set $\S(Q,h)$ of singular placements.

\begin{thm}\label{t:finiteness}
There are finitely many linear isotopy classes of linear graph knots in the flat torus $\bT^3$ with a given labelled quotient graph.
\end{thm}

The following short but deep proof echoes a proof used by Randell \cite{ran} in connection with invariants for finite piecewise linear knots in $\bR^3$. However we remark that
an alternative more intuitive proof of this general finiteness theorem could be based on the fact that the isotopy classes of  the linear graph knots can be labelled by finitely many crossing diagrams (appropriate to the 3-torus). Also direct arguments are available to show such finiteness for labelled quotient graphs with 1 or 2 vertices.

\begin{proof}
The set $\S(Q,h)$ is a closed semialgebraic set, defined by a set of polynomials and inequalities. The open set set $\K(Q,h)$ is equal to $\bT^{3n}\backslash \S(Q,h)$. Since this set is the difference of two algebraic sets it follows from the structure of real algebraic varieties \cite{whi} that the number of  connected components of $\K(Q,h)$ is finite.
\end{proof}


The theorem implies that the isotopy classes of linear graph knots are countable, since labelled quotient graphs are countable, and so in principal these classes may be listed
by various schemes. For example, for each $n$ there are finitely
many labelled quotient graphs of depth 1 which have $n$ vertices and so finitely many linear isotopy classes of linear graphs knots with $n$ vertices.

The corollary of the next elementary proposition gives a similar finiteness for the periodic isotopy classes of embedded periodic nets.

\begin{prop}\label{p:deformequivalence}
Let $\N$ and $\N'$ be linear 3-periodic nets in $\bR^3$.

 Then the following are equivalent.
\\
\indent (i) $\N$ and $\N'$ are periodically isotopy equivalent.

(ii) There are right-handed periodicity bases $\ul{a}$
and $\ul{a'}$ for $\N$ and $\N'$ such that the linear graph knots $\lgk(\N,\ul{a})$ and $\lgk(\N',\ul{a'})$ are linearly isotopic. 
\end{prop}

\begin{proof}
Suppose that  (i) holds. Let $\N_0=\N$ and $\N_1=\N'$ and assume the equivalence is implemented, as in the definition of periodic isotopy, by a path of intermediate nets $\N_t$ together with
(a)  a continuous path of bases $t \to \ul{a}^t$, $0\leq t\leq1$, where  $\ul{a}^t$ is a periodicity basis for $\N_t$, and (b) bijective functions
$f_t$ from the set of nodes of $\N_0$ to the set of nodes of $\N_t$. The functions $f_t$ necessarily respect the periodic structure.
Let $\ul{a}= \ul{a}^0$ and 
$\ul{a'}= \ul{a}^1$. It follows that the resulting path
$t \to \lgk(\N_t,\ul{a}^t)$  an isotopy between
$\lgk(\N,\ul{a})$ and $\lgk(\N',\ul{a'})$.

Suppose that (ii) holds, with
$K=\lgk(\N,\ul{a})$ and $K'=\lgk(\N',\ul{a'})$. A linear isotopy equivalence $(f_t)$ between $K$ and $K'$ extends uniquely, by periodic extension, to a periodic isotopy equivalence between $\N$ and $\N'$.
\end{proof}

\begin{cor}\label{c:finitelymany} Let $(H,\lambda)$ be a labelled quotient graph. Then there are finitely many periodic isotopy classes of linear 3-periodic nets $\N$ which have the labelled quotient graph $(H,\lambda)$ with respect to some periodicity basis.
\end{cor}

In future work it will be of interest to focus on individual topologies and to determine the finitely many periodic isotopy classes of depth 1. Of particular interest are those with some sense of maximal symmetry over their periodic isotopy class. In fact we formalise this idea in  Section \ref{ss:maxsymmetryMulticomponent} in connection with homogeneous multicomponent nets. 

We now note two basic examples of connected  self-entangled nets,  which we regard as \emph{periodic isotopes} of their maximal symmetry embedded nets.

\begin{eg}\label{eg:self_entgld_dia} \emph{Self-entangled diamond.}
The multi-node fragment in Figure \ref{f:selfentdiaBig} shows part of an embedded net, say $\N$, whose topology is {\bf dia}. 
That $\N$ and the usual maximum symmetry net $\N_{\rm dia}$ are not periodically isotopic  follows from an examination of the catenation of cycles.
 Specifically the diagram shows that $\N$  has 2 disjoint 6-cycles of edges which are linked. This property does not hold for $\N_{\rm dia}$ and so they cannot be periodically isotopic. 
\end{eg}

\begin{center}
\begin{figure}[ht]
\centering
\includegraphics[width=5cm]{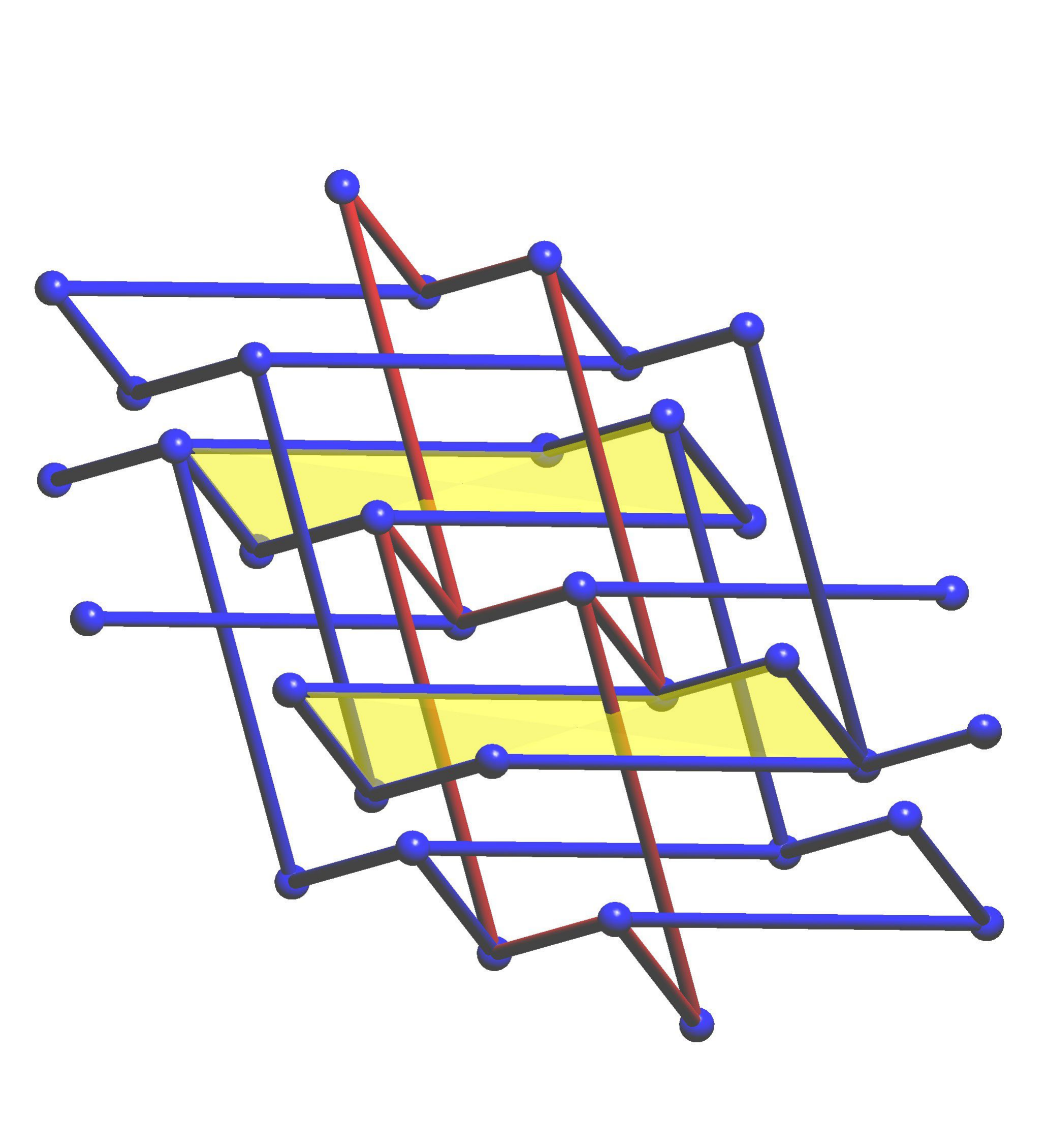}
\caption{Catenated 6-cycles in a self-entangled embedding of {\bf dia}.}
\label{f:selfentdiaBig}
\end{figure}
\end{center}

\begin{eg}\label{eg:self_entgld_cds}
{\emph{Self-entangled embeddings of} {\bf cds}.}
The maximal symmetry net $\N_{\rm cds}$ (associated with cadmium sulphate) has an underlying periodic net ${\bf cds}$  with quotient graph $H(1,2,1)$. 
The left hand diagram of Figure \ref{f:cds} indicates a linear graph knot for {\bf cds} and  the 3-periodic extension of this diagram defines a model embedded net which is periodically isotopic to $\N_{\rm cds}$.
To be precise,  define this net as the model net $\M(p_1, p_2)$ with
$p_1=(1/2,1/2,1/2), p_2=(1/2,1/4,1/2)$ and with labelled quotient graph  LQG$(\N,\ul{b}) = \H=(H(1,2,1),\lambda)$ where
$\lambda$ assigns the labels, $(0,0,1)$ to the loop edge associated with $p_1$,
$(1,0,0)$ to loop for $p_2$ and the labels $(0,0,0)$ and $(0,1,0)$ to the 2 remaining edges.
\begin{center}
\begin{figure}[ht]
\centering
\includegraphics[width=4cm]{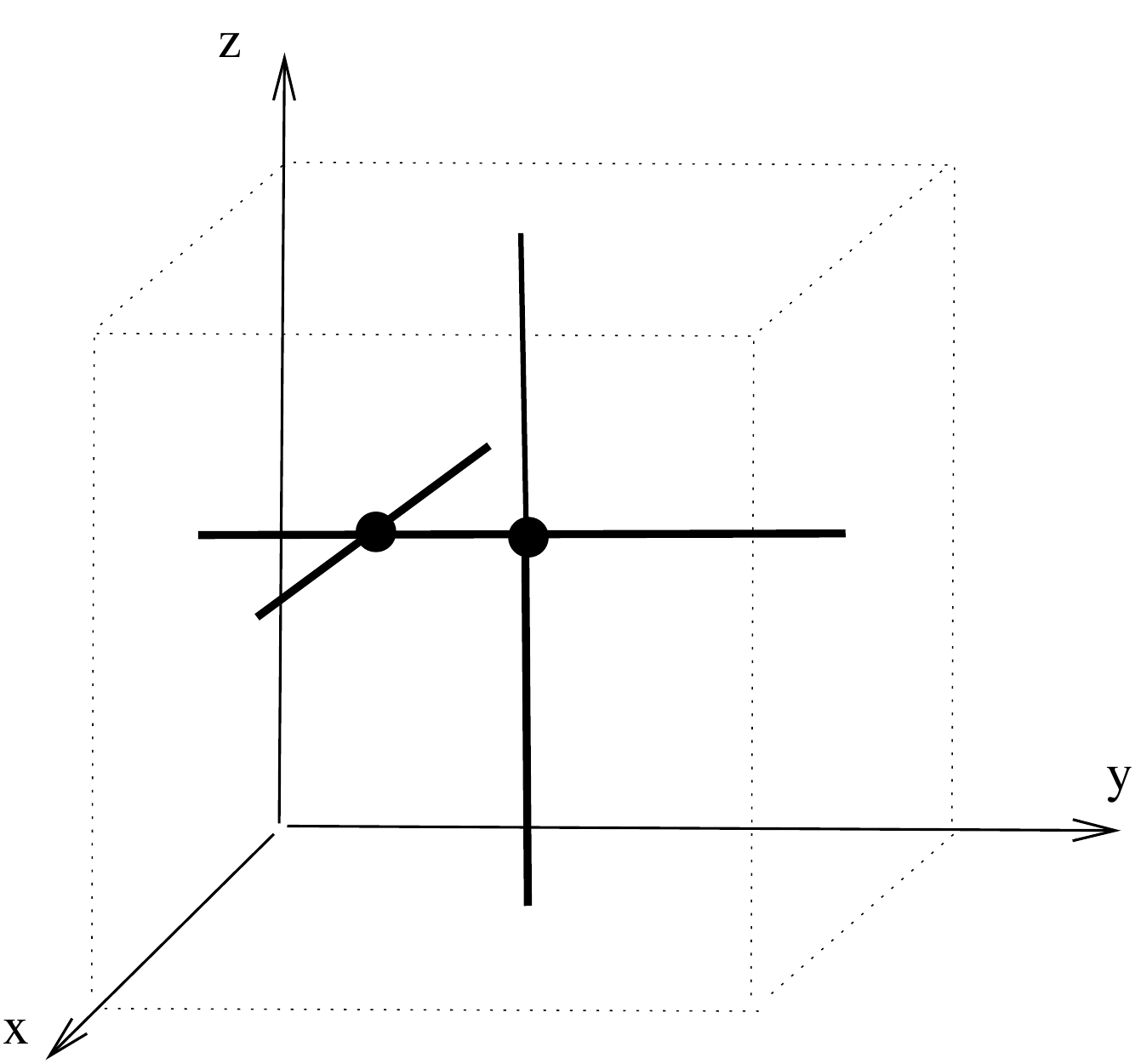}
\quad \quad 
\includegraphics[width=4cm]{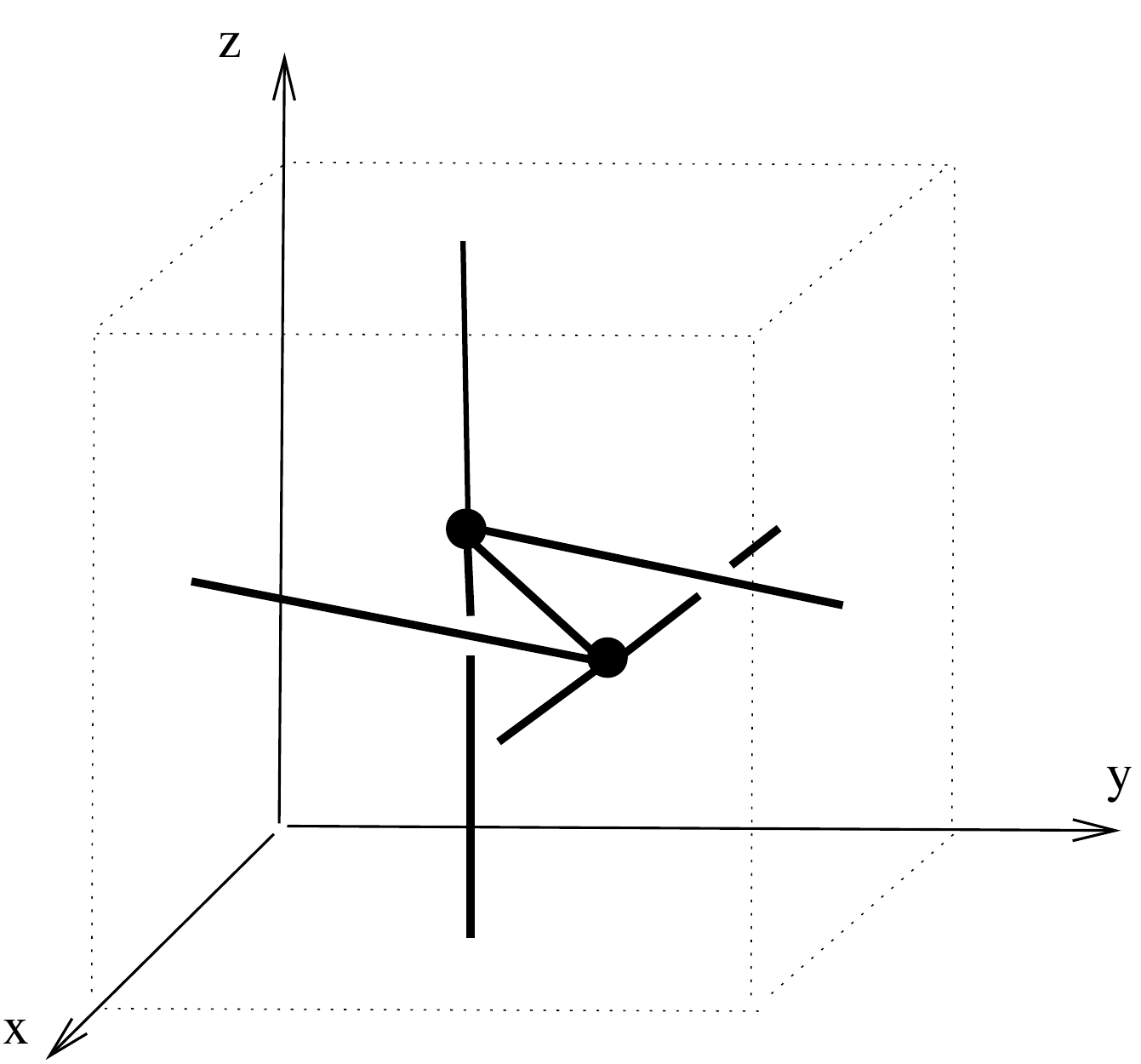}
\quad 
\caption{Linear graph knots for distinct periodic isotopes of {\bf cds}} 
\label{f:cds}
\end{figure}
\end{center}

In the manner of Example \ref{s:mainexample}  let us now view  $p_2$ as variable point $p_2'=(x,y,z)$ within the semiopen cube $[0,1)^3$. The positions of $p_2'$ together with the labelled quotient graph, define model nets 
as long as there are no edge crossings.
Let $\O$ be the set of these positions for $p_2'$. Then, viewed as a subset of $[0,1)^3$ (\emph{not} as a subset of the flat 3-torus) the set $\O$ decomposes as the union of 5 path-connected components $\O_1, \dots , \O_5$. The set $\O_1$ is the subset of $\O$ with $y<1/2$, the set $\O_2$ is the subset with $y>1/2, x>1/2, z<1/2$ (the right hand figure of Figure \ref{f:cds} corresponds to a point in $\O_2$), and  $\O_3$ is the subset with 
$y>1/2, x<1/2, z<1/2$. The sets $\O_4, \O_5$ are similarly defined to $\O_2, \O_3$, respectively, except that $z$ is greater than $1/2$.  

Let $\M_1,\dots , \M_5$ be representatives for the 5 path-connected components. The net $\M_1=\M(p_1, p_2)$ is a model net for $\N_{\rm cds}$ while the net $\M_2$ is a periodic isotope.
This can be seen once again by exhibiting different catenation properties. Specifically,
$\M_{\rm cds}^\alpha$ has a 6-cycle of edges which is linked to (penetrated by) an infinite linear subnet, while
$\M_{\rm cds}$ does not have such catenation.



\end{eg}


 

 

\subsection{Entangled nets, knottedness and isotopies.}
\label{ss:knottednessand2formsisotopy} 
The examples above concern \emph{connected} self-entangled nets and their connected graph knots on the 3-torus and there is a natural intuitive sense in which such nets can be "increasingly knotted" by moving through homotopies to embeddings with an increasing number of edge crossings.  However the linear graph knot association is also a helpful perspective for multicomponent nets whose components are \emph{not} self-entangled so may be  equal to, or perhaps merely isotopic to, their individual maximal symmetry embeddings. In this case there are intriguing possibilities for the nesting of such "unknotted components" and their associated space groups. We address this topic in Sections \ref{s:groupmethods} and \ref{s:entanglednets} as well as the attendant crystallographic issue of formulating a notion of maximal symmetry for such multi-component nets. 


For completeness we note two further forms of isotopy equivalence which will not be of concern to us.

\emph{(i) Relaxed periodic isotopy.}\label{r:per_def_equiv} The notion of periodic isotopy equivalence in Definition \ref{d:deformationequivalentnets} can be weakened in a number of ways.   One less strict form, which one could call \emph{relaxed periodic isotopy}, omits the condition (a), requiring periodic basis continuity, and so allows a general continuous path of intermediate (noncrossing) periodic nets $\N_t$. 
Since the continuity requirement in (b) of the node path functions $(f_t)$ is one of point-wise continuity on the set of nodes $\N_0$ and, moreover, the ambient space is not compact, it follows that such  paths of periodic embedded nets can connect embedded nets that are not periodically isotopic. In particular, one can construct
relaxed periodic isotopies which untwist infinitely twisted components.
 

\emph{(ii) Ambient isotopy.} The usual definition of ambient isotopy for a  pair $K_1, K_2$ of knots (or links) in $\bR^3$ requires the existence of a path $h_t$ of homeomorphisms of $\bR^3$ (the ambient space) such that $h_0$ is the identity map and $h_1(K_1)=K_2$. Here, for $x\in \bR^3$, we have $h_t(x) = h(t,x)$ where $h:[0,1]\times \bR^3\to \bR^3$ is a continuous function. Also, the closed sets $K_t=h(t,K_0)$, for $0\leq t\leq1$, form a path of knots (or links) between $K_0$ and $K_1$.

One may similarly define \emph{ambient isotopy} for embedded periodic nets \cite{del-oke-nets}. In this case the intermediate closed sets $L_t$, defined by $L_t = h_t(|\N_0|)$ are the bodies of general string-node nets $\L_t$.
It is natural to impose the further condition that these sets are the bodies of (proper) linear $3$-periodic nets, and this then gives a definition of what might be termed \emph{locally periodic ambient isotopy}. In this case the set of restriction maps $f_t = h_t|_{|\N_0|}$ define a relaxed periodic isotopy.

\medskip

\section{Group methods and maximal symmetry isotopes}\label{s:groupmethods}

We now give some useful group theoretic perspectives for multicomponent frameworks, starting with the general group-supergroup construction in Baburin  \cite{bab} for transitive nets.
This method underlies various algorithms for construction and enumeration. In this direction we also define maximal symmetry periodic isotopes in terms of extremal group-supergroup indices of the components. Finally, turning towards generically, or randomly, nested components, we indicate the role of Burnside's lemma in counting \emph{all} periodic isotopes for classes of shift-homogeneous nets.


\subsection{Group-supergroup constructions 
}\label{ss:group-supergroup}

{

Let $\N = \N_1 \cup \dots \cup \N_n$ be a linear 3-periodic net which is a disjoint union of connected linear 3-periodic nets in $\bR^3$. Let $G$  be the space group of $\N$ and  assume that it acts transitively on the $n$ components of $\N$. Thus $\N$ is a transitively homogeneous net, or is of transitive type.

Let $g_1 = id$, the identity element of $G$, and note that for each $i=2,\dots ,n$ there is an element $g_i\in G$ with $g_i\N_1=\N_i$. Also, let $H_i\subset G$ be the subgroup of elements $g$ with $g\cdot \N_i = \N_i$, for  $i=1,\dots , n$.


\begin{lem}
The cosets of $H_1$ in $G$ are $g_1H_1, \dots , g_nH_1$.
\end{lem}

\begin{proof}
The cosets $g_iH_1$ are distinct, since their elements map $\N_1$ to the distinct subnets $\N_i$. On the other hand if $g\in G$ then $g\N_1=\N_j$ for some $j$ and so $g_j^{-1}g\N_1=\N_1,  g_j^{-1}g \in H_1, g\in g_jH$ and $gH_1=g_jH_1.$
\end{proof}

Write $\stab(v; G)$ to indicate the subgroup of $G$ which fixes the node $v$ of $\N$ and similarly define the stabiliser group of an edge $e$ of $\N$.

\begin{lem}
Let $v$ (resp. $e$) be a node (resp. edge) of $\N_i$ for some $i$. Then
\[
\stab(v; G)= \stab(v;H_i), \quad \quad \stab(e; G)= \stab(e;H_i).
\]
\end{lem}

\begin{proof}
It suffices to show that if $g$ fixes an element (vertex or edge) of $\N_i$ then $g\N_i=\N_i$. Observe that $\N_i$ is the maximal connected subnet of $\N$ containing the element. Also, for any subnet $\M$ the image $g\cdot \M$ is connected if and only if $\M$ is connected, and so the lemma follows.
\end{proof}

These lemmas feature in the proof of the following theorem \cite{bab}.

\begin{thm}
If $g\in G$ is a mirror element then $g \in H_1$.
\end{thm}

The significance of this result  is that it shows that the construction of a transitive type entangled net $\N$  with a connected component $\M$ requires the space group $\fS(\N)$ to be free of mirror symmetries which are not in $\fS(\M)$.
 In fact this necessary condition is frequently a \emph{sufficient} condition and this leads to effective constructions of novel entangled nets where these nets have components $g_i\M$ with multiplicity equal to the index of $\fS(\M)$ in $\fS(\N)$.

\subsection{Maximal symmetry periodic isotopes}
\label{ss:maxsymmetryMulticomponent}
 Let $\N$ be a multicomponent embedded 3-periodic net in $\bR^3$ with space group $\fS(\N)$ and let $\N_1, \N_2,  ... , \N_i, ... , \N_n$ be representatives of the equivalence classes of the components of $\N$ for the translation subgroup of $\fS(\N)$. Also, as in the previous section, let $H_i$ be the \emph {setwise} stabiliser of $\N_i$ in $\fS(\N)$. Regarding $H_i$  as a subgroup of $Aut(G(\N_i))$ (cf. Section \ref{ss:maxsymmetryRCSR}) we may compute the indices $m_i = |Aut(G(\N_i)) : H_i|$. Here we restrict our scope to crystallographic nets $\N_i$ (Klee, \cite{kle}) and therefore the indices are always finite. These indices evidently coincide when $\fS(\N)$ is transitive on the components of $\N$ and this is our primary focus.

We say that $\fS(\N)$ is a \emph{maximal symmetry space group for the periodic isotopy class of $\N$} if the nondecreasing rearrangement $m(\fS(\N))$ of $m_1, \dots  , m_n$, which we call the \emph{multi-index} of $\fS(\N)$, is minimal for the lexicographic order when taken over all groups $\fS(\N')$ where $\N'$ is periodically isotopic to $\N$. In this case we refer to $\N$ as a \emph{maximal symmetry periodic isotope} and we write  $\fS_{\rm max}(\N)$ for this space group, noting that $\fS_{\rm max}(-)$ is only defined for such minimal multi-index embedded nets. We note that a maximal symmetry proper embedding of a multicomponent net need not be unique, as might be already the case for (connected) single component nets (cf. Section \ref{ss:maxsymmetryRCSR}).

In the same way one may define  maximal symmetry groups for periodic homotopy and one may consider other equivalence relations 
depending on the matter at hand but these issues shall not concern us here.

 We note that a maximal symmetry embedding for periodic isotopy is related to the concept of an \emph {ideal geometry} of a knot (Evans, Robins and  Hyde, \cite{eva-et-al-idealgeometry}, and references therein) that is required to minimize some energy function. However, as well as a certain arbitrariness in the choice of energy function and the possibility of overlooking a global minimum, the result of optimization depends on the imposed periodic boundary conditions. Thus the determination a maximal translational symmetry embeddings remains problematic in the search for an ideal geometry of a multicomponent periodic net. In contrast, our definition, being essentially group-theoretic, aims to capture isotopically intrinsic properties of embedded nets which are independent of such constraints.


Maximizing the  symmetry of interpenetrated {embedded nets} is important for a number of reasons, e.g. to characterize their transitivity properties and to {derive} possible distortions which might occur in a crystal structure {by examining} group-subgroup relations. Furthermore, the knowledge of a maximal symmetry can be used to explicitly construct a deformation path that relates an embedding  with maximal symmetry to a distorted embedding $\N'$ with higher multi-index. A periodic homotopy path can be constructed  relative to a common subgroup of { $\fS_{max}(\N)$ and $\fS(\N')$}, for example, by interpolating between coordinates, 
and this path is often crossing free and so a periodic isotopy. 

Determination of maximal symmetry is a highly non-trivial task. The only general approach to the problem was proposed in \cite{bab} that is based on subgroup relations between automorphism groups of connected components and a respective Hopf ring net.  Along these lines  maximal symmetry embeddings and their symmetry groups have been determined for $n$-grids as in Section \ref{ss:classesof_npcu}.  

\subsection{Counting periodic isotopy classes by counting orbits}
\label{ss:burnsidemethod}
Let us now consider translationally transitive embedded nets $\N$  with $n$ components which are randomly arranged, in the sense that we make no further assumptions. We are interested in calculating the number of periodic isotopy classes for a given topology. In the next section we solve this problem for $n$-fold {\bf pcu} by reducing the counting to a combinatorial calculation, namely to a calculation of the number of orbits of a finite set of "normalised" $n$-pcu nets under the action of a finite group of isometries, where the finite group is generated by cube rotations and shifts. 

The method is generally applicable but for a translationally transitive $n$-pcu embedding a normalisation of $\N$ takes a particularly natural form in which the components have integral coordinates. 
While a normalised net is not uniquely associated with $\N$ it turns out that their multiplicity corresponds to the cardinality of an orbit under the finite group action, and so counting the number of orbits gives the count we seek. A standard formula for counting such orbits is given by Burnside's Lemma which states the following. 
Let $G$ be a finite group acting on a finite set $X$ with group action $x \to g\cdot x$, so that the orbit of an element $z\in X$ is the set $\{g\cdot z:g \in G\}$.  Then the number of distinct orbits is given by
\[
\frac{1}{|G|}\sum_{g\in G} |X_g|,
\]
where $X_g$ denote the set of points $x$ with $g\cdot x=x$. In this way the problem is reduced to counting, for each symmetry element $g$, the number of normalised nets which have this symmetry.

\section{Classifying multicomponent entangled nets}\label{s:entanglednets}
We next determine the number of periodic isotopy types of various families of embedded nets $\N$ whose components are embeddings of the net {\bf pcu}. 
The simplest family here consists of those nets $\N$ with $n$ parallel components, each being a shifted copy of the model net $\M_{\rm pcu}$. 
In this case we refer to $\N$ as a \emph{multigrid} or  \emph{$n$-grid}. 
Such  nets have  dimension type $\{3;3\}$ and are shift-homogeneous. 

For practical purposes, both in this section and in Section \ref{s:latticenets}, we focus on the following hierarchy of 4 equivalence relations for embedded nets:
\medskip

1. Nets $\N_0, \N_1$ in $\fN$ are \emph{affinely equivalent} (resp. \emph{orientedly affine equivalent}) if they have translations $\N_0', \N_1'$
with
$\N_0'= X\N_1'$ for some invertible $3 \times  3$ matrix $X$
(resp. with $\det X>0$).
\medskip

2. The pairs $(\N_0, \ul{a}^0), (\N_1, \ul{a}^1)$, with given periodicity bases, are \emph{strictly periodically isotopic} if there is a continuous path of embedded nets $\N_t$ in $\fN$ with an associated continuous path of periodicity bases from $\ul{a}^0$ to $\ul{a}^1$.
\medskip

3.  $\N_0, \N_1$ are \emph{periodically isotopic} if  they have strictly periodically isotopic pairs for some choice of periodicity bases $\ul{a}^0, \ul{a}^1$.
\medskip

4. $\N_0, \N_1$ are \emph{topologically isomorphic}, or have the same topology, if their structure graphs (underlying nets)
are isomorphic as countable graphs.}
\medskip

\subsection{Translation-transitive $n$-grids}
We first consider embeddings of { $n$-grids} with a strong form of  homogeneity. Specifically
we give group-supergroup methods which  determine the periodic isotopy types of \emph{translation-transitive} { $n$-grids.} 

Considering the  translation transitivity assumption it follows that the shift vectors relating parallel copies of a single component grid  are in fact coset representatives of some lattice with respect to the sublattice generated by the standard periodicity basis of a connected component. The number of cosets is equal to the index of a sublattice. This observation gives a recipe for generating translation-transitive $n$-grids, by enumerating superlattices of index $n$ for the
lattice of a connected component
while discarding the associated $n$-grids which  fail to be noncrossing.

A determination of index $n$ superlattices can be made with the following lemma. See also \cite{cas}, \cite{dav-et-al}.

\begin{lem}\label{l:arithmetic}
Let $n$ have a factorisation $n=p_1p_2p_3$, with $1\leq p_i\leq n,$ and let

\[
  L= 
  \left[ {\begin{array}{ccc}
   p_1 & 0 & 0\\
   q_1 & p_2 & 0\\
   r_1 & q_2 & p_3\\
  \end{array} } \right]
\]
be a matrix with integral coefficients satisfying $0\leq q_1< p_2$, $0\leq  q_2 < p_3$ and $0\leq  r_1 < p_3$. The rows of the inverse matrix $L^{-1}$ generate a superlattice of $\bZ^3$ of index $n$. Moreover, every superlattice of $\bZ^3$ of index $n$ has such a representation.
\end{lem}

A computational determination of the number,  $\beta_{\rm tt}(n)$, of periodic isotopy types, can now be implemented with the following 3-step algorithm. Some of the values are recorded in the summary Table 1.

\medskip

(i) Using the lemma, generate all superlattices with index $n$.

(ii) Discard such a superlattice if its corresponding $n$-grid  has edge crossings.

(iii) Reduce the resulting list to a (maximal) set of superlattices which are pairwise inequivalent under the point group of a primitive cubic lattice.

\medskip

We have indicated that this (practical) 3-step generation-and-reduction algorithm gives the number of  congruence classes of  translationally transitive $n$-grids. That this number also agrees with the (a priori smaller) number of periodic isotopy classes  (up to chirality) is essentially a technical issue.
 This follows from
Theorem \ref{t:multigridcount} (iii) and Appendix A. Moreover, for the same reason the algorithm determines exactly the translationally transitive $n$-grids which are  maximal symmetry periodic isotopes.
 
 We remark that a similar 3-step algorithm can be applied in the case of translationally transitive embeddings of e.g. $n$-fold {\bf dia}, $n$-fold {\bf srs} and other nets.  We conjecture that if connected components are crystallographic nets in their maximal symmetry configurations, then step (iii) leads directly to the classification into periodic isotopy classes.

\bigskip

\subsection{{A combinatorial enumeration of} n-grids}\label{ss:n-foldpcu} 
We now consider the wide class of general multigrids, with no further symmetry assumptions.
The combinatorial objects relevant to periodic isotopy type counting are given in terms of various finite groups acting on finite sets of patterns which we now define.

Let $T= \{1,\dots ,n\}^3,$ viewed as a discrete 3-torus, and 
let $C_n$ be the cyclic group of order $n$. In particular $C_n$ can act on $T$ by cyclically permuting one of the 3 coordinates.
Also, let $R$ be the rotation symmetry group of the  unit cube $[0,1]^3$. Then $R$ acts on the discrete torus $T$ in the natural way. 

Let $\X(n)$ be the finite set of $n$-tuples, or \emph{patterns}, $\{p_1,\dots , p_n\}$ where the points lie in $T$ and have \emph{distinct coordinates}, so that for all pairs $p_i, p_j$ the difference $p_i-p_j$ has nonzero coordinates. In particular $|\X(n)|= (n!)^2.$  These $n$-tuples in fact correspond to the coordinates of the nodes appearing in a unit cell of the $n$ components of a normalised $n$-grid.

Finally, for $n\in \bN$, let $\rho(n), \alpha(n), \beta(n)$, respectively, be the number of orbits in $\X$ under the natural action of the groups 
\[
R, \quad
C_n\times C_n \times C_n, \quad 
C_n\times C_n \times C_n \times R.
\]

Recall that a linear graph knot for an embedded net $\N$ is determined by a choice of periodicity basis $\ul{a}$ and is denoted $\lgk(\N,\ul{a})$. In the case of an $n$-grid $\N$ with its standard periodicity basis $\ul{b}$ we refer to $\lgk(\N,\ul{b})$  as the standard linear graph knot for $\N$. Evidently
$\lgk(\N,\ul{b})$ appears as the union of $n$ disjoint translates in the flat 3-torus of
$K_{\rm pcu}$.

\begin{thm}\label{t:multigridcount} (i) The number of linear isotopy types of standard linear graph knots of $n$-grids is  $\alpha(n)$.

(ii) The number of  rotational isotopy types of standard linear graph knots of $n$-grids is  $\beta(n)$. 

(iii) The number of periodic isotopy 
classes of  $n$-grids is  $\beta(n)$.
\end{thm}

The proof of this theorem is given in Appendix A. The essential argument involves a \emph{discretisation} in which, in (ii) for example, the components are separately shifted by a (joint) isotopy to an evenly spaced position. Then $n$ nodes in a unit cell corresponds to a pattern of $n$ coordinate distinct points in the discrete torus $\{1,2,\dots ,n\}^3$. 
 Additionally, for (iii) one must resolve the technical problem in Remark \ref{r:conjecture} in the case of $n$-grids and show that the triple cyclic order of coordinates (modulo the rotation group $R$) is indeed a periodic isotopy invariant.  We do this in Lemma \ref{l:amplifiedgrids}, and the equivalence given in Proposition \ref{p:deformequivalence} is a helpful step in the proof. We also note that the periodic isotopy that one needs to construct in the proof, when the cyclic orders coincide modulo  $R$, is simply a concatenation of a periodic isotopy of local component translations (to achieve equal spacing),  followed by a an elementary periodic isotopy induced by a path of affine motions corresponding to a (bulk) rotation and final translation. 



\subsection{Translational isotopy and framed $n$-grids} 
The general formulation of periodic isotopy of necessity entails some technical complexity in the proofs. We now note two restricted but natural $n$-grid contexts where the determination the number of equivalence classes simplifies. We omit the formal proofs. In the first of these we define a more restricted form of isotopy while in the second context we distinguish, or colour, one of the component grids. 

Let us say that a multi-grid is 
 \emph{aligned} if 
its components $\M_i$ are translates of the model net $\M_{\rm pcu}$ with node set $\bZ^3$. 

\begin{defn}Two
aligned $n$-grids $\M, \M'$ are \emph{translationally isotopic} if for some labelling of the components
there are continuous functions $g_i: [0,1] \to \bR^3$, for $1\leq i \leq n$, with $g_i(0)=0$ for all $i$, such that 
\medskip

(i) for each $t$ the embedded net
\[
\M_t = (\M_1+g_1(t)) \cup \dots \cup (\M_n+g_n(t)) 
\] 
is a (noncrossing) linear 3-periodic net,

(ii) $\M= \M_0$ and  $\M' =\M_1$.
\end{defn}

This simple form of isotopy in fact corresponds to strict periodic isotopy for these nets with respect to the standard periodicity basis. It is a form of ``local" periodic isotopy in the sense that the deformation paths of the nodes are localised in space. In particular deformation paths incorporating bulk rotations are excluded.

\begin{thm} The number of translational isotopy classes of aligned $n$-grids is $\alpha(n)$.
\end{thm}

For the second variation, let us define a \emph{framed $n$-grid} to be an $(n+1)$-grid with a distinguished component, the framing component. Thus  a \emph{framed $n$-grid} is a coloured $n+1$ grid where all but 1 of the components are of the same colour.  
Periodic isotopy for coloured $n$-grids may be defined exactly as before but with the additional requirement that the maps $(f_t)$ respect colour. 

 It is evident that the cube rotation group $R$ acts naturally on such framed $n$-grids. Also, as indicated in our remarks following Theorem \ref{t:multigridcount}, counting periodic isotopy types reduces to counting orbits of patterns $p$ of $n+1$ coordinate-disjoint points, $p = (p_1, \dots , p_{n+1})$, in the discrete torus $\{1, 2,\dots , n+1\}^3$.
However, in view of the colour preservation we may assume, by shifting, that $p_{n+1}$ lies in the $R$-orbit of $(1,\dots ,1)$, and from this it follows (varying the proof of Theorem \ref{t:multigridcount})  that the periodic isotopy classes correspond to the $R$-orbits of the $n$-tuples $(p_1, \dots , p_n)$.




\begin{thm}
The number of periodic isotopy classes of framed $n$-grids is
$\rho(n)$.
\end{thm}

\subsection{Employing Burnside's lemma} 
We can now make use of Burnside's lemma to compute values of $\alpha(n), \beta(n)$ and $\rho(n)$. The following formula readily shows that $\alpha(5) =128, \alpha(7)= 74088$ for example.

\begin{prop}
Let $p$ be a prime number. Then
\[
\alpha(p) =  \frac{1}{p^3}((p!)^2 + (p-1)^3p^2)
\]
\end{prop}

\begin{proof}
Note that a group element $g=abc\neq 000$ in $C_p\times C_p\times C_p$ with $ a$ or $b$ or $c$ equal to $0$ does not fix any pattern under the cyclic action on $T=\{1,2,\dots , p\}^3$ and so $|\X_{abc}|=0$ in this case.
Also, every pattern is fixed by the identity element and so $|\X_{000}|=(p!)^2$. It remains to consider the $(p-1)^3$ group elements $abc$ with none of $a,b,c$ equal to $0$. 

The group element $111$ acts as a diagonal shift  and so any fixed  pattern of nodes in $T$ is determined by the unique node occupying a particular face of $T$. Conversely any of the $p^2$ node locations on this face determines a unique fixed pattern for the action of $111$. Thus  $|\X_{111}|=p^2$. 

Since $p$ is prime the same argument applies to any group element $abc$ with none of $a,b,c$ equal to the identity element $0$, since $a,b,c$ each have order p. There are $(p-1)^3$ such elements $abc$ and so the formula now follows from Burnside's Lemma.
\end{proof}

 The case of composite $n$ is similar. In the case that each of $a,b,c$ have order $r$ where $r$ divides $n$ the size
of the fixed set $\X_g$ for $g=abc$ is the product $n^2(n-r)^2(n-2r)^2\cdots r^2$. All other elements except the identity have no fixed patterns. In this way we obtain $\alpha(4) = 12 $ and $\alpha(6) = 2424$.

Similarly, for the framed $n$-grids one may compute $\rho(2)=1, \rho(3)=4, \rho(4)= 33$. Evidently there is rapid subsequent growth rate since the Burnside Lemma formula quickly leads to the lower bound $\rho(n)\geq (n!)^2/(24\times n^3)$.

\subsection{Classes of embedded $n$-{\bf pcu}}\label{ss:classesof_npcu}

Figure \ref{f:summary} gives examples of small $n$-grids with contrasting transitivity properties.

\begin{center}
\begin{figure}[ht]
\centering
\includegraphics[width=14cm]
{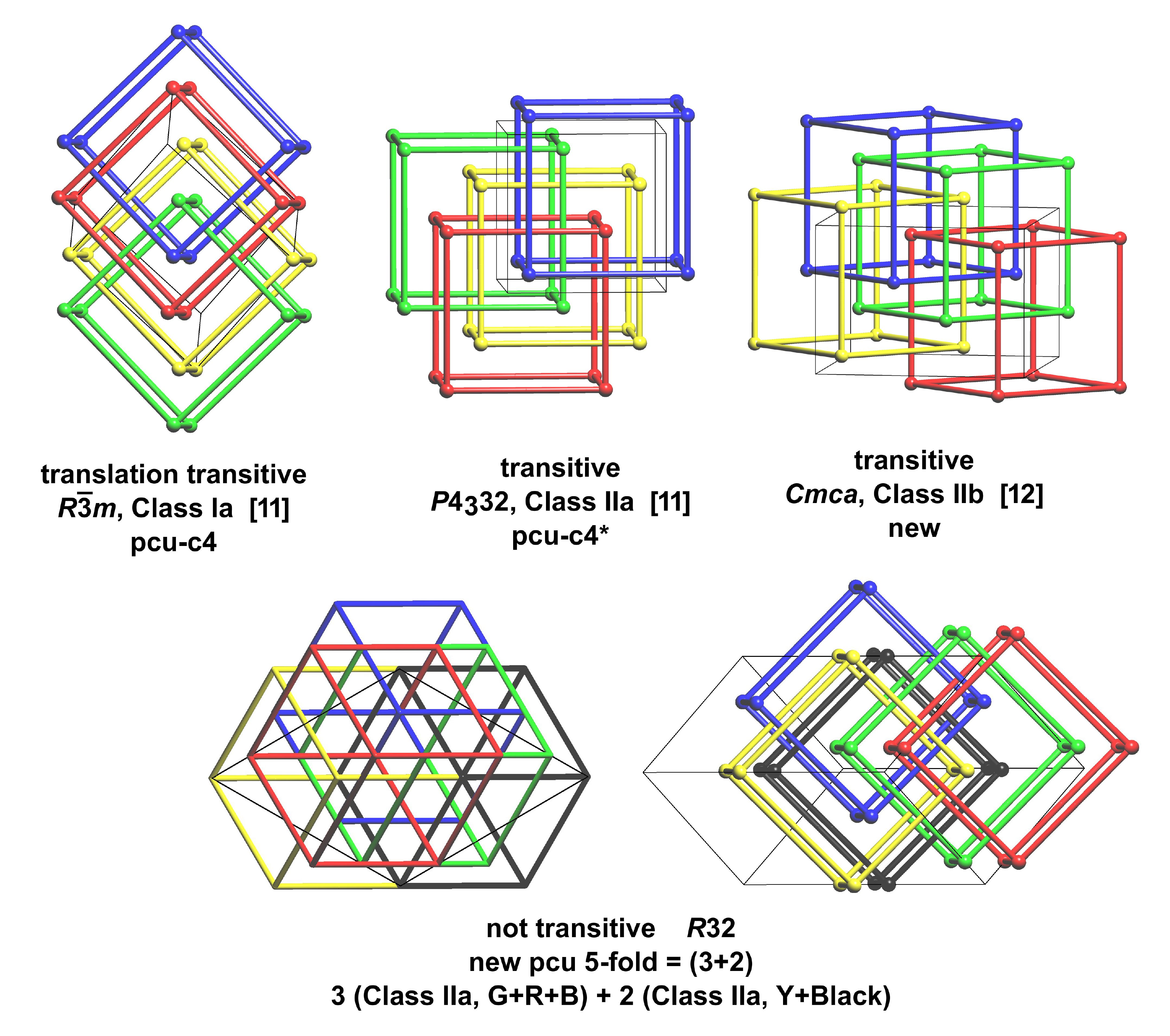}
\caption{Four examples of shift homogenous $n$-grids with interpenetration class according to \cite{bab-bla-car-cia-pro}, the vertex and edge transitivity [ve] and the RCSR names where available. 
}
\label{f:summary}
\end{figure}
\end{center}

In  Table 1 we summarise the number of classes of $n$-grids for various types of $n$-grid and forms of isotopy for some small values of $n$ with
the values of $\alpha(n)$ and $\beta(n) $  obtained via Burnside's Lemma as before. The count $\beta_t(n)$ 
is for \emph{transitive} $n$-grids in the sense given in Section \ref{s:homtype}, and for $n$-grids this coincides with vertex transitivity. The count $\beta_{\rm tt}(n) $ is for \emph{translation-transitive} $n$-grids which have components that are equivalent by translations in the space group. These counts, which coincide if $n$ is prime, are obtained using the group-supergroup algorithm of the Section \ref{ss:maxsymmetryMulticomponent}. 

\bigskip

\begin{center}
\begin{tabular}{|c|c|c|c|c|c|c|c|}
\hline
$n$   & 2  &3 & 4 & 5&6&7&$n$-grids/isotopy \\
\hline
\hline
$\alpha(n)$   & 1  &4 & 12 & 128&2424 &74088& $n$-grids/translational isotopy 
\\
\hline
$\beta(n)$  & 1  &1 & 3 & 9 &89&-& $n$-grids/periodic isotopy\\
\hline
$\beta_t(n)$  & 1  &1 & 3 & 2 &7&4& {transitive $n$-grids}/periodic isotopy\\
\hline
$\beta_{\rm tt}(n)$   & 1  &1 & 1 & 2&1&4&{translation-transitive $n$-grids}/periodic isotopy\\
\hline
\hline
\end{tabular}\\ 
\end{center}\medskip
\begin{center}
Table 1. Counts of isotopy classes of $n$-grids.
\end{center}


\medskip

Figure \ref{f:summary} summarises homogeneity and transitivity types of multicomponent embedded nets.
\medskip

\begin{center}
\begin{figure}[ht]
\centering
\includegraphics[width=16cm]
{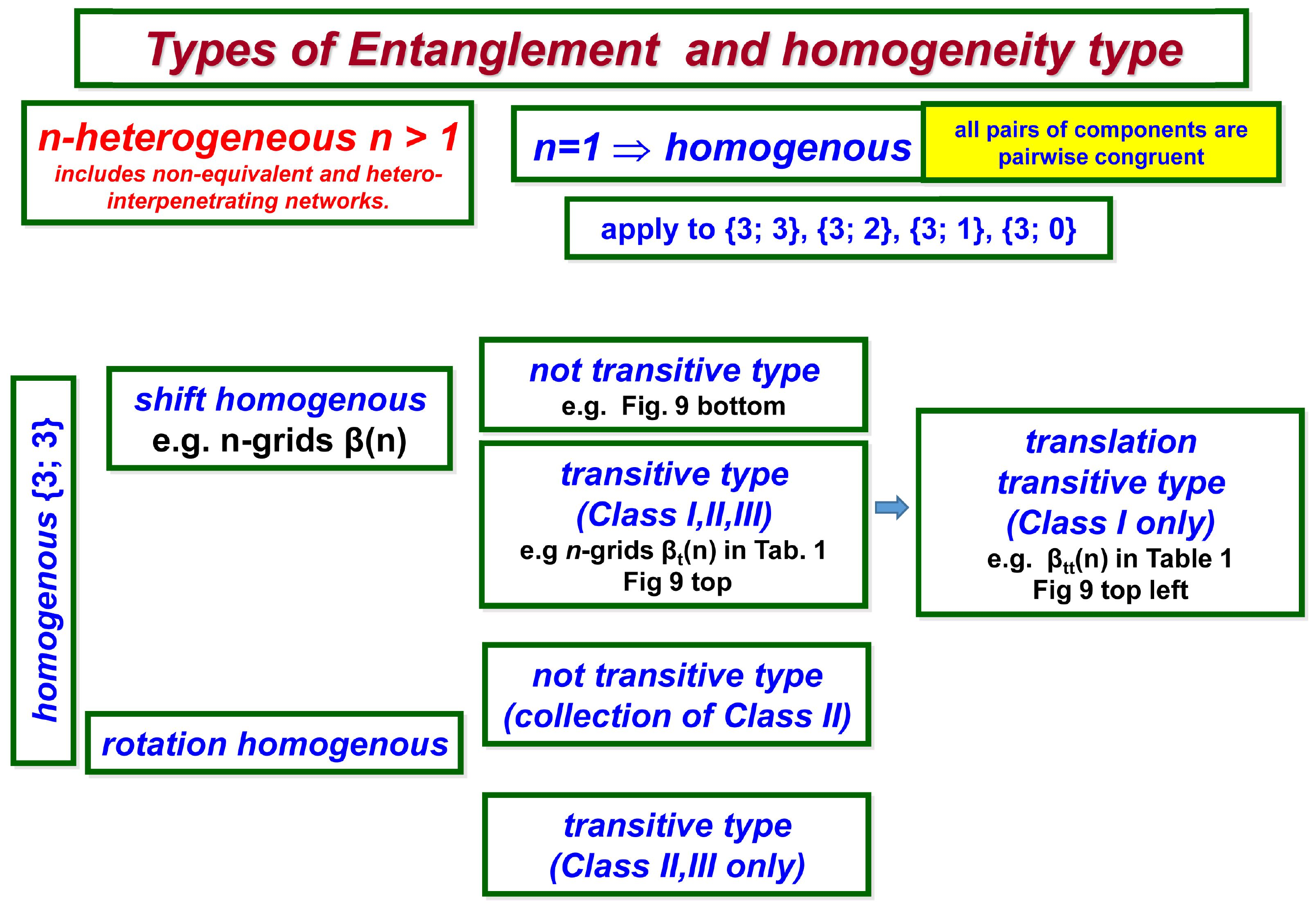}
\caption{
Homogeneity and transitivity types of entangled nets.}
\label{f:summary}
\end{figure}
\end{center}

\section{Classifying lattice nets}\label{s:latticenets}

\subsection{Depth 1 disconnected nets with a single vertex QG}
A model net $\M$ which has adjacency depth 1 with respect to the standard basis $\ul{b}$ is determined by a set $F_e$
of edge representatives $[a,b]$ for the translational orbits of edges. In the case that there is a single orbit for the nodes we may assume that there is a node at the origin and choose the \emph{unique} edge-orbit representative $[a,b]$ such that $(a,b)$ is a subset of the semiopen cube $[0,1)^3$. Such  representative edges are determined up to sign by the vectors $a-b$, or equivalently in this case, by the labels of the depth 1 labelled quotient graph $LQG(\M,\ul{b})$.
We use the following terminology for edges in $F_e$. This will also be useful in subsequent sections.

The 3 axial edges are denoted
$a_x, a_y, a_z$  and   $d_1, \dots , d_4$ denote the 4 diagonal edges which are incident to
$(0,0,0)$, $(1,0,0)$, $(1,1,0)$, $(0,1,0)$ respectively.
 The 3 face diagonal edges which are incident to the origin are denoted $f_x, f_y, f_z$, corresponding to the directions
$(0,1,1)$, $(1,0,1)$, $(1,1,0)$, while the edges $g_x, g_y, g_z$ are the other face diagonals, parallel to the respective vectors $(0,1,-1)$, $(1,0,-1)$, $(1,-1,0)$. Thus we may define any set $F_e$ by means of an ordered subword $w$ of the ordered word
\[
a_xa_ya_zf_{x}f_{y}f_{z}g_xg_y{g_z}d_1d_2d_3d_4
\]

In view of the noncrossing condition it is elementary to see that every model net $\M$ is affinely equivalent, simply by rotations, to a \emph{standardised model net} defined by the \emph{standard ordered word} of the form
$w= w_1w_2d_1$ or $w_1w_2$ , where $w_1$ is either $a_x, a_xa_y, a_xa_ya_z$ or the null word, and $w_2$ is a face subword with 0, 1, 2 or 3 letters, of which there are 27 possibilities.

We now determine the depth 1 embedded bouquet nets that are disconnected, that is, which have more than one and possibly infinitely many connected components. It turns out that there are 6 embedded nets up to affine equivalence and we now give 6 model nets for these types.
\medskip

(i) $\M_{\rm a}$ is the model net determined by $F_e=\{a_x\}$ and consists of parallel copies of a 1-periodic linear subnet. 

(ii)  $\M_{\rm aa}$ is determined by the word  $a_xa_y$ and is the union of parallel planar embeddings of {\bf sql}.

(iii) $\M_{\rm aaf_z}$ is the net for $a_xa_yf_z$ and is the union of parallel planar embeddings of {\bf hxl}.

(iv) $\M_{\rm fff}$ is the net for $f_xf_yf_z$ and is the translation transitive union of 2 disjoint copies of an embedding of ${\bf pcu}$.

(v) $\M_{\rm ggd}$ is the net for $g_xg_yd_1$ and is the translation transitive union of 3 disjoint copies of an embedding of ${\bf pcu}$.

(vi) $\M_{\rm ggg}^d$ is the net for $g_xg_yg_zd_1$ and is the translation transitive union of 3 disjoint copies of an embedding of ${\bf hex}$.
\medskip

\begin{center}
\begin{figure}[ht]
\centering
\includegraphics[width=4.3cm]{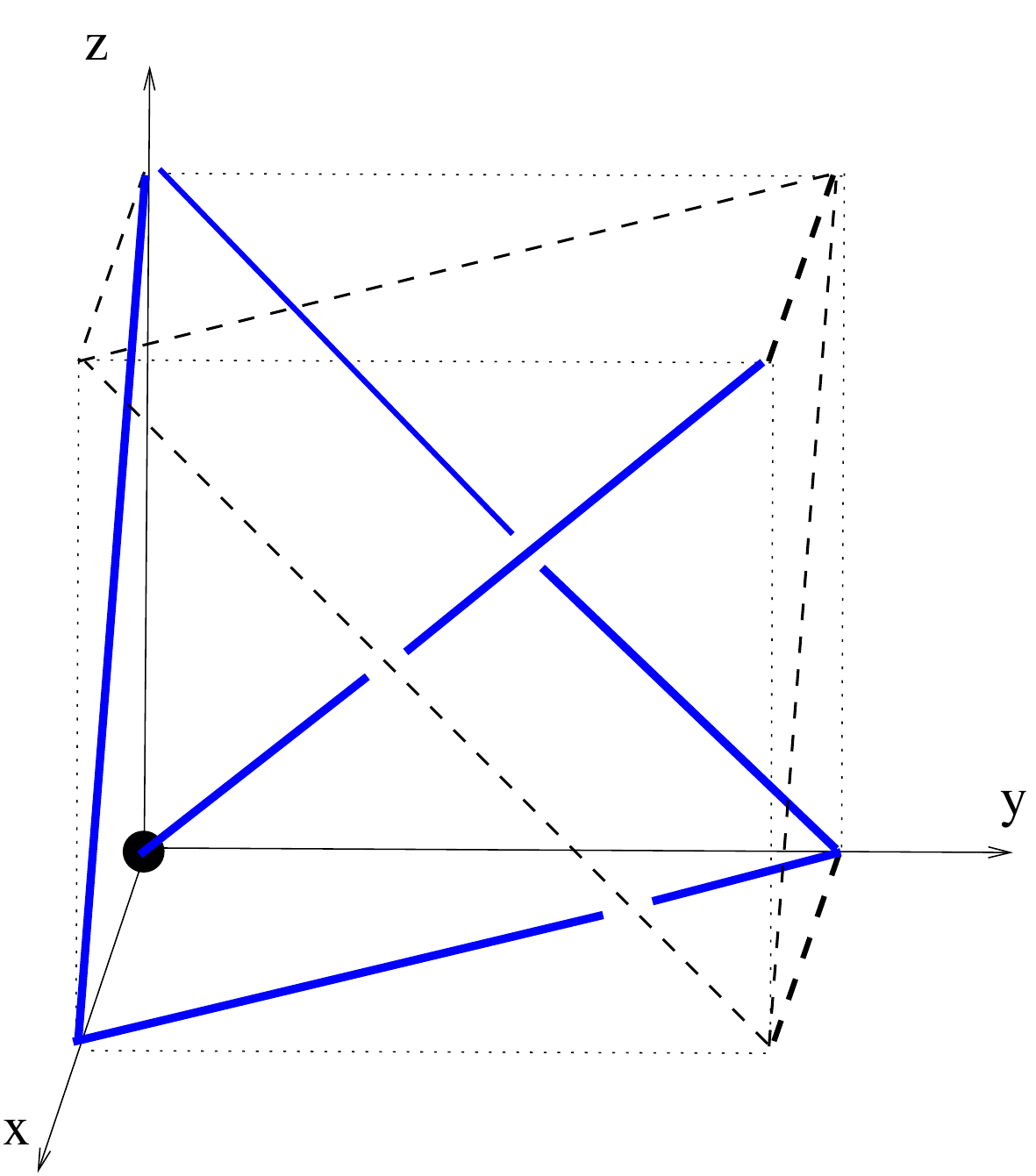}\quad \quad \quad
\includegraphics[width=4.3cm]{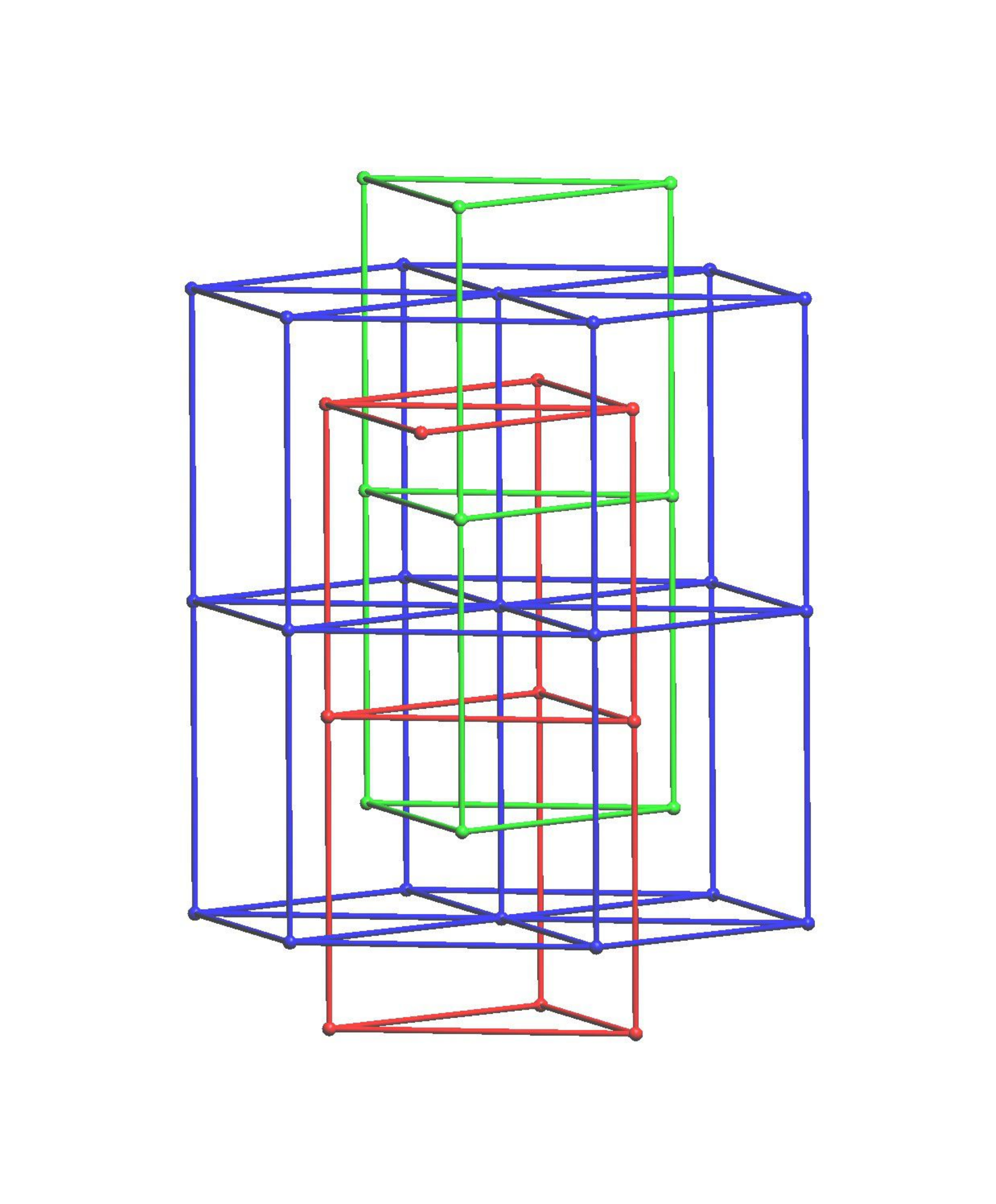}
\caption{\{a\} A single vertex quadruple-edged building block for the model net $\M_{\rm ggg}^d$. \{b\} A (rescaled) fragment of this net. The diagonal edge in the building block gives a $3^2$-penetrating edge.}
\label{f:gggd}
\end{figure}
\end{center}

In the above list, and in Tables 1, 3 below, we use a compact notation where the letter subscripts for the nondiagonal edges are suppressed if they appear in alphabetical order, and where $d$ indicates the diagonal edge $d_1$.
Thus, the model net for $w= g_xg_yg_zd_1$, which could be written as $\M(g_xg_yg_zd_1)$, is written in the compact form $\M_{\rm ggg}^d$. Its repeating unit, or motif, is indicated in Figure \ref{f:gggd} along with a fragment of the embedding rotated so that the penetrating edges are vertical.

\begin{thm}
There are 6 affine equivalence classes of disconnected embedded nets with adjacency depth 1 and a single vertex quotient graph.
\end{thm}

\begin{proof}
Let $\M$ be a model net of the type stated, with generating edge set $F_e$ with $|F_e|=m$. If $m=1$ (resp. $m=2$) then $\M$ is  affinely equivalent to $\M_{\rm a}$ (resp. $\M_{\rm aa}$).

Let $m=3$. Then the 3 edges of $F_e$ have translates, under the periodic structure, to 3 edges in $\M$ which are incident to a single node. Suppose first that the triple is coplanar. Then it determines a planar subnet $\M_1$ say, which is an embedding of {\bf hxl}. Also $\M$ is equal to the union of the translates of $\M_1$
of the form $\M_1+ nb$ where $b$ is a vector of integers and $n\in \bZ$. Thus $\M$ is affinely equivalent to $\M_{\rm aaf}$.

On the other hand, if the edges of $F_e$ are not coplanar then  $\M_1$ is an embedding of ${\bf pcu}$. Examination shows that this occurs with $\M$ disconnected, only for words $w$ of the forms (i) $fff$, giving 2 components, (ii) $fgg, gfg$ or $ ggf$, each of which is of type $fff$ after a translation and rotation, and (iii) $ggd$, gives which gives 3 components.

For $m\geq 4$ the model net  $\M_{\rm ggg}^d$ is the only net which is not connnected.
\end{proof}

\begin{center}
\begin{tabular}{|c|c|c|c|c|c|}
\hline
Model net   & Edge word  &Coordination & Net/multiplicity\\
\hline
\hline
 $\M_{\rm a}$&$a_x$ & 2& {line $\infty$}\\
\hline
 $\M_{\rm aa}$&$a_xa_y$ & 4 & {\bf sql $\infty$}\\
\hline
 $\M_{\rm aaf_z}$&$a_xa_yf_z$ & 6 & {\bf hxl} $\infty$\\
\hline
 $\M_{\rm fff}$&$f_xf_yf_z$ & 6 & {\bf pcu-c}\\
\hline
$\M_{\rm ggd}$&$g_xg_yd_1$ & 6 & {\bf pcu-c3}\\
\hline
 $\M_{\rm ggg}^d$&$g_xg_yg_zd_1$ & 8 & {\bf hex-c3}\\
\hline
\end{tabular}\\
\end{center}\medskip

\begin{center}
Table 2. Disconnected nets with a single vertex, depth 1, labelled quotient graph.
\end{center}
\bigskip

\subsection{Connected lattice nets with depth 1}
Let $\fN_1$ denote the family of proper linear 3-periodic nets with a periodic structure basis providing a depth 1 labelled quotient graph with a single vertex. 
We now consider the subfamily $\fN_1^c$ of connected nets $\N$ in $\fN_1$. 
These nets also gives building block nets for embedded nets with a double vertex quotient graph, and for multicomponent nets.  
In Theorem \ref{t:1vertexQG} and Corollary \ref{t:single_v_isotopy} we   classify the nets   up to oriented affine isomorphism and up to  periodic isotopy respectively, there being 19 classes for each equivalence relation. 
As in the previous section it will suffice to consider model nets. 
{Moreover} each model net $\M$ in $\fN_1^c$  is determined by an ordered word for the edges of a repeating unit $F_e$  {and these edges $[a,b]$ are subsets of the unit cell $[0,1)^3$ except for one of their endpoints}.
In view of connectivity and noncrossing conditions the quotient graph $QG(\M)$ is a bouquet graph with a single vertex and loop edges of multiplicity $m=3, 4, 5, 6$ or $7$ {(as implied by Lemma \ref{l:7and8lemma})}. 

To distinguish these model nets we make use of some {new} readily computable {local} features { which can be read off from the repeating unit and} which provide some {readily computable} structural invariants under affine isomorphism.

\begin{defn}
The \emph{hxl-multiplicity} $\hxl(\N)$ of an embedded net $\N$ is the number of translation classes of
planar 2-periodic subnets of $\N$ which are completely triangulated. 
\end{defn}

For the model nets $\M$ in $\fN_1^c$ this multiplicity is equal to the number of triples of edges $[a,b]$ in $F_e$ whose edge vectors, $b-a,$ form a coplanar triple.
It may  also be computed from the point symbol (PS) as the number of 3-cycles divided by 6. Thus the point symbol of {\bf fcu} is $3^{24}.4^36.5^6$ and so $\hxl(\N_{\rm fcu})=4.$




The next definition might be viewed as a strong form of local catenation.

\begin{defn}  (i) An edge of an embedded net is \emph{$3^2$-penetrating} if there exist 2 disjoint {parallel} edge-cycles of length 3 and an edge $[a,b]$ which  passes through them in the sense  that the open line segment $(a,b)$ intersects the convex hull of each cycle.

(ii) An edge of an embedded net \emph{$4^2$-penetrating} if it passes through 2 disjoint {parallel} untriangulated parallelograms.
\end{defn}
\begin{center}
\begin{figure}[ht]
\centering
\includegraphics[width=12cm]{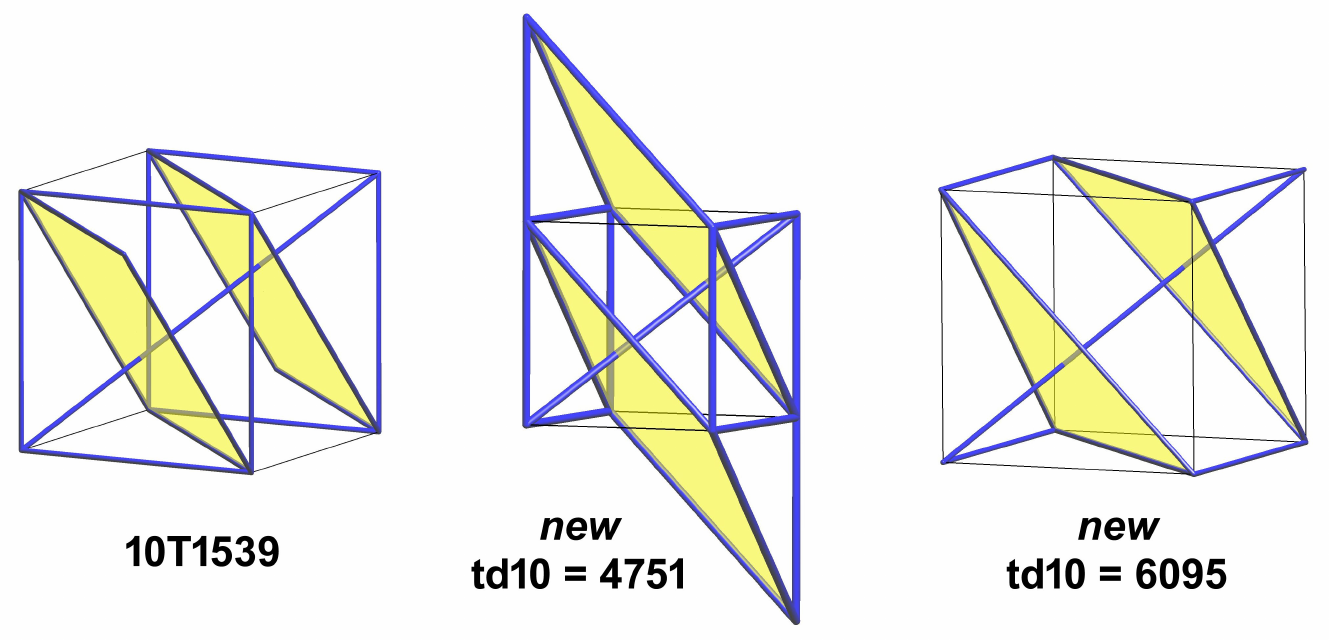}
\caption{Penetrating edges of type $3^2$ and $4^2$ as observed for three 10-coordinated nets from Table 3.}
\label{f:penetrating}
\end{figure}
\end{center}
One can check for example that for the  model net $\M$ in $\fN_1^c$ with a defining ordered word $w$ 
there exists a $3^2$-penetrating edge 
if and only if $w$ contains the subword $g_xg_yg_zd_1$. Also there exists a $4^2$-penetrating edge if and only if $w$ contains $d_1$ and precisely 2 of the 3 letters $g_x, g_y, g_z$. See Figure \ref{f:penetrating}.

We similarly define when an edge is $3^1$-penetrating or $4^1$-penetrating. In fact there are no depth 1 lattice nets with a $3^1$-penetrating edge. In general let us say that $\N$ has \emph{property $3^k$}  if there are $3^k$-penetrating edges but no $3^{k+1}$ penetrating edges. We also define \emph{property $4^k$} similarly.
We indicate these properties in column 5 of Table 3.



\subsection{Classification of depth 1 lattice nets} We now define 19 model nets $\M$ in $\fN_1^c$ with standard orthonormal basis as a depth 1 periodicity basis and where in each case the node set is the subset $\bZ^3$ of $\bR^3$. We do this, as in the previous section, by specifying a defining edge word, as listed in column 2 of Table 3. 
The 9 nets without the strong edge penetration property {(of type $3^2$ or $4^2$)} appear in the RCSR 
whereas the other 10 nets do not. This reflects the fact that the strongly penetrated nets can be viewed as exotic forms in reticular chemistry. Indeed, there are 3 new topologies which have not been observed either in the RCSR and or the ToposPro net databases.

\bigskip

\begin{center}
\begin{tabular}{|c|c|c|c|c|c|c|c|}
\hline
Model net   & Edge word & $\hxl(\M)$ &Coordination & Penetration & Topology&td10&$\pi$(-)\\
\hline
\hline
 $\M_{\rm pcu}$&$a_xa_ya_z$ &0& 6&0 & {\bf pcu}&1561&48\\
\hline
\hline
 $\M_{\rm pcu}^{d}$  &$a_xa_ya_zd_1$&0& 8  & 0  & {\bf bcu}&2331&48\\
\hline
 $\M_{\rm pcu}^{f}$ &$a_xa_ya_zf_x$&1& 8 & 0& {\bf hex}&2331&24\\
\hline
$\M_{aad}^g$&$a_xa_yg_xd_1$ 
&0& 8   & $4^1$ & {\bf ilc}&3321&12\\
\hline
$\M_{ad}^{gg}$
& $a_xg_xg_yd_1$
&0& 8  & $4^2$ & 8T17 &4497&4
\\
\hline
$\M_{ad}^{g_yg_z}$& $a_xg_yg_zd_1$ &0& 8  & $4^2$& 8T21 &4041&12
\\
\hline
\hline
$\M_{\rm pcu}^{ff}$&$a_xa_ya_zf_xf_y$ &2& 10 &0 & {\bf bct} &3101&16\\
\hline
$\M_{pcu}^{gd}$
&$a_xa_ya_zg_xd_1$ &1& 10 &$4^1$ & {\bf ile}&3761&8
\\
\hline
$\M_{aad}^{gg}$&$a_xa_yg_xg_yd_1$ &0& 10 & $4^2$& 10T1539 &4991&4
\\
\hline
$\M_{aad}^{g_xg_z}$& $a_xa_yg_xg_zd_1$
& 1  & 10 & $4^2$ & \emph{new}  &4751&4
 \\
\hline
$\M_{ad}^{ggg}$& $a_xg_xg_yg_zd_1$ &1& 10 & $3^2$& \emph{new} &6095&4
\\
\hline
\hline
$\M_{\rm pcu}^{fff}$
&$a_xa_ya_zf_xf_yf_z$ &3& 12& $4^1$ & {\bf ild} & 4201&12
\\
\hline
$\M_{\rm pcu}^{ggg}$
& $a_xa_ya_zg_xg_yg_z$ &4& 12 &0 & {\bf fcu} &3871&48
\\
\hline
$\M_{\rm pcu}^{ggd}$
& $a_xa_ya_zg_xg_yd_1$ &2& 12& $4^2$& 12T1305 &5191&4
\\
\hline
$\M_{aad}^{ggf}$
&$a_xa_yg_xg_yf_zd_1$ &1& 12& $4^2$ & 12T1657 &5431&12
\\
\hline
$\M_{aad}^{ggg}$& $a_xa_yg_xg_yg_zd_1$&2& 12& $3^2, 4^1$& \emph{new} &6421&4
\\
\hline
\hline
$\M_{\rm pcu}^{fffd}$&$a_xa_ya_zf_xf_yf_zd_1$ &6& 14& $4^1$ & {\bf bcu-x}&4641&48\\
\hline
$\M_{\rm pcu}^{ggfd}$& $a_xa_ya_zg_xg_yf_zd_1$ &4& 14& $4^2$& 14T199 & 5631&12
\\
\hline
$\M_{\rm pcu}^{gggd}$& $a_xa_ya_zg_xg_yg_zd_1$ &4& 14& $3^2, 4^1$& 14T957 &6621&12
\\
\hline
\end{tabular}\\ 
\end{center}\medskip

\begin{center}
Table 3. Connected nets with a single vertex, depth 1, labelled quotient graph.
\end{center}
\bigskip

\begin{center}
\begin{figure}[ht]
\centering
\includegraphics[width=3.5cm]{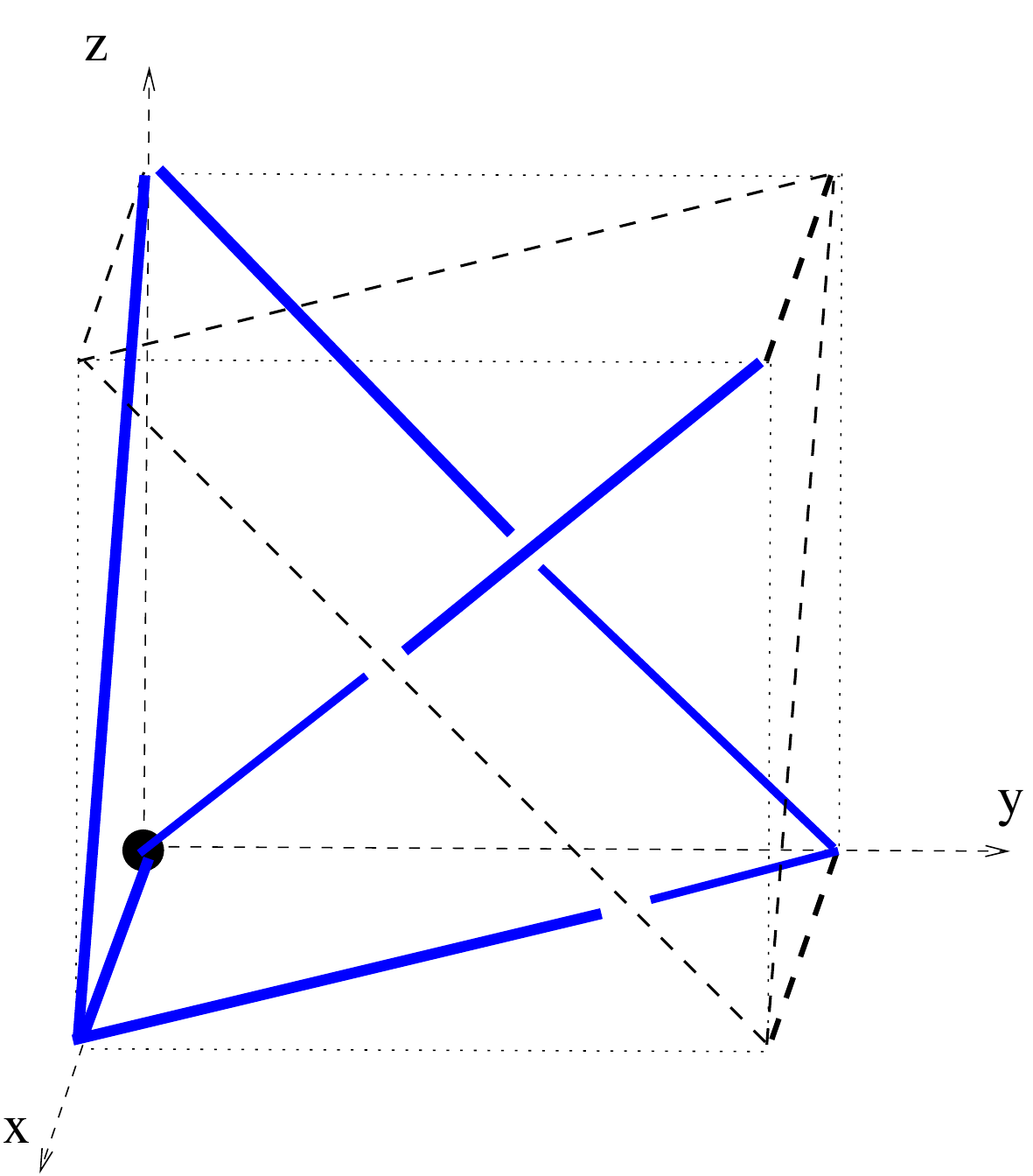}\quad 
\includegraphics[width=3.5cm]{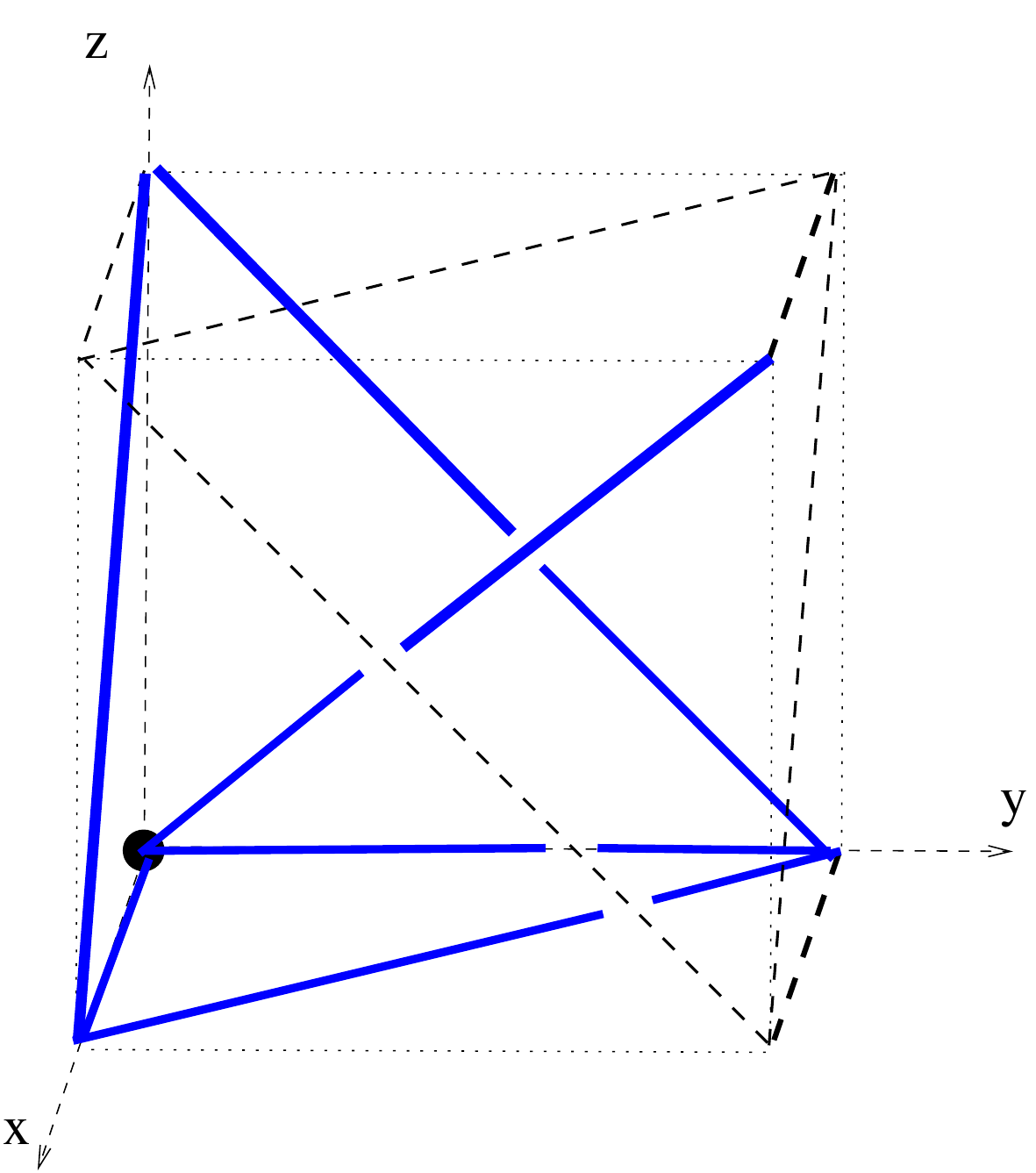}
\quad 
\includegraphics[width=3.5cm]{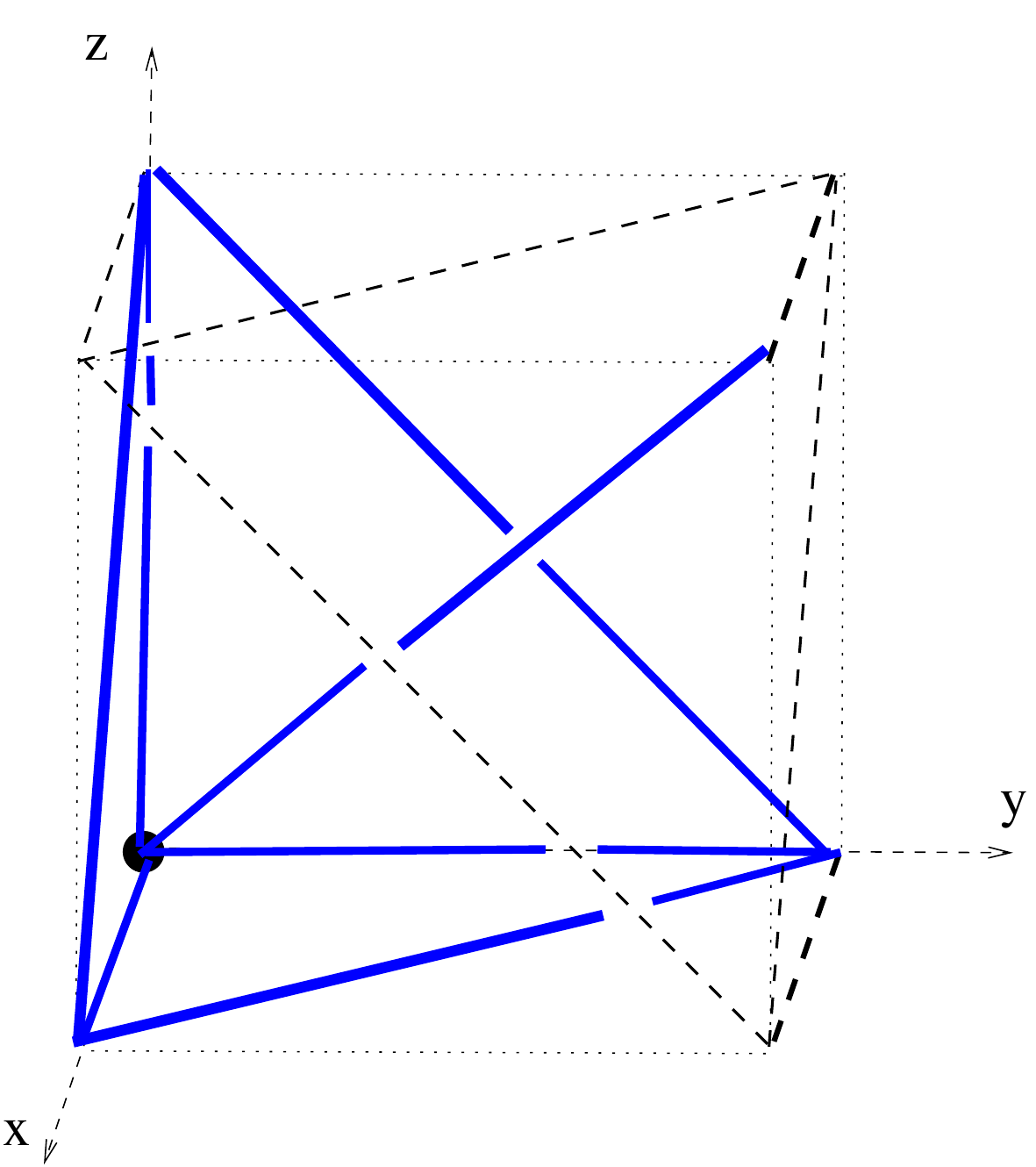}
\caption{Motifs for the model nets $\M_{ad}^{ggg},$ $\M_{aad}^{ggg}$ and $\M_{\rm pcu}^{gggd}$ (14T957)  which have a $3^2$-penetrating edge.}
\label{f:3nets}
\end{figure}
\end{center}

We also record in the final column the cardinality of the point group of the maximal symmetry net with the given topology, which we may denote by  $\pi({\bf pcu})$ etc.
\medskip

Let us define an \emph{elementary affine transformation} of $\bR^3$ to be a rotation, translation, or  a linear map whose representing matrix has entries 1 on the main diagonal and a single nonzero nondiagonal entry equal to 1 or $-1$. These maps, such as
$(x,y,z)\to (x-z,y,z)$, map model nets to model nets and play a useful role in case-by-case analysis.

\begin{rem}\label{r:ilcremark} We note that the countable graph {\bf ilc}, represented by the model net $\M_{aad}^g$ can be represented in other ways. The model net $\M_{fff}^d$ gives one such alternative.
The topology is also made apparent by its  equivalence, by elementary transformations, with the net  obtained  from the {\bf pcu} model net by the addition of integer translates of the long diagonal edges with edge vector $(1,2,1)$. However, in this case the standard basis is a periodic structure basis of depth 2.
\end{rem}

\begin{thm}\label{t:1vertexQG}
There are 19  oriented affine equivalence classes of connected lattice nets with depth 1.
\end{thm}  


We have obtained this classification by means of a case-by-case proof as well as a verification by an enumeration of lattice nets using GAP. The following interesting special case, with two new nets,  illustrates the general proof method. (See also the Extended Supplementary Information.)

\begin{proof}[Determination of the 10-coordinated connected lattice nets of depth 1.]  Suppose first that a model net $\M$ in this case has 3 axial edges and 2 face edges. Then it is straightforward to see that it is equivalent by elementary affine transformations to the model net $\M_{\rm pcu}^{ff}$, for {\bf bct}. Also, any model net of type $aaafd$ is similarly equivalent to this type. On the other hand, a type $aaagd$ model net has $\hxl$-multiplicity equal to $1$, rather than $2$, and so  represents a new equivalence class. Its topology is {\bf ile}.

Consider next the model nets with 2 axial edges and no diagonal edges. These are equivalent by elementary affine transformations to a model net with 3 axial edges and so they are equivalent to the model nets in Table 3 for {\bf bct} and {\bf ile}.
The same is true for the 9 nets of type $aawd$ where $w$ is a word in 2 facial edges which is not of type $gg$.

Thus, in the case of 2 axial edges it remains to consider the types $a_xa_ywd_1$ with $w= g_xg_y, g_xg_z$ and $g_yg_z$. Each of these has a penetrating edge of type $4^2$. The first two are model nets in the list and they give new and distinct affine equivalence classes in view of their penetration type and differing $\hxl(\N)$ count. The third net, for the word $a_xa_yg_yg_zd_1$ is a mirror image of the first net and is orientedly affinely equivalent to it, by Remark \ref{r:mirror2vertex} for example.

It remains to consider the case of 1 axial edge, $a_x$, together with $d_1$ and 3 facial edges. If there are 2 edges of type $f_x, f_y$ or $f_z$ then there is an elementary equivalence with a model net with 2 axial edges. The same applies if there is a single such edge. (For an explicit example consider $a_xf_xg_yg_zd_1$ and check that the image of this net under the transformation $(x,y,z)\to (x,y-z,z)$ gives a depth 1 net with 2 axial edges.)


Finally the model net for $a_xg_xg_yg_zd_1$ appears in the listing and gives a new class with penetration type $3^2$.
\end{proof}

\medskip
\begin{center}
\begin{figure}[ht]
\centering
\includegraphics[width=11cm]
{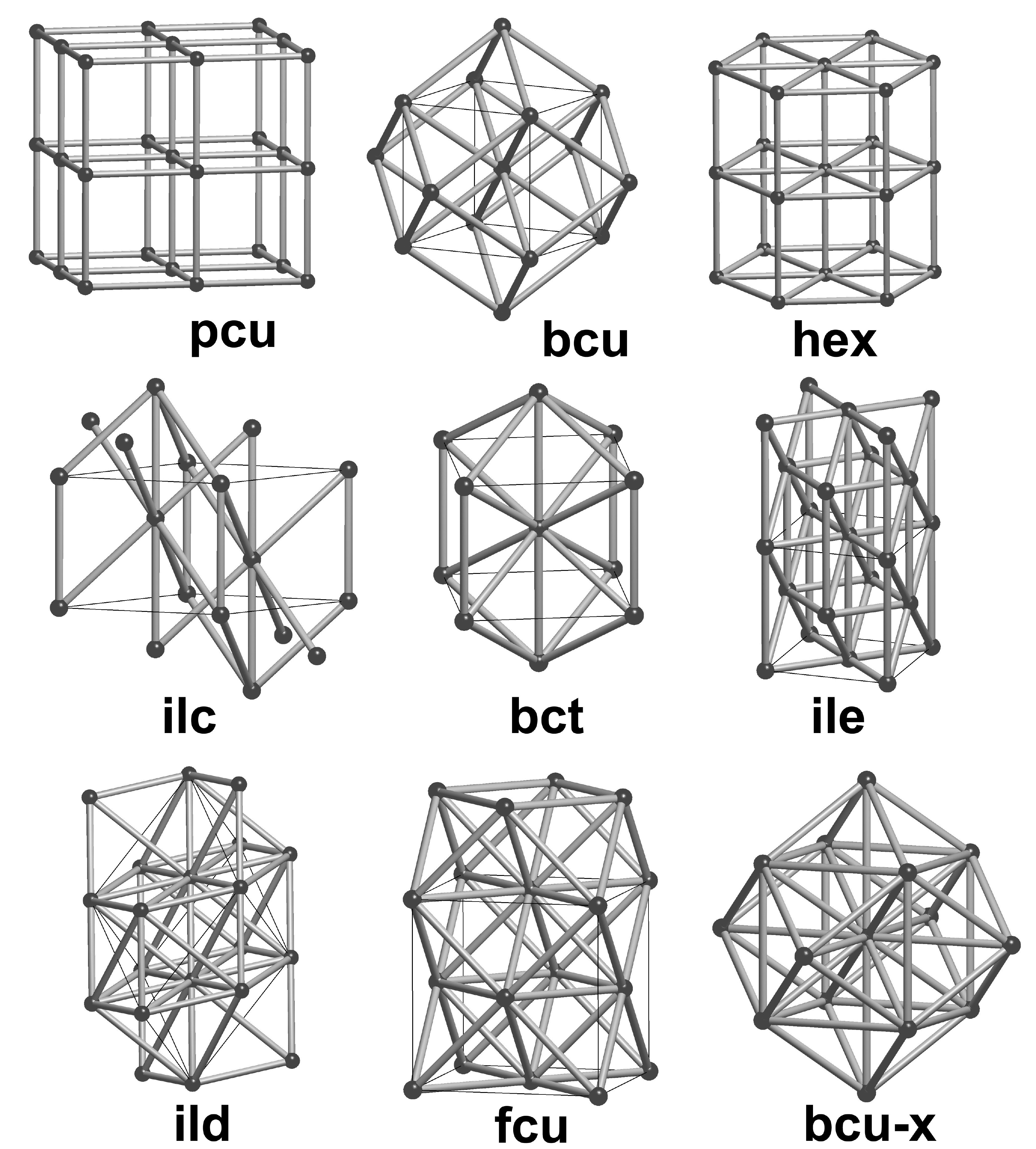}
\caption{Maximal symmetry embeddings of the 9 model nets of $\fN_1^c$ which do not have the $3^2$- or $4^2$-penetration property. }
\label{f:rcsr9nets}
\end{figure}
\end{center}
\medskip

\begin{cor}\label{t:single_v_isotopy}
There are 19 periodic isotopy classes of connected linear 3-periodic nets in $\bR^3$ with adjacency depth 1 and a single vertex quotient graph.
\end{cor}

\begin{proof}
If the connected linear 3-periodic nets $\N_1, \N_2$ are orientedly affinely equivalent then, as previously observed, 
they are periodically isotopic. Thus there are at most 19 periodic isotopy classes. On the other hand periodically isotopic embedded nets have structure graphs which are isomorphic as countable graphs. Since the 19 model nets have nonisomorphic structure graphs the proof is complete.
\end{proof}

Theorem \ref{t:1vertexQG}, together with the linear implementation of graph isomorphisms indicated in Theorem \ref{t:kostousov},  implies that the structure graphs of the 19 model nets must be nonisomorphic as graphs. 
This also follows on examining the td10 topological density count. We remark that
 Table 3 , without the final td10 column, almost distinguishes 19 affine equivalence classes (and hence, by Theorem \ref{t:kostousov}, the structure graphs) since we have only appealed to topology density to distinguish the curious pair $\M_{\rm ad}^{gg}, \M_{\rm ad}^{g_yg_z}$.



\section{Double lattice nets and further directions}
\label{s:furtherdirections}
We give a brief indication of research directions in the determination of periodic isotopy classes and periodic isotopes for  embedded nets with a double vertex quotient graph as well as research directions in rigidity and flexibility.

\subsection{Double lattice nets}\label{ss:doublelattice}
For convenience we define a \emph{double lattice net} to be an embedded periodic net $\N$ in $\bR^3$ whose set of nodes is the
union of 2 translationally equivalent rank 3 lattices and we let $\fN_2$ be the family of proper double lattice nets with adjacency depth 1.




The double vertex quotient graph $\N$ in $\fN_2$ consists of 2 bouquet graphs and a number of nonloop edges. We denote these graphs  as $H(m_1,m_2,m_3)$ where $m_1$ and $m_3$ are the loop multiplicites, with $m_1\geq m_3\geq 0$, and $m_2$ is the multiplicity of the connecting edges. From Lemma \ref{l:7and8lemma} we have the  necessary conditions  $0\leq m_1\leq 7$ and $0\leq m_2\leq 8$ as well as  $m_3\geq 1$ if $m_2=1$, since, from the definition of a linear 3-periodic net, there can be no nodes of degree 1. If $\N$ is a net in $\fN_2^c$, the subfamily of connected nets, then we also have the additional condition $m_2 \geq 1$.

Each net $\N \in \fN_2$  admits a unique  3-fold decomposition
$
\N = \N_1 \cup \N_2 \cup \N_3
$
where $\N_1$ and $\N_3$ are the disjoint 3-periodic subnets associated with the two bouquet subgraphs and where $\N_2$ is the  net, with the same node set as $\N$, associated with the subgraph with non-loop edges. The subnets $\N_1, \N_3$ may have no edges if one or both vertices has  no loop edges.
When loops are present on both vertices then the nets $\N_1, \N_3$ are \emph{bouquet nets},  and are of  three possible dimension types, namely  $\{3;1\}, \{3;2\}$ or $\{3;3\}$. As we have seen earlier, for type $\{3;1\}$ there is 1 affine isomorphism class of embedded nets, for type
$\{3;2\}$ there are 2 such classes and for type $\{3;3\}$ there are 3 classes for disconnected nets and 19 classes for connected nets.

Thus in the 3-fold decomposition of a net $\N$ in $\fN_2$, each of the subnets $\N_1,$ $\N_3$ is either devoid of edges or is \emph{separately} orientedly affinely equivalent to one of the 25 model nets for $\fN_1$. The relative position (parallel or inclined, for example) of these component nets allows for considerable diversity for the entangled net $\N_1\cup \N_3$. In particular, while $\N$ is affinely equivalent to a \emph{general} model net $\M_1\cup \M_2\cup \M_3$, with standard periodic structure basis $\ul{b}$, in general we can only additionally arrange that one of the subnets $\M_1, \M_3$ is equal to a translate of one of the specific 25 model nets in Tables 2 and 3.

Evidently there is a considerably diversity for the periodic isotopy classes of embedded nets with depth 1 and a double vertex quotient graph. We now show that there is even a marked increase in the number of topologies for such nets. 



For $1 \leq m \leq 8$ define $\fN_2^*(0,m,0)$  
to be the family of nets $\N$ in $\fN_2$ which have a periodicity basis with a depth 1 bipartite quotient graph $H(0,m,0)$ with an edge carrying the label $(0,0,0)$. The label condition here ensures the natural condition that $\N$ has an edge between the pair of representative joints in the semi-open unit cell for the periodicity basis.
In fact this convention, which we call the \emph{unit cell property}, is the natural convention used by
Chung, Hahn and Klee \cite{chu-hah-kle} in their schemes for the enumeration of periodic nets. 

 For $m=1,2,3$ the nets of this type are not connected. For $m=4$ it is well known that there is a unique \emph{connected} topology $G(\N)$ for the nets in $\fN_2^*(0,m,0)$, namely the diamond net {\bf dia}. For higher values of $m$ we are able to determine the topologies through a computational analysis based in part on the indivisibility criterion Proposition \ref{p:divisible}.
See also the extended supplementary information.





 \bigskip


\begin{prop}\label{p:117topologies} There are 117 nonisomorphic topologies for bipartite double lattice nets 
with the unit cell property, which are connected and have adjacency depth 1. Moreover, the number of $m$-coordinated topologies, for  $m=4, 5,6,7$ and 8, are, respectively, 1, 11, 31, 40 and 34. 
\end{prop}

\subsection{Rigidity and flexibility}\label{s:rigidity}
The analysis of infinitesimal rigidity and flexibility for connected crystal frameworks $\C$ is a well-developed mathematical topic. 
In its simplest form a velocity field on the node set is assumed to be periodic with respect to a given periodicity basis $\ul{a}$. This is the so-called \emph{fixed lattice theory} and in fact it corresponds exactly to the rigidity theory of fixed edge-length graph knots on a fixed flat torus for the parallelepiped defined by the periodicity basis.  In this case a finite matrix, the \emph{periodic rigidity matrix} for the pair $(\C, \ul{a})$, determines the space of periodic infinitesimal flexes and so this matrix is a discriminator for the (strict) \emph{periodic rigidity} of $\C$ with respect to $\ul{a}$. On the other hand the \emph{flexible lattice theory} allows for infinitesimal motions of the periodicity basis and so embraces a larger finite dimensional vector space of velocity fields with a correspondingly larger rigidity matrix. See \cite{bor-str}, \cite{pow-affine}. Recently necessary and sufficient conditions have been given for infinitesimal rigidity with respect to the infinite dimensional space of \emph{all} velocity fields. See Kastis and Power \cite{kas-pow-1}.




The fixed lattice theory also has close connections with the analysis of rigid unit modes (RUMs) in material crystals with a connected bond-node net.  
See for example the RUM mode analysis in \cite{bad-kit-pow}, \cite{pow-poly}. In fact this analysis also applies to disconnected crystal frameworks with several components if there are no interaction constraints between the components.  
Indeed, suppose that $\C$ belongs to the interpenetration class and let $\ul{a}$ be a periodicity basis for both $\C$ and each of its finitely many components $\C_i$. Then the RUM spectrum $\Omega(\C)$ of $\C$, with respect to $\ul{a}$, is the union of the RUM spectra of its components. 



A crystal framework is said to be \emph{critically coordinated}, or to be a \emph{Maxwell framework}, if the quotient graph satisfies
$|E|=3|V|$. This is often interpreted as an equality between the total number of constraints (provided by $|E|$ equations) that restrict the total number of degrees of freedom of a repeating unit of nodes, which is $3|V|$. It also implies an equality of limits of averages over increasing volumes for these constraint/freedom quantities. It is for such frameworks, which includes all zeolite frameworks for example, that the RUM spectrum is typically a nontrivial  algebraic variety exhibiting detailed structure \cite{dov-exotic}, \cite{pow-poly}, \cite{weg}.

In the light of this it is of interest to determine the basic 
Maxwell frameworks $\C$ which have a depth 1 labelled quotient graph with either 1 or 2 vertices. From Proposition \ref{p:117topologies} it follows that there are  31  topologies for  crystal frameworks of this type with the unit cell property and quotient graph $H(0,6,0)$.
These remarks suggest that it would be worthwhile to augment periodic net database resources with  tools for the identification of Maxwell lattices and the calculation of flexibility information related to RUM spectra.



\section{Appendix A}

{\bf The proof of Theorem \ref{t:multigridcount}.}

\begin{proof}
Note first that any connected component $K_i$ of $K$ is determined by the position of its unique node in $[0,1)^3$.  
Thus $K$ is determined by the position $p_i=(x_i, y_i, z_i), 1 \leq i \leq n$, of its $n$-tuple of nodes. Also, in view of the disjointness of components two such nodes $p_i, p_j$ have differing corresponding coordinates in $[0,1)$.
Consider a deformation path $(f_t)$ from $K$ to $K'$.
Since the graph knots $f_t(K)$ are also graph knots of $n$-grids, and edge collisions cannot occur in the deformation, it follows that  
the cyclical order of the $x$- $y$- and $z$-coordinates of the points  $f_t(p_1), \dots , f_t(p_n)$, is constant. Thus the ordered triple of cyclic orders for the coordinates is an invariant for linear graph knot isotopy.
 
Despite the constraint of coordinate distinctness we see that $K$ can be linearly isotopic to an $n$-grid graph knot $K'$ determined by $p_i'=(x_i', y_i', z_i'), 1 \leq i \leq n$,
where the $n$ coordinates $x_i'$ lie at the midpoints of the distinct subintervals of the form
$[j/n, (j+1)/n), 0 \leq j \leq n-1$. This spacing is achieved by  simultaneously translating the points $p_i$ in the $x$-direction at appropriate independent speeds while maintaining $x$-coordinate distinctness.
Additionally, the equal spacing of the $y$- and $z$-coordinates  can be achieved by similar isotopies which locally translate in the $y$- and  $z$-direction. The resulting position is unique up to the cyclic permutation action of $C_n\times C_n\times C_n$ on the coordinate axes. It follows now that two graph knots of $n$-grids are linearly  isotopic if the cyclic order of their coordinates coincide. Thus the set of cyclic orders is a complete invariant for linear graph knot isotopy and (i) and (ii) follow.

 Assume next that the $n$-grids $\N$ and $\N'$ are periodically isotopic. It will suffice to show that their linear  graph knots are rotationally linearly isotopic. 
 
 Without loss of generality we may assume that the components have node sets that lie on translates of the lattice $\bZ^3$ in $\bR^3$. Thus, by the definition of periodic isotopy there are periodicity bases $\ul{a}=\{a_1,a_2,a_3\}$ and $\ul{a'}=\{a_1',a_2',a_3'\}$, with integer entries, such that $(\N, \ul{a})$ and $(\N', \ul{a}')$ are strictly periodically isotopic by means of a deformation path $(f_t)$ and an associated path of bases 
 $\ul{a}^t$ from $\ul{a}$ to $\ul{a'}$. 
Define $k_1$ to be a common multiple of the $x$-coordinates of
$\{a_1,a_2,a_3\}$ and similarly define $k_2, k_3$ for the $y$- and $z$-coordinates.
Then there is an implied periodic isotopy between  $(\N, k\cdot \ul{b})$ and $(\N', \ul{a}'')$, for some periodicity basis $\ul{a}''$ with integer entries. This is given by the \emph{same} periodic isotopy deformation path $(f_t)$ but with a new associated path of  bases (for lower translational symmetry) which is determined by the initial basis $k\cdot \ul{b}$ and the path $\ul{a}^t$.
So, without loss of generality we may assume at the outset that $\ul{a}=  k\cdot \ul{b}$.


We next show that $\ul{a}'$ is equal to  $k'\cdot \ul{b}$ where $k'$ is a cyclic permutation of $k$. Thus we will obtain that the linear graph knots $\lgk(\N_0,k\cdot \ul{b}),$ $\lgk(\N_1,k'\cdot \ul{b})$ are rotationally linearly isotopic.

To see this  consider 
a single component $\N_0^1$ of the $n$-grid $\N_0$. Note that the linear graph knot $\lgk(\N_0^1,k\cdot \ul{b})$ has minimal discrete length cycles $c_1, c_2, c_3$ with homology classes $\delta_1, \delta_2, \delta_3$, respectively, equal to the standard generators of the homology group $H_1(\bT^3, \bZ)=\bZ^3$ of the containing flat 3-torus. These discrete lengths are $k_1, k_2, k_3$. Moreover we see, from the rectangular geometry of $\N_0^1$, the following uniqueness property, that if $c_1$ and $c_1'$ are 2 such minimal length cycles for $\delta_1$ which share a node then $c_1 = c_1'$. Indeed, the minimality implies that the edges of $c_1$ are parallel or, equivalently, that the nodes of $c_1$ can only differ in $x$-coordinate.

Let $\N_1^1$ be the corresponding component of $\N_1$. In fact 
$\N_1^1= f_1(\N_0^1)$). The linear graph knot $\lgk(\N_1^1,\ul{a}')$ is, by definition, equal to the affine rescaling of the intersection of the body $|\N_1^1|$ with the semiopen parallelepiped defined by the periodicity vectors $a_1', a_2', a_3'$. We note that if
$\ul{a}'$ is \emph{not} of the form  $k'\cdot \ul{b}$ then for at least one of the the standard generators $\delta_i$ (associated with $a_i'$) of the flat 3-torus homology group $H_1(\bT^3, \bZ)=\bZ^3$, the minimal length cycles do \emph{not} have the uniqueness property. This follows from elementary geometry since not all edges of the cycle can be parallel when $a_i'$ is not parallel to a coordinate axis.

On the other hand the linear isotopy between $\lgk(\N_0^1,k\cdot \ul{b})$ and $\lgk(\N_1^1,k'\cdot \ul{b})$ preserves the lengths of cycles of edges and so the claim follows.
Since we have shown that $\lgk(\N, k\cdot \ul{b})$ and $\lgk(\N', k\cdot \ul{b})$ are isotopic linear graph knots up to a rotation it follows from the technical lemma, Lemma \ref{l:amplifiedgrids}
that the graph knots $\lgk(\N,\ul{b})$ and $\lgk(\N', k\cdot \ul{b})$ are linearly isotopic up to a rotation, and so (iii) now follows from (ii). \end{proof}

\begin{rem}\label{r:conjecture} 
Recall the notation $k\cdot \ul{a}= (k_1a_1,k_2a_2,k_3a_3)$ introduced in the proof of Theorem  \ref{t:equivalencerelation}. Let us say that this denotes the \emph{amplification} by $k\in \bZ^3_+$ of the periodic structure basis $\ul{a}$.
We pose the following general problem. If $\N_1, \N_2$ are connected linear periodic nets which have a common periodic structure basis $\ul{a}$ and if  the pair 
$(\N_1,k\cdot \ul{a})$ is periodically isotopic to
the pair $(\N_2,k\cdot \ul{a})$, then does it follow that $\N_1$ and $\N_2$ are periodically isotopic ? In view of the proposition above this is equivalent to the corresponding problem for linear graph knots $K$ and their $k$-fold \emph{amplifications} which we may write as $k\cdot K$. We expect that this is true and therefore  that the amplified knots are isotopic if and only if the unamplified knots are isotopic. In fact one can verify this connection for various specific classes of interest, as we do below in the case of multigrid nets.
\end{rem}

The following technical lemma resolves the question of Remark \ref{r:conjecture} in the case of $n$-grids. 

\begin{lem}\label{l:amplifiedgrids} Let $\N_0$ and $\N_1$ be shift homogeneous  $n$-grids with standard linear graph knots $K_0$ and $ K_1$ and suppose that for some $k\in \bZ^3$ the amplified graph knots $k\cdot K_0$ and $k \cdot K_1$ are isotopic. Then $K_0$ and $K_1$ are strictly linearly isotopic.
\end{lem}

\begin{proof}
Figure \ref{f:amppcu} indicates a subgraph knot, $C_0$ say, of one of the components of $k\cdot K_0$ in the case that $k=(3,6,5)$. We refer to this as a "chain". It consists of a small cube of edges attached to 3 cycles of edges in the axial directions. Let $p_1$ denote the vertex which is common to these 3 cycles. The other $n-1$ components of $k\cdot K_0$ have similar chains which are shifts of $C_0$ and there is a \emph{unique} such chain where the shift of $p_1$ lies in the semiopen small cube $p_1+[0,1)^3$ of the flat 3-torus. Let $p_2, \dots , p_n$ be these axial joints and  let $J_0= J_0(p_1,\dots , p_n)$ be the union of these chains (giving a linear subgraph knot of $k\cdot K_0$).
\begin{center}
\begin{figure}[ht]
\centering
\includegraphics[width=5.5cm]{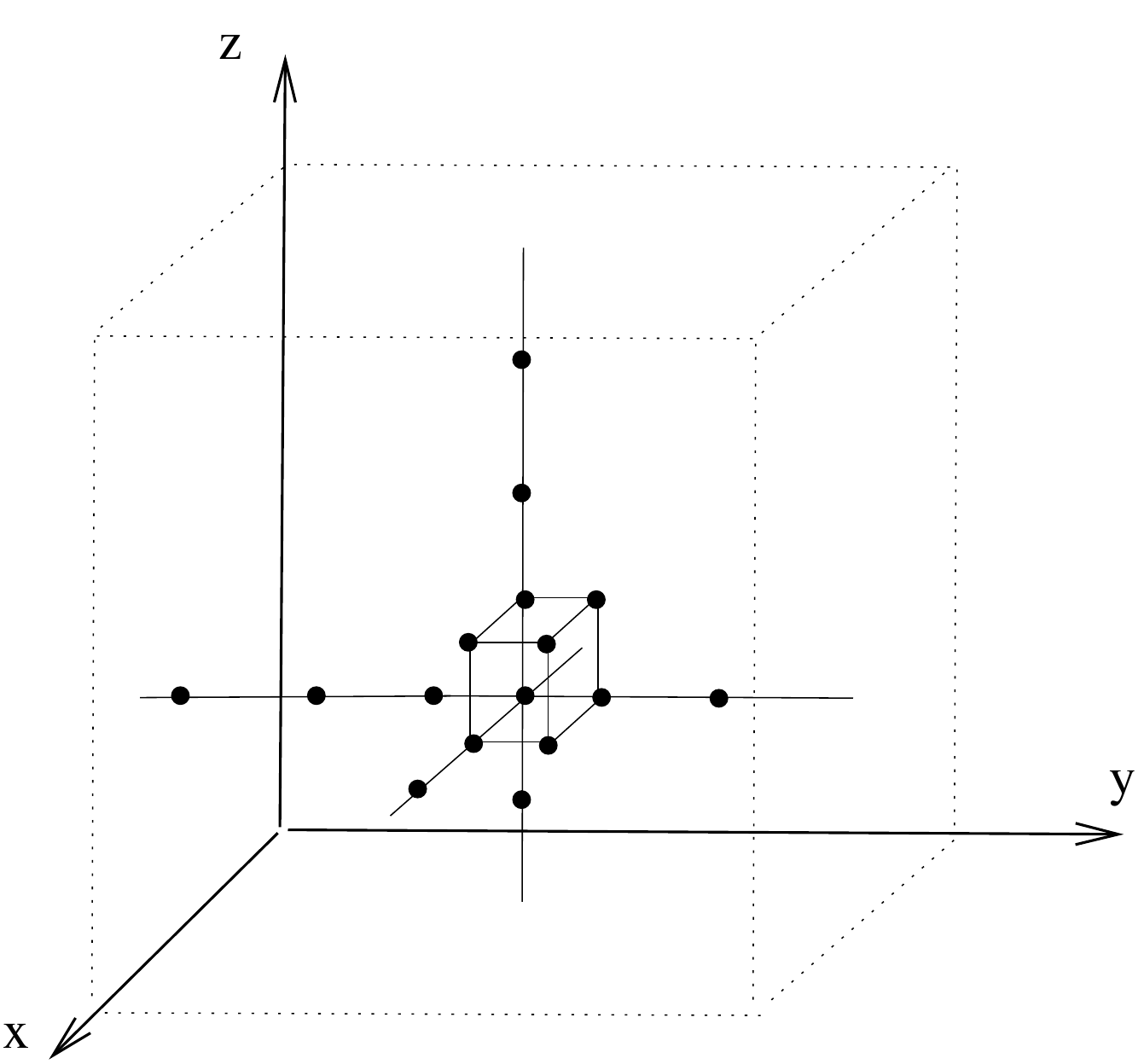}
\caption{A chain subgraph knot of $k\cdot K_{\rm pcu}$.} 
\label{f:amppcu}
\end{figure}
\end{center}

Suppose that $k\cdot K_0$ and $k\cdot K_1$ are linearly isotopic, by the isotopy $(g_t), 0\leq t\leq 1$. This restricts to a linear isotopy from $J_0$ to a subgraph knot  $g_1(J_0)$ of  $k\cdot K_1$. In this isotopy the images under $g_t$, for $0 <t<1$,  of the $n$ axial cycles of $J_0$ in a specific coordinate direction need not be linear. However, since there can be no collisions the cyclical order for $t=0$ agrees with the cyclical orders for $t=1$. It follows that 
$g_1(J_0)$, which has the form
$J_0(q_1,\dots , q_n)$, is a subgraph knot of $k\cdot K_1$ of the same cyclical type as the subgraph knot $J_0$. Since $K_0$ and $K_1$ are also defined by the cyclical order of $p_1, \dots p_n$ and $q_1, \dots ,q_n$ it follows that they are linearly isotopic.
\end{proof}

\section{Appendix B}



{\bf The proof of Theorem \ref{t:1vertexQG}.}

\begin{proof} Let $\M$ be model net with adjacency depth 1 and a single vertex quotient graph. We show that $\M$ is equivalent
to one of the 19 model nets by an elementary affine transformation.

\medskip

{\bf The case $m=3$.} In all cases it is clear that $\M$ is equivalent to $\M_{\rm pcu}$.
\medskip

{\bf The case  $m=4$.} We consider 4 subcases:

(i) Assume that 3 of the edges of $F_e$ are axial edges. Then  $\M$ is obtained from $\M_{\rm pcu}$ by the addition of an additional edge to the motif. If this is a facial edge then, by rotation and translation $\M$ is equivalent to $\M_{\rm pcu}^f$, the model net for the word $a_xa_ya_zf_x$. If the extra edge  is a diagonal  edge then $\M$ is equivalent to $\M_{\rm pcu}^d$.

(ii) Assume that exactly  $2$ of the 4 edges of $F_e$ are axial edges. We may these are $a_x, a_y$ and we may also assume that neither of the remaining 2 edges is in the $xy$-plane since in this case there would be a triple of coplanar edges in $F_e$ and $\M$ would be equivalent to $\M_{\rm pcu}^f$. Suppose first that there is no diagonal edge and so $\M$ is of type $a_xa_yw$ with $w$ one of $f_xf_y, f_xg_y, g_xf_y, g_xg_y$. These nets are pairwise equivalent by rotation about the $z$-axis and translation. By an elementary affine transformation they are thus all equivalent to $\M_{\rm pcu}^d$. 

Assume on the other hand that only 1 of the 2 extra edges is a facial edge. Translating and rotating we may assume that this edge is $f_x$. Also we may assume a noncoplanarity position of the diagonal edge with respect to $f_x$ and $a_x$, as in Figure \ref{f:motifexamples2}, since otherwise there is an oriented affine equivalence with the model net for ${\bf hex}$.
\begin{center}
\begin{figure}[ht]
\centering
\includegraphics[width=3.5cm]{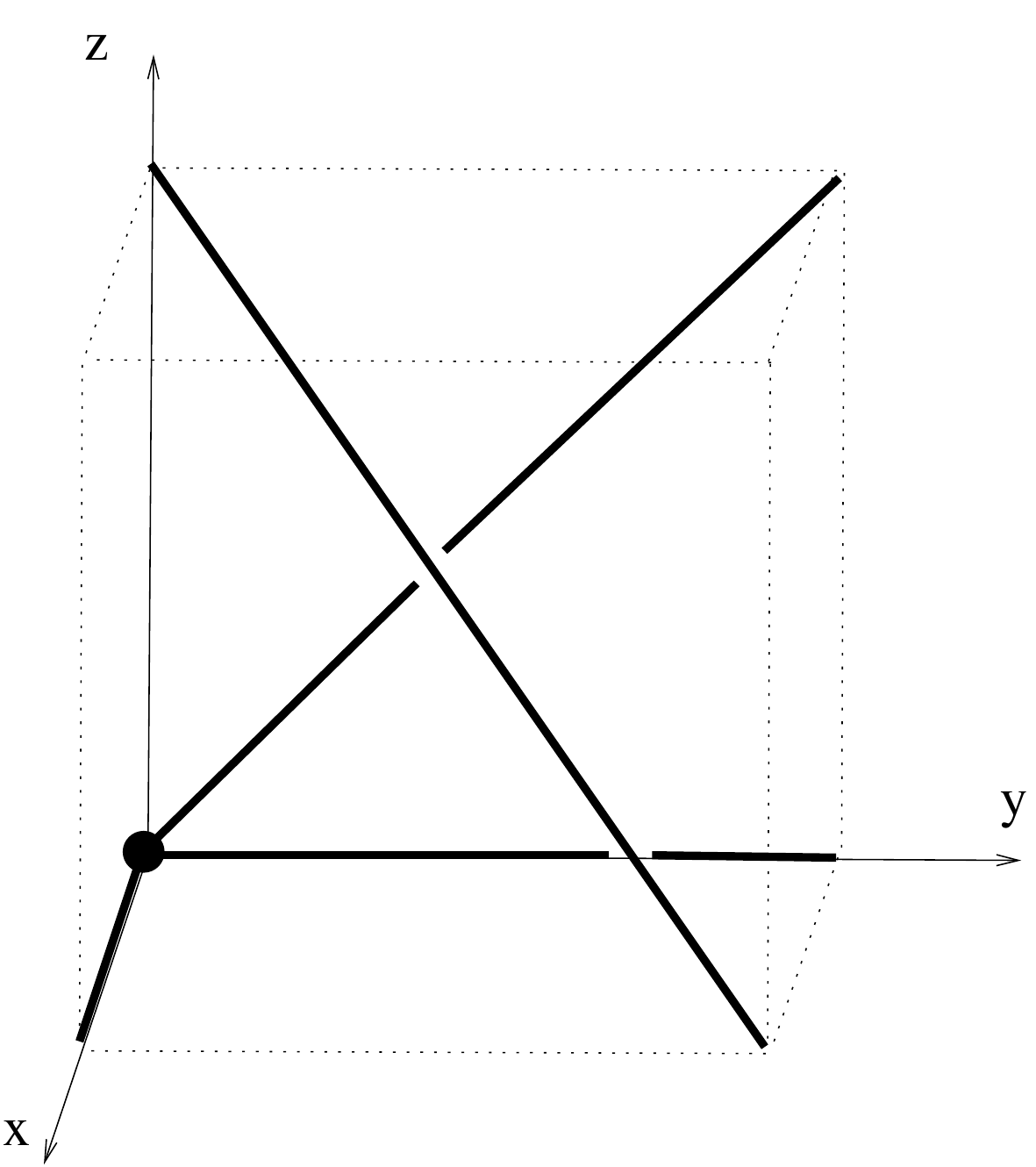}\quad
\includegraphics[width=3.5cm]{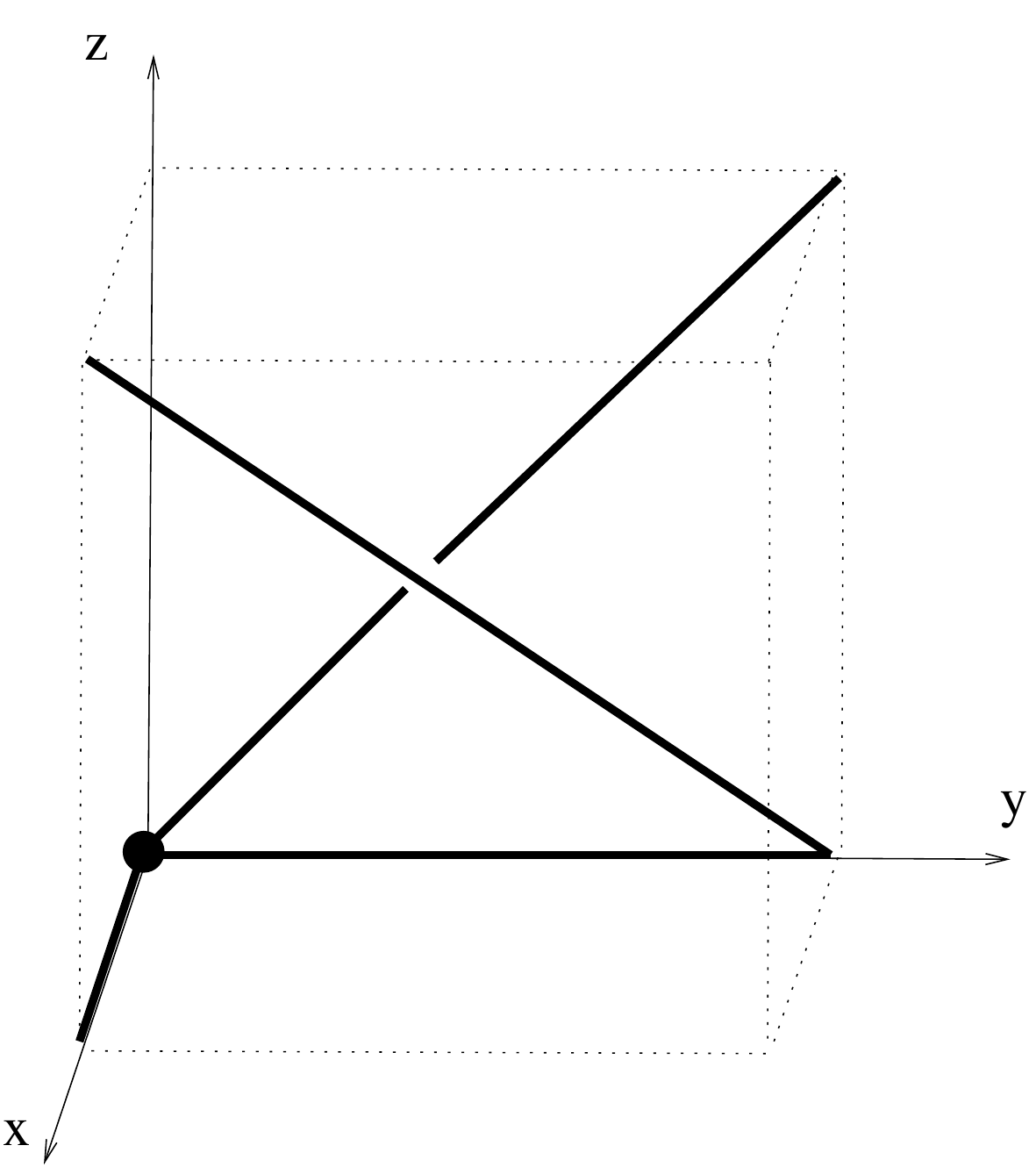}\quad
\caption{Some motifs of type $aafd$.}
\label{f:motifexamples2}
\end{figure}
\end{center}
The resulting 2 model nets, $\M_1$ and $\M_2$ are equivalent by a rotation about the line through the centre of the cube in the $x$-axis direction. Thus $\M_1$ and $\M_2$ are equivalent to the model net $\M_{aad}^g$
for the word $a_xa_yg_xd_1$. 

(iii) Assume that exactly $1$ of the 4 edges of $F_e$ is an axial edge, which we may assume lies in the $x$- axis. If the 3 remaining edges are the $f$-edges that are incident to the origin, then the transformation of $\M$ by the map
$(x,y,z)\to (x-z,y,z)$ has type $aaad$ and so is equivalent to $\M_{\rm pcu}^d$.
If the 3 remaining $f$ edges are not of this form then they are either coplanar (and, as before, $\M$ is equivalent to $\M_{\rm pcu}^f$) or only 1 of these 3 edges is incident to the origin, as in Figure \ref{f:motifexamples3}. In these cases $\M$ is equivalent to a model net with 2 axial edges and so the previous arguments suffice.

\begin{center}
\begin{figure}[ht]
\centering
\includegraphics[width=3.5cm]{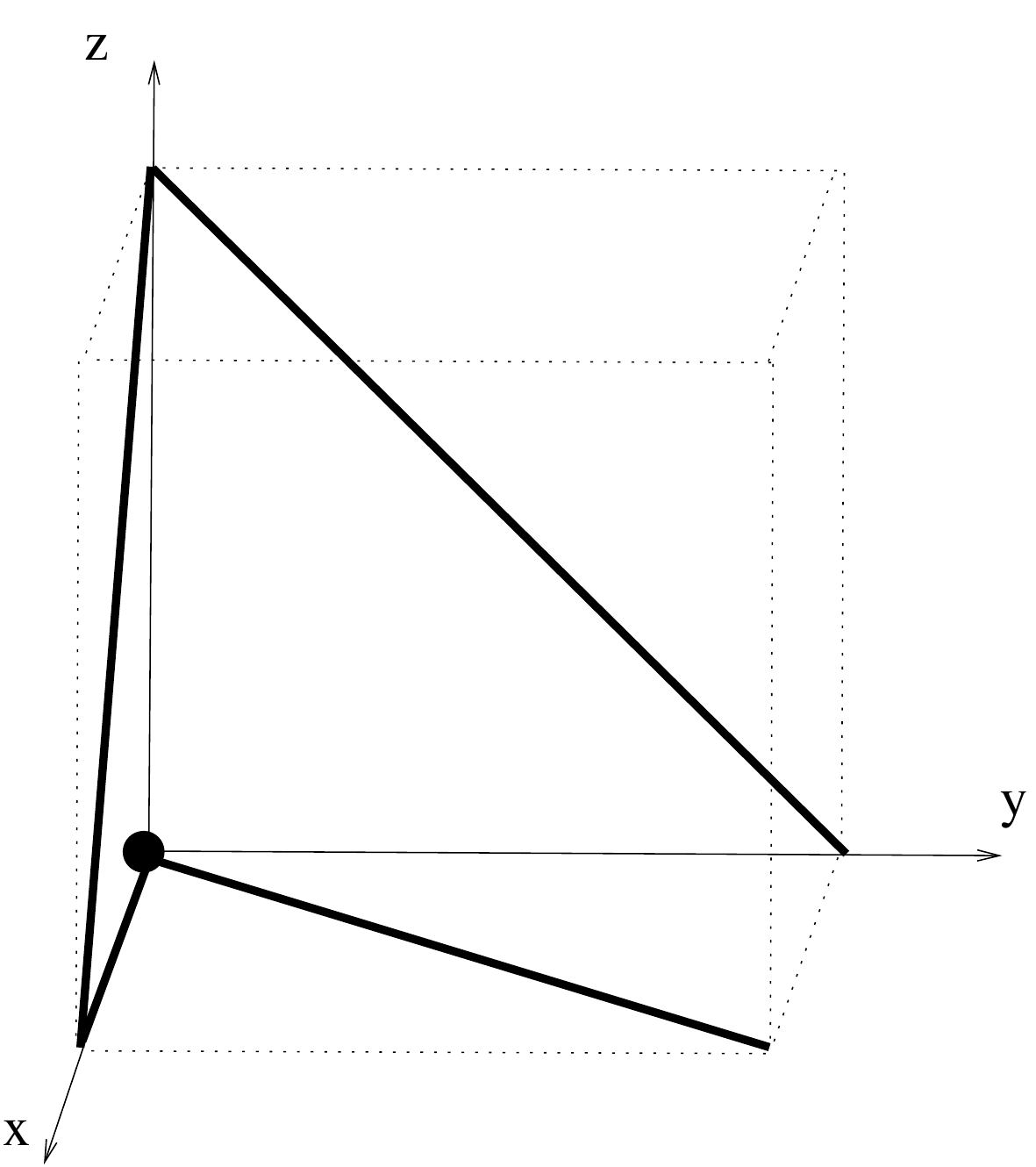}\caption{A motif with 3 non coplanar facial edges.}
\label{f:motifexamples3}
\end{figure}
\end{center}

Thus we may assume that the defining word for $\M$ is of type $affd, afgd$ or $aggd$. Moreover by rotational and translational equivalence we may assume that the possible types are $a_xffd_1, a_xfgd_1$ or $a_xggd_1$.
If all four edges are incident to the origin then $\M$ is equivalent by an elementary affine transformation to a model net with 2 axial edges and so there are no new cases to consider. Also if 3 edges are incident to the origin then once again the net is equivalent to the net for {\bf hex}, and so it remains to consider the cases $a_xg_xg_yd_1, a_xg_xg_zd_1$
and $a_xg_yg_zd_1$ indicated in Figure \ref{f:4penetrating}. 

Note that the first and third nets are the nets $\M_{\rm ad}^{gg}$ and $\M_{\rm ad}^{g_yg_z}$ in the list of model nets. That these nets are not isomorphic follows from their topological density counts.
The  second net
has a rotation about the diagonal which is a mirror image of the first net and so is equivalent to it by elementary transformations in view of Remark \ref{r:mirror2vertex}.

\begin{center}
\begin{figure}[ht]
\centering
\includegraphics[width=3.5cm]{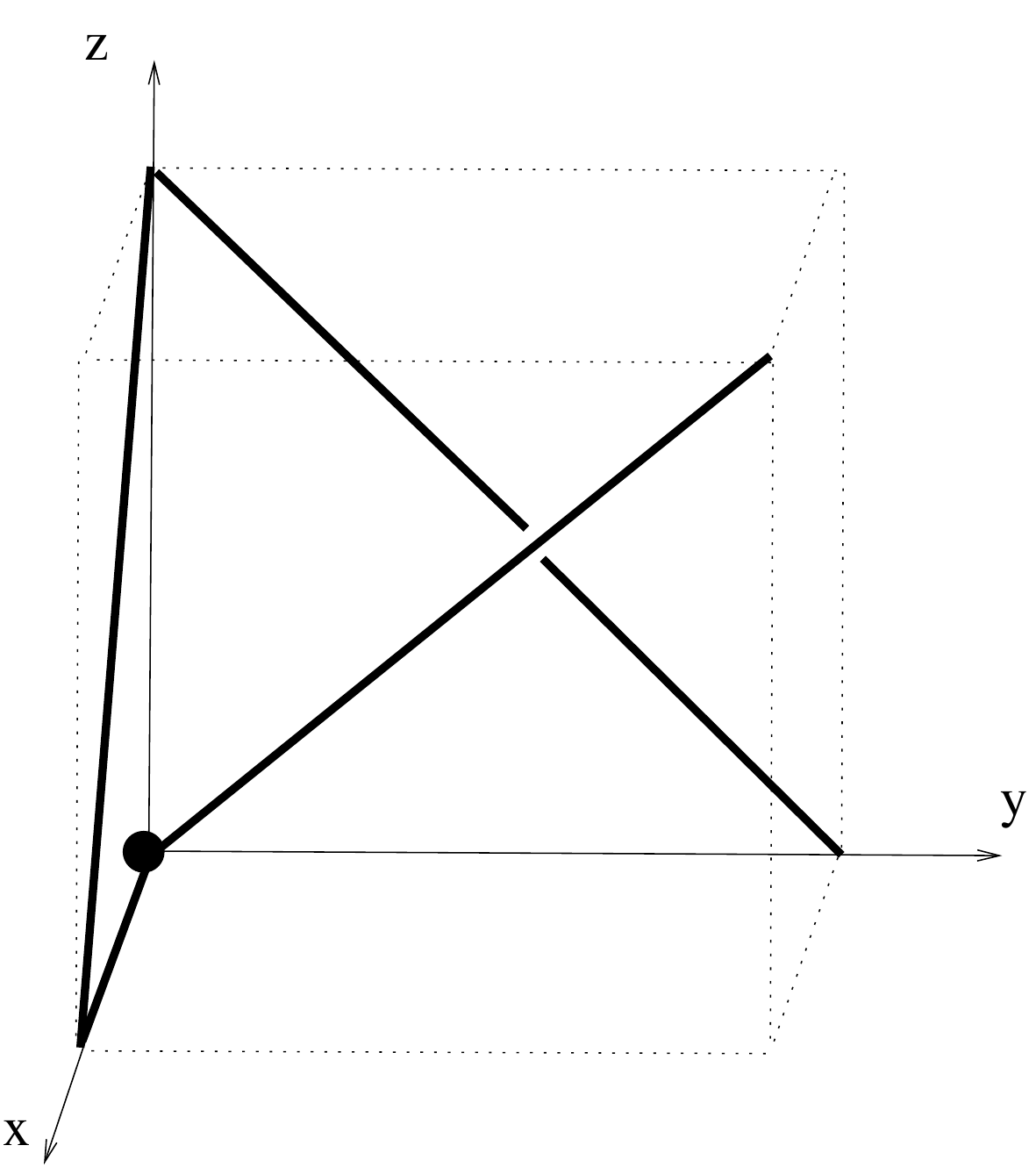}\quad
\includegraphics[width=3.5cm]{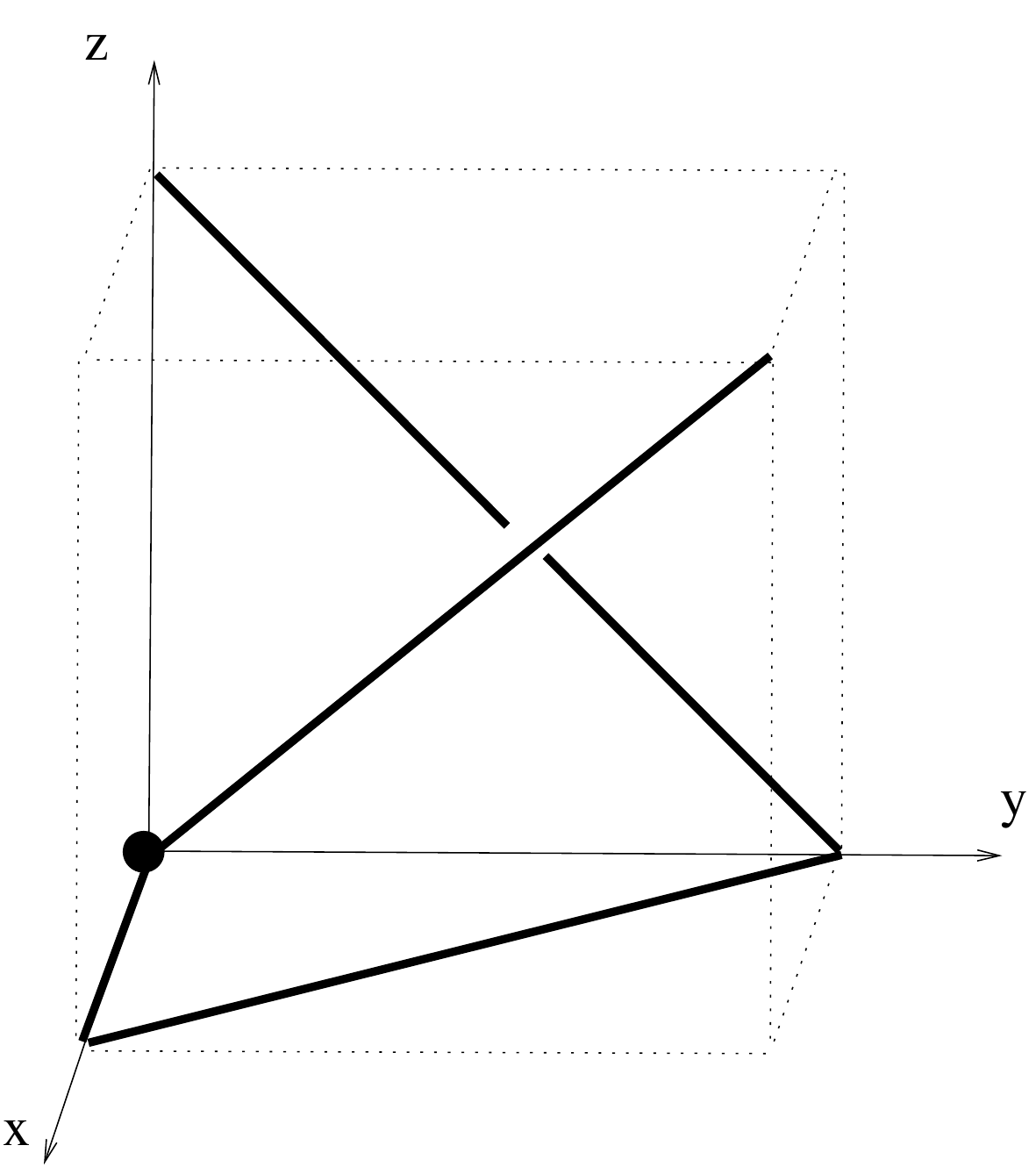}\quad
\includegraphics[width=3.5cm]{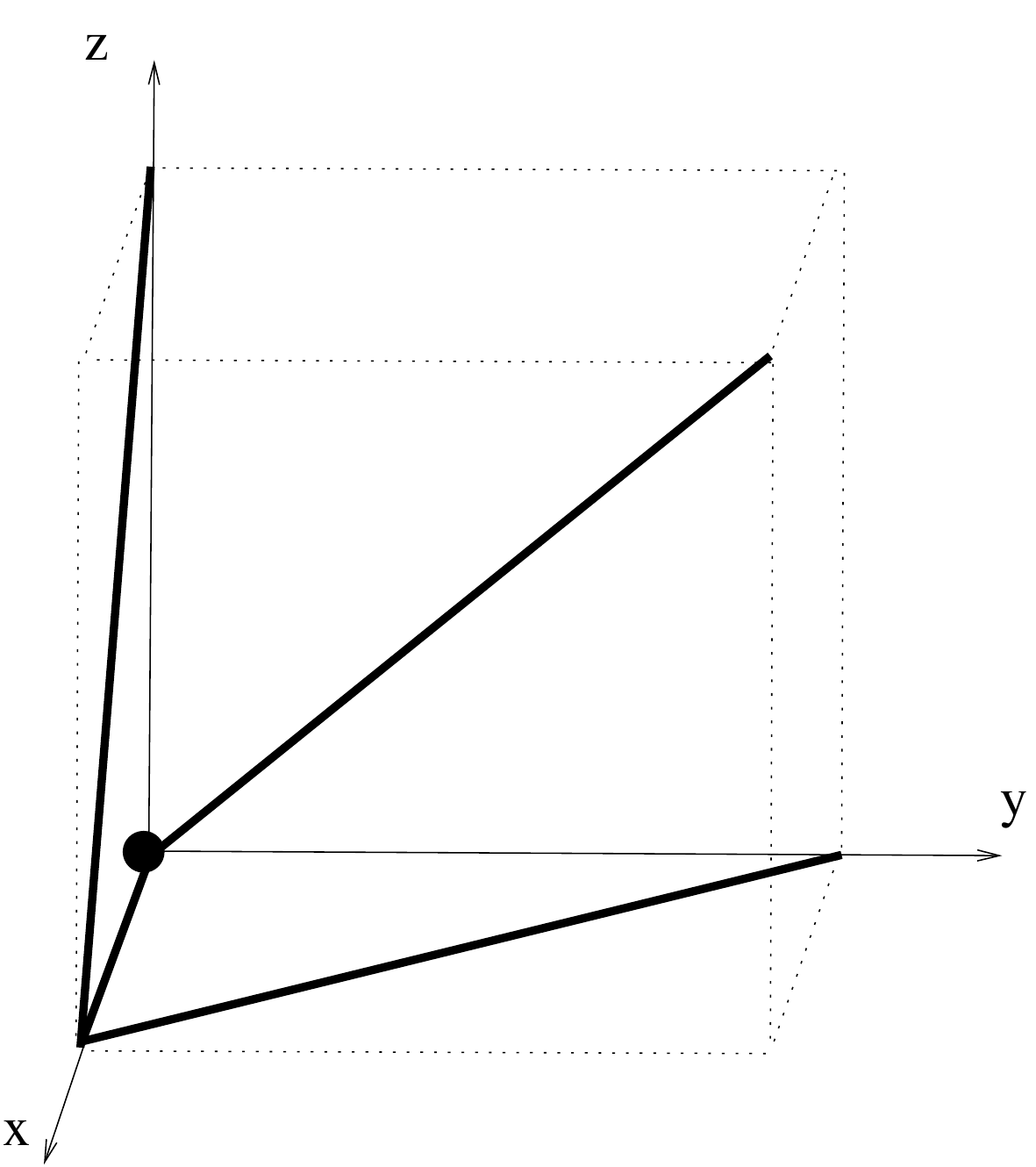}
\caption{Motifs for the model nets $\M(a_xg_xg_yd_1), \M(a_xg_xg_zd_1)$
and $\M(a_xg_yg_zd_1)$ .}
\label{f:4penetrating}
\end{figure}
\end{center}

(iv) Finally, for the case $m=4$, we assume that there are no axial edges. By rotational symmetry there are 4 cases which, under the convention are uniquely specified by the words $fffd, ffgd, fggd$ and $gggd$. The last of these corresponds to a disconnected net, as we have seen in the previous section, the first gives an alternative model net for {\bf ilc} (as we have remarked prior to the proof),
and the other 2 nets, for
$ffgd$ and $fggd$, are easily seen to be affinely equivalent to a model net with 1 axial edge.
\medskip

{\bf The case $m=5$.} It is straightforward to see that if $\M$ has 3 axial edges and 2 face edges then it is equivalent to the model net $\M_{\rm pcu}^{ff}$ for {\bf bct}. Also, type $aaafd$ is equivalent to this type.
On the other hand, type $aaagd$ has $\hxl$-multiplicity equal to $1$, rather than $2$, and so is in a new equivalence class, also with no edge penetrations. In fact this model net has topology {\bf ile}.

Consider next the model nets with 2 axial edges and no diagonal edges. These also have no penetrating edges and are of $\hxl$-multiplicity 1 or 2. Moreover it is straightforward to show that each is equivalent by elementary affine transformations to a model net with 3 axial edges and so they equivalent to the model nets for {\bf bct} and {\bf ile} respectively. 
The same is true for the 9 nets of type $aawd$ where $w$ is a word in 2 facial edges which is not of type $gg$.

Thus, in the case of 2 axial edges it remains to consider the types $a_xa_ywd_1$ with $w= g_xg_y, g_xg_z$ and $g_yg_z$ each of which has a penetrating edge of type $4^2$. The first two of these are model nets in the list and give new and distinct affine equivalence classes in view of their penetration type and differing $\hxl(\N)$ count. The third net, for the word $a_xa_yg_yg_zd_1$ is a mirror image of the first net and is orientedly affinely equivalent to it, by Remark \ref{r:mirror2vertex} for example. 

It remains to consider the case of 1 axial edge, $a_x$, together with $d_1$ and 3 facial edges. If there are 2 edges of type $f_x, f_y$ or $f_z$ then there is an elementary equivalence with a model net with 2 axial edges. The same applies if there is a single such edge. For an explicit example consider $a_xf_xg_yg_zd_1$. The image of this net under the transformation $(x,y,z)\to (x,y-z,z)$ gives a depth 1 net with 2 axial edges. The transformation of motifs is indicated in Figure \ref{f:afggd}.

\begin{center}
\begin{figure}[ht]
\centering
\includegraphics[width=3.5cm]{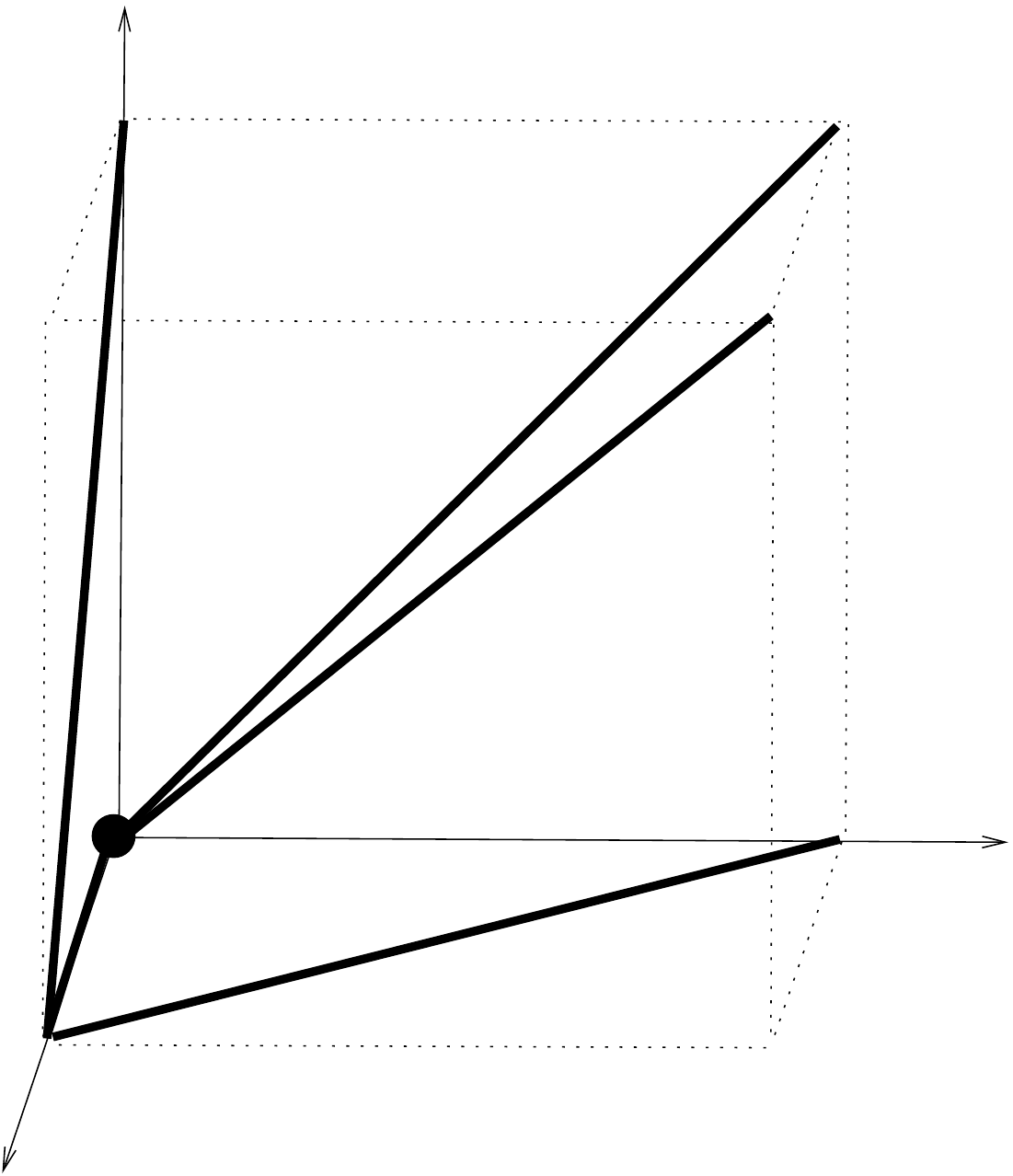}\quad \quad
\includegraphics[width=3.5cm]{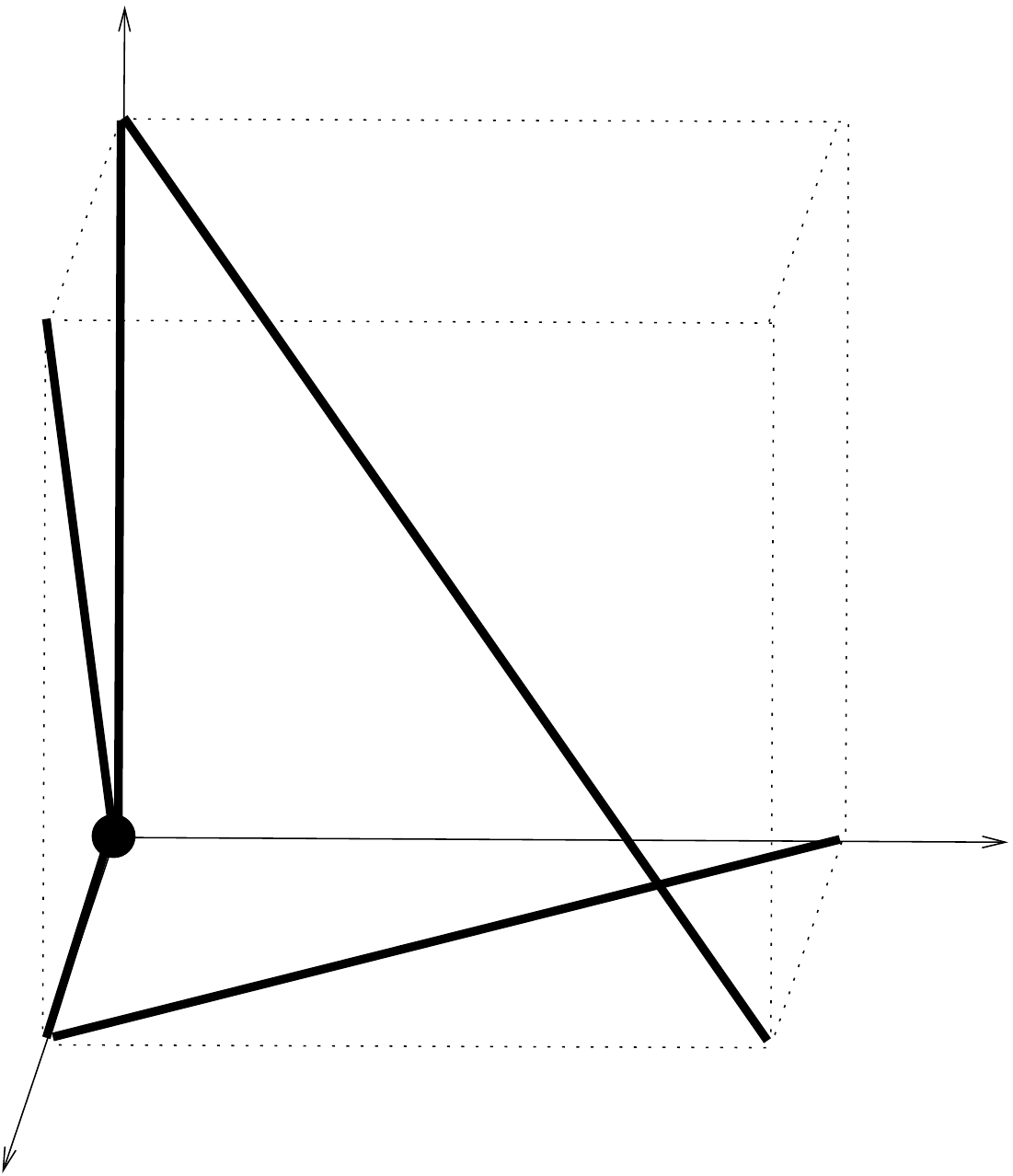}
\caption{Change of motif under $(x,y,z)\to (x,y-z,z)$.}
\label{f:afggd}
\end{figure}
\end{center}

Finally the model net for $a_xg_xg_yg_zd_1$ appears in the listing and gives a new affine class with penetration type $3^2$.

\medskip

{\bf The case $m=6$.} 
We first assume that there is no diagonal edge in the motif for $\M$ and therefore no edge penetration of type $4^2$ or $3^2$. There are 2 distinguished model nets in the list for this case, one with the 3 facial edges of type $f$ (a net with topology {\bf ild}) and one where the 3 facial edges are of type $g$ (a net with topology {\bf fcu}). Two other choices of facial edges are possible (up to rotation) and these are readily seen to be equivalent  to the {\bf ild} and {\bf fcu} nets.

We may now assume that there exists a diagonal edge in the standardised form of the edge word defining $\M$. If there are 3 axial edges then there are 3 possibilities, namely types $aaaffd, aaafgd, aaaggd.$ The first 2 cases are not new, since the transformation $(x,y,z)\to (x,y-z,z)$ give motifs without a diagonal edge, while the model net for $aaaggd$ appears in the list, with penetration type $4^2$ and $\hxl(\M)=2$.

We may now assume that $\M$ has a standardised word  $a_xa_ywd_1$ where $w$ is a word in 3 facial edges.
For $w$ of $fff$ type there are 3 cases, namely $f_xf_yg_z$, $f_xg_yf_z$ and $g_xf_yf_z$, each of which transforms by an elementary transformation (respectively, $x\to x-z, y \to y-z$ and $x\to x-z$) to a case with 3 axial edges.
For $w$ of type $fgg$ there are 3 cases, namely
$f_xg_yg_z$, $g_xf_yg_z$ and $g_xg_yf_z$. The first and second of these are not new, since the transformations $y\to y-z$ and $x\to x-z$, respectively, lead to an equivalence with $\M_{\rm pcu}^{ggd}$, while the third case is the model net $\M_{aad}^{ggf}$.

Finally, for $w$ of type $ggg$ we have the model net $\M_{aad}^{ggg}$.
\medskip

{\bf The case $m=7$.} 
There are 4 cases of standardised edge word of the form $aaawd$ with $w$ of type $fff, ffg, fgg$ or $ggg$. The model net for 
$a_xa_ya_zf_xf_yf_zd$ is obtained from the  model net for $a_xa_ya_zf_xf_yg_zd$ by the transformation $y\to y-z$ followed by a rotation. Thus there is a maximum of 3 equivalence classes with representative model nets $\M_{\rm pcu}^{fffd}, \M_{\rm pcu}^{ggfd}, \M_{\rm pcu}^{gggd}$. Since these are distinguished by their edge penetration type the proof is complete.
\end{proof}




\end{document}